\documentclass[reqno,11pt]{amsart}
\usepackage{latexsym,amssymb,amsfonts,amsmath,amsthm}
\usepackage{eucal}
\usepackage{mathtools}
\usepackage[dvipsnames,usenames]{xcolor}
\usepackage[colorlinks,citecolor=blue,urlcolor=blue]{hyperref}

\usepackage{tikz}
\usetikzlibrary{decorations.markings}
\usepackage{pgfplots}

\newtheorem{thm}{Theorem}[section]

\newtheorem{lem}[thm]{Lemma}
\newtheorem{prop}[thm]{Proposition}
\theoremstyle{definition}

\theoremstyle{remark}
\newtheorem{rem}[thm]{Remark}
\numberwithin{equation}{section}

\newcommand{\Real}{\mathbb R}
\newcommand{\eps}{\varepsilon}

\newcommand{\E}{\mathbb{E}}
\newcommand{\F}{\mathcal{F}}
\newcommand{\sign}{\mathrm{sign}}

\newcommand{\Res}{\mathrm{Res}}
\renewcommand{\Re}{\mathrm{Re}}
\renewcommand{\Im}{\mathrm{Im}}

\def\Xint#1{\mathchoice
   {\XXint\displaystyle\textstyle{#1}}%
   {\XXint\textstyle\scriptstyle{#1}}%
   {\XXint\scriptstyle\scriptscriptstyle{#1}}%
   {\XXint\scriptscriptstyle\scriptscriptstyle{#1}}%
   \!\int}
\def\XXint#1#2#3{{\setbox0=\hbox{$#1{#2#3}{\int}$}
     \vcenter{\hbox{$#2#3$}}\kern-.5\wd0}}

\def\dashint{\Xint-}

\begin{document}

\title[
Eigenproblems for fractional processes
]
{
Exact spectral asymptotics of fractional processes
}

\author{P. Chigansky}%
\address{Department of Statistics,
The Hebrew University,
Mount Scopus, Jerusalem 91905,
Israel}
\email{pchiga@mscc.huji.ac.il}

\author{M. Kleptsyna}%
\address{Laboratoire de Statistique et Processus,
Universite du Maine,
France}
\email{marina.kleptsyna@univ-lemans.fr}

\author{D. Marushkevych}%
\address{Laboratoire de Statistique et Processus,
Universite du Maine,
France}
\email{dmytro.marushkevych.etu@univ-lemans.fr}

\thanks{P. Chigansky is supported by ISF 558/13 grant}

\date{\today}%

\begin{abstract}

Eigenproblems frequently arise in theory and applications of stochastic processes, but only a few 
have explicit solutions. Those which do, are usually solved by reduction to the generalized 
Sturm--Liouville theory for differential operators. This includes the Brownian motion and a whole class 
of processes, which derive from it by means of linear transformations. 
The more general eigenproblem for the {\em fractional} Brownian motion (f.B.m.) is not solvable in closed form, 
but the exact asymptotics of its eigenvalues and eigenfunctions can be obtained, using a method based on analytic properties 
of the Laplace transform. In this paper we consider two processes closely related to the f.B.m.: the fractional 
Ornstein--Uhlenbeck process and the integrated fractional Brownian motion. While both derive from the f.B.m. 
by simple linear transformations, the corresponding eigenproblems turn out to be much more complex and their asymptotic 
structure exhibits new effects.  
 
\end{abstract}

\maketitle


\section{Introduction}

Covariance operator of a centered stochastic process $X=(X_t, t\in [0,1])$ with covariance 
kernel $K(s,t)=\E X_s X_t$  is the self-adjoint integral operator
$$
f\mapsto (K f)(t) = \int_0^1 K(s,t)f(s)ds.
$$
The associated eigenproblem consists of finding all pairs $(\lambda, \varphi)$ satisfying the equation 
\begin{equation}\label{eigpr}
K \varphi = \lambda \varphi.
\end{equation}
For square integrable kernels, this problem is well known to have countably many solutions $(\lambda_n, \varphi_n)$, 
$n\in \mathbb{N}$, where eigenvalues $\lambda_n$ are real and nonnegative and converge to zero, when put in the decreasing 
order, and the corresponding eigenfunctions $\varphi_n$ form a complete orthonormal basis in $L^2(0,1)$. 

Even though spectral decomposition is one of the earliest and most useful tools 
in the theory of stochastic processes, eigenproblems are rarely tractable and only a few are known to have reasonably explicit 
solutions. One general solution technique is reduction to linear differential equations. It is applicable whenever the covariance operator 
can be identified with the Green function of a differential operator. In this case the two share the same 
eigenstructure, which is typically more accessible for the latter through the generalized Sturm--Liouville theory
(see \cite{NN04ptrf}, \cite{N09, N09b}).

The simplest example is the Brownian motion with $K(s,t)=s\wedge t$, for which \eqref{eigpr} 
readily reduces to the boundary value problem for the simple o.d.e 
\begin{equation}\label{fBmode}
\lambda \varphi''(t) + \varphi(t)=0
\end{equation}
subject to $\varphi(0)=0$ and $\varphi'(1)=0$. The explicit solution yields familiar formulas: 
\begin{equation}
\label{Bmeig}
\lambda_n = \frac 1{\nu_n^2 }\quad \text{and} \quad \varphi_n(t) = \sqrt{2} \sin \nu_n t, \quad 
\end{equation}
where $\nu_n =  \pi n-\pi/2$. Essentially the same method works for a whole class of processes, which 
derive from the Brownian motion through linear transformations, including the integrated Brownian motion, 
the Brownian bridge, the Ornstein--Uhlenbeck process and others.

The standard Brownian motion is a special case of the {\em fractional} Brownian motion (f.B.m.), 
which is the centered Gaussian process $B^H =(B^H_t, t\in [0,1])$ 
with  covariance function 
$$
K(s,t) = \frac 1 2\left(s^{2H}+t^{2H}-|s-t|^{2H}\right),
$$
where $H\in (0,1)$ is its Hurst exponent. This is the only $H$-self similar Gaussian process with stationary increments. 
Unlike the standard Brownian motion, corresponding to $H=\frac 1 2$, the f.B.m. is neither a semimartingale 
nor a Markov process for any other value of $H$.  For $H>\frac 1 2$ the increments $B^H_n-B^H_{n-1}$, $n\in \mathbb{N}$ are positively correlated and 
their covariances are not summable. This long range dependence makes the f.B.m. useful in a variety of applications, see e.g. \cite{PT17}.

Despite the fact that the f.B.m. has been extensively studied since its introduction in \cite{MvN68}, 
its eigenstructure remained elusive for quite a while, see \cite{M82}. The exact first order asymptotics 
of the eigenvalues was obtained only a decade and a half ago in \cite{Br03a, Br03b} (see also \cite{LP04}, \cite{NN04tpa}) 
and this result remained state of the art until recently, when a more detailed asymptotic picture was revealed in \cite{ChK}:

\begin{thm}\label{main-thm-fbm}\

\medskip

\noindent
{\bf 1.} For $H\in (0,1)$ the eigenvalues are given by the formula 
\begin{equation}
\label{lambda_n_fBm}
\lambda_n =  \sin (\pi H)\Gamma(2H+1) \nu_n^{-2H-1} \qquad n=1,2,...
\end{equation}
where the sequence $\nu_n$ satisfies 
\begin{equation}\label{fBmnu}
\nu_n =  \Big(n -\frac 1 2\Big)\pi + \frac {1-2H}{ 4}  \pi  + \arcsin \frac{\ell_H}{\sqrt{1+\ell_H^2}}+O(n^{-1}) \quad \text{as}\ n\to\infty,
\end{equation}
with the constant 
$
\displaystyle
\ell_H:=\frac
{
\sin \frac{\pi}{2} \frac{ H-1/2}{H+1/2}
}
{\sin \frac \pi 2 \frac{1}{H+1/2}}.
$

\medskip
\noindent
{\bf 2.} The corresponding unit norm eigenfunctions admit the approximation 
\begin{multline}\label{phinfBm}
\varphi_n(x)  
= 
 \sqrt 2 \sin\bigg( \nu_{n} x+\frac {2H-1}{ 8}  \pi-\arcsin \frac{\ell_H}{\sqrt{1+\ell_H^2}}\bigg) \\
 - \frac {\sqrt{2H+1}} { \pi }  \int_0^{\infty}    \rho_0(u)
\bigg(
e^{-  x \nu_n u} \frac{ u-\ell_H}{ \sqrt{1+\ell_H^2}}+
(-1)^{n}   e^{-  (1-x) \nu_nu}\
\bigg)du + n^{-1}  r_n(x),
\end{multline}
where the residual $r_n(x)$ is bounded by a constant, depending only on $H$, and
$\rho_0(u)$ is an explicit function (see \eqref{rhofOU} below). 

\medskip
\noindent
{\bf 3.}  The eigenfunctions satisfy
\begin{equation}\label{eigffun}
\varphi_n(1) = (-1)^{n} \sqrt{2H+1}   \big(1+O(n^{-1})\big)
\quad\text{and}\quad
\int_0^1 \varphi_n(x)dx =  -\sqrt{\frac{2H+1}{1+\ell_H^2}}\; \nu_n^{-1}.
\end{equation}

\end{thm}

\medskip 
Formulas \eqref{lambda_n_fBm}-\eqref{fBmnu} furnish asymptotic approximation of the eigenvalues, accurate up to the second 
order with an estimate for residual.  Asymptotics \eqref{phinfBm} reveals that the eigenfunctions comprise of the oscillatory term, 
similar to that in \eqref{Bmeig}, and the boundary layer, whose contribution persists only near the endpoints of the interval. 
The boundary layer vanishes in the standard Brownian case $H=\frac 1 2$. 
Expressions \eqref{eigffun} give the exact asymptotics for two particular linear functionals of the eigenfunctions,
which turn out to be useful in applications discussed in \cite{ChK}. It is possible to derive explicit formulas for other 
functionals of interest, either directly, using approximation \eqref{phinfBm}, or through intermediate steps of the proof.

Theorem \ref{main-thm-fbm} is proved in \cite{ChK} by a technique, based on analytic properties of the Laplace transform. 
It reduces the eigenproblem to a certain integro-algebraic system of equations, which turns out to be more amenable to 
asymptotic analysis. It would be natural to expect that, just as in the standard Brownian case, the same method applies to 
processes, which derive from the f.B.m by means of linear transformations. While ultimately this is indeed the case, such extension 
is far from being straightforward and its implementation faces a whole new level of complexity. 

In this paper we consider two processes: the integrated fractional Brownian motion and 
the Ornstein--Uhlenbeck type process, generated by linear stochastic equation driven by fractional noise.  
While these two processes are related to the f.B.m. by simple linear functionals, their equivalent 
integro-algebraic systems have much more complicated form and are derived in entirely different ways. 
The corresponding asymptotic spectral structure also turns out to be significantly more intricate and exhibit 
new effects: for example, eigenfunctions of the integrated f.B.m. have {\em multiscale} boundary layer, 
whose components vanish at both polynomial and exponential rates. It is quite surprising that such basic 
operation as integration makes the second order term in \eqref{nuifBm} as subtle as \eqref{DeltaH}-\eqref{bbb},
compared with that in \eqref{fBmnu} for the f.B.m. itself.

%

\section{Main results} 

\subsection{Integrated fractional Brownian motion}
The first integral of the Brownian motion  
$ 
X_t = \int_0^t B_s ds
$ 
is the centered process with covariance function 
$$
K(s,t) = \int_0^s\int_0^t (u\wedge v) du dv.
$$
For covariance operator with this kernel, eigenproblem \eqref{eigpr} also reduces to a simple differential equation,
whose explicit solution yields  
$$
\lambda_n =  \frac 1{\nu_n^4}, \quad n=1,2,...
$$
with $\nu_n=\pi n - \pi/2 + O(e^{-n})$, being the increasing sequence of positive roots of equation 
$$
\cos \nu \cosh \nu + 1 =0.
$$
The corresponding normalized eigenfunctions satisfy 
$$
\varphi_n(t) \propto  \frac{\cos \nu_n  +\cosh \nu_n}{ \sin \nu_n  +\sinh \nu_n  }\big(\sinh \nu_n t-\sin \nu_n t\big) -
\big(\cosh \nu_n t-\cos \nu_n t\big),
$$
where proportionality symbol $\propto$ stands for equality up to a multiplicative factor asymptotic to 1, which normalizes 
$\varphi_n$ to unit $L^2(0,1)$ norm.

The eigenstructure of integrated processes was comprehensively studied in \cite{GHT03}, \cite{NN04ptrf}. However  
none of the approaches suggested so far applies to the integrated f.B.m. $X_t := \int_0^t B^H_s ds$. 
Covariance function of this process  
\begin{equation}
\label{KifBm}
K(s,t) = \int_0^t\int_0^s \tfrac 1 2\left(u^{2H}+v^{2H}-|v-u|^{2H}\right) dudv,
\end{equation}
satisfies the scaling property 
$$
K(sT,tT) = T^{2H+2} K(s,t),\quad s,t\in [0,1] \quad T>0,
$$
and hence no generality is lost if  eigenproblem \eqref{eigpr} is considered on the unit interval.  
Our first result details the corresponding asymptotic spectral structure:

\begin{thm}\label{thm-ifBm}
\

\medskip 
\noindent 
{\bf 1}. 
The eigenvalues of covariance operator with kernel \eqref{KifBm} satisfy
$$
\lambda_n = \sin (\pi H) \Gamma(2H+1)\nu_n^{-2H-3}\quad n=1,2,...
$$
where    
\begin{equation}\label{nuifBm}
\nu_n = \pi \Big(n-\frac 1 2\Big) + \frac{1-2H}{4}\pi + \arctan \Delta(H) + O(n^{-1}),
\end{equation}
and constant $\Delta(H)$ is given by the expression
\begin{equation}\label{DeltaH}
\Delta(H) = \frac{\frac 1 3 b_0^3-b_2}{b_1-b_1^2 +\frac 1 2 b_0^2+b_2b_0-\frac 1 {12} b_0^4}
\end{equation}
with  
\begin{equation}\label{bbb}
b_0 = \frac{
\sin\Big(\frac \pi 2 \frac{H+\frac 1 2}{H+\frac 3 2}\Big) 
}
{
\sin \Big(\frac \pi 2 \frac 1 {H+\frac 3 2}\Big)
},
\qquad  b_1 = \frac 1 2,\qquad  
b_2 = \frac 1 3 \frac
{
\sin\Big(\frac {3\pi} 2 \frac{H+\frac 1 2}{H+\frac 3 2}\Big)
}
{
\sin \Big(\frac {3\pi} 2 \frac 1 {H+\frac 3 2}\Big)
}.
\end{equation}

\medskip 
\noindent
{\bf 2}. The corresponding unit norm eigenfunctions admit the approximation
$$
\varphi_n(x) = \varphi^{(1)}_n(x) + \varphi^{(2)}_n(x) + \varphi^{(3)}_n(x) + n^{-1} r_n(x)
$$
where residual $r_n(x)$ is bounded uniformly in both $n\in \mathbb{N}$ and $x\in [0,1]$ and 

\medskip

\begin{enumerate}
\addtolength{\itemsep}{0.7\baselineskip}
\renewcommand{\theenumi}{\alph{enumi}}
\item the oscillatory term is given by
$$
\varphi^{(1)}_n(x) = \sqrt{2} \cos \Big(\nu_n x + \frac {2H+1}8\pi - \arctan \Delta(H)\Big);
$$

\item the polynomial boundary layer term is given by 
$$
\varphi^{(2)}_n(x) = -
\frac{\sqrt{2H+3}}{\pi }
\int_{0}^\infty  \rho_0(t)
 \left(  Q_0(t)  e^{-t\nu_n x}  - (-1)^n Q_1(t) e^{-t\nu_n(1-x)}\right) dt,
$$
where function $\rho_0(t)$ and polynomials $Q_0(t)$ and $Q_1(t)$ are given by explicit expressions 
\eqref{rho0ifBm2}-\eqref{QifBm2} for $H<\frac 1 2$ and  \eqref{rhorho}-\eqref{Q0Q1} for $H>\frac 1 2$;

\item the exponential boundary layer term vanishes for $H\le \frac 1 2$ and otherwise is given by 
$$
\varphi^{(3)}_n(x) = C_0 e^{- c \nu_n x}   \cos\Big(  s \nu_n x  + \varkappa_0 \Big)
+
C_1 e^{- c \nu_n (1-x)}   \cos\Big(  s \nu_n (1-x)  + \varkappa_1 \Big) 
$$
where amplitudes $C_0$ and $C_1$ and phases $\varkappa_0$ and $\varkappa_1$ are explicit constants and 
$$
c := \cos \frac{\pi}{2}\frac{ H-\frac 1 2}{H+\frac  3 2}>0 \quad\text{and} \quad 
s :=\sin \frac{\pi}{2}\frac{ H-\frac 1 2}{H+\frac  3 2}>0.
$$ 

\end{enumerate}

\medskip
\noindent
{\bf 3.} The eigenfunctions satisfy 
$$
\varphi_n(1)  = 
 (-1)^n   \sqrt{2H+3} \big( 1+O(n^{-1})\big)
$$
and
$$
\int_0^1\varphi_n(x)dx   =  
\nu_n^{-1}   \sqrt{2H+3}\, C_H \big( 1+O(n^{-1})\big)
$$
with explicit constant $C_H$, given by \eqref{funphin} for $H<\frac 1 2$ and 
\eqref{tildeC} for $H>\frac 1 2$.

\end{thm}

\medskip 

As for the f.B.m  the eigenfunctions in this case also comprise
of the oscillatory and boundary layer terms. The effect of the boundary layer is asymptotically negligible in the 
interior of the interval, and it pushes the eigenfunctions to zero at $x=0$ and to $\pm \sqrt{2H+3}$ at $x=1$. 
However, unlike before, it exhibits multiscale behavior: for $H>\frac 1 2$, the two components $\varphi^{(2)}_n(x)$ 
and $\varphi^{(3)}_n(x)$ vanish at different, polynomial and exponential, rates respectively as $n\to\infty$.

\subsection{Fractional Ornstein--Uhlenbeck process}
In stochastic analysis the Ornstein--Uhlenbeck process $X=(X_t,\, t\in [0,1])$ can be derived from the standard 
Brownian motion $B=(B_t,\, t\in [0,1])$ as the solution of the Langevin equation 
\begin{equation}\label{Leq}
X_t = X_0 + \beta \int_0^t X_s ds + B_t, 
\end{equation}
where $\beta\in \Real$ is the drift parameter and $X_0\sim N(0,\sigma^2)$ is the initial condition, 
independent of $B$. A nontrivial initial condition introduces a rank one perturbation to the covariance operator, 
which is inessential for our purposes (see \cite{N09b} and \cite{ChKM2}), and hereafter we will set $X_0=0$ for simplicity.  
In this case, the covariance function of $X$ is given by the formula 
$$
K(s,t) = \int_0^{t\wedge s} e^{\beta(t-\tau)}e^{\beta(s-\tau)}d\tau.
$$

Eigenproblem \eqref{eigpr} with this kernel reduces to a boundary 
value problem similar to \eqref{fBmode}, whose explicit solution yields  
\begin{equation}
\label{OUeig}
\lambda_n = \frac 1{\nu_n^2+\beta^2}\quad \text{and}\quad \varphi_n(t) \propto  \sqrt{2} \sin \nu_n t,
\end{equation}
where $\nu_n = \pi n -   \pi/ 2 + O(n^{-1})$ is the increasing sequence of positive roots of equation 
$$
\nu/\beta = \tan \nu.
$$

\medskip

In the fractional setting, the Ornstein--Uhlenbeck process can be defined in a number of nonequivalent ways \cite{CKM03}, 
and here we consider solution of the Langevin equation \eqref{Leq}, driven by the f.B.m. 
The covariance function is given in this case by the formula (see, e.g., \cite{PT17}):
\begin{equation}\label{OUKst}
K_\beta(s,t) = 
\int_0^t    e^{\beta(t-v)} \frac d{dv} \int_0^s H |v-u|^{2H-1} \sign(v-u)  e^{\beta(s-u)}du dv.
\end{equation}
Note that this kernel satisfies the scaling property
\begin{equation}\label{Kbeta}
K_\beta(sT,tT) = T^{2H} K_{\beta T}(s,t), \quad s,t\in [0,1]\quad T>0
\end{equation}
and therefore asymptotic approximation of solutions to the eigenproblem on an arbitrary interval $[0,T]$ 
can be obtained from that on the unit interval.   
For $H>\frac 1 2$ integration and derivative in \eqref{OUKst} are interchangeable and it simplifies to
\begin{equation}
\label{KstHlarge}
K_\beta(s,t) =  
\int_0^t  \int_0^s   e^{\beta(t-v)} e^{\beta(s-u)} H(2H-1)  |v-u|^{2H-2} du dv.
\end{equation}
The fractional Ornstein--Uhlenbeck process of this type inherits long-range dependence property from the f.B.m., see \cite{CKM03}.

\medskip 

Our next result generalizes \eqref{OUeig} beyond the standard Brownian case $H=\frac 1 2$: 

\medskip

\begin{thm}\label{thm-fOU}
For any $H\in (0,1)$ the eigenvalues of covariance operator \eqref{OUKst} satisfy
\begin{equation}\label{lambda_n_fOU}
\lambda_n = \sin (\pi H) \Gamma(2H+1)\frac{\nu_n^{1-2H}}{\nu_n^2+\beta^2}, \quad n=1,2,...
\end{equation}
where $\nu_n$ is the sequence with asymptotics \eqref{fBmnu}. The unit norm eigenfunctions $\varphi_n$ admit approximation 
\eqref{phinfBm} and satisfy \eqref{eigffun}. 
\end{thm}

\medskip 
This result exhibits a curious dependence separation of the spectral asymptotics on its parameters. The drift parameter 
$\beta$ enters eigenvalues formula \eqref{lambda_n_fOU} only in the dominator, while all of its other ingredients depend solely 
on the Hurst parameter $H$. This includes $\nu_n$, which turns out to be the same as for the f.B.m., at least up to the 
second asymptotic term. Moreover, the eigenfunctions do not depend on $\beta$ in the first approximation.

\section{A sample application}
Spectral decomposition of stochastic processes has numerous applications, see a partial list in \cite{ChK}, and accurate approximations  
often prove useful.
As an example, in this section we revisit the classical problem of signal estimation in 
white noise. Consider the process 
\begin{equation}\label{Yt}
Y_t = \mu \int_0^t X_s ds + \sqrt{\eps} B_t, \quad t\in [0,T]
\end{equation} 
where $\mu$ and $\eps>0$ are real constants, $B_t$ is the Brownian motion and $X_t$ is an independent {\em signal} 
process, whose trajectory is to be estimated given the observed trajectory of  
$Y=(Y_t, t\in [0,T])$.
Here $\mu$ is interpreted as the {\em channel gain} and the formal derivative of $B_t$ is viewed as the white 
noise disturbance, whose intensity is controlled by $\eps$. 

The optimal in the mean squared sense estimator of $X_t$ given the observation path is the conditional expectation 
$\widehat{X}_t = \E (X_t|\F^Y_T)$, where $\F^Y_T=\sigma\{Y_t, t\in [0,T]\}$. If $X$ is a centred Gaussian process with 
covariance function $K(s,t)=\E X_sX_t$, this estimator is given by suitably defined stochastic integral 
$$
\widehat X_t = \frac 1 \mu \int_0^T h(s,t) dY_s,
$$ 
where kernel $h(s,t)$ solves the integral equation 
\begin{equation}\label{WHeq}
\eps h(s,t) + \int_0^T \mu^2 K(r,s) h(r,t)dr = \mu^2 K(s,t), \quad 0\le s\le t \le T.
\end{equation}
The corresponding minimal mean square error $P_\eps(t) = \E (X_t-\widehat X_t)^2$ satisfies    
$$
P_\eps(t) = K(t,t) - \int_0^T h(r,t) K(r,t)dr = \frac \eps {\mu^2} h(t,t),
$$ 
and an important engineering question is how it scales with the noise intensity.

In the Gauss-Markov case, when $X$ is the Ornstein--Uhlenbeck process driven by the standard Brownian motion,
integral equation \eqref{WHeq} can be solved explicitly by reduction to the Riccati o.d.e. 
(see, e.g., Theorem 12.10 in \cite{LS2}) and, $P_\eps(t)$ can be computed in a closed form. A calculation gives the 
following {\em high signal-to-noise} asymptotics  
$$
 P_\eps(t) \simeq \sqrt{\eps/\mu^2}  \begin{cases}
\frac 1{2 }  & t\in (0,T)\\
 1 & t=T
\end{cases}\qquad \text{as\ } \eps\to 0,
$$
where $f(\eps)\simeq g(\eps)$ stands for $f(\eps) = g(\eps)\big(1+o(1)\big)$ as $\eps\to 0$. 
Note that the {\em smoothing} estimator $\widehat X_t$ with $t<T$ outperforms the {\em filtering} estimator 
$\widehat {X}_T$ by factor $2$ in the limit. 

Generalizing this result to fractional setting, when $X$ is the fractional Ornstein - Uhlenbeck process \eqref{Leq}, 
cannot be easily approached along the same lines due to lack of the Markov property. The more tractable alternative 
is to use spectral approximations from Theorem \ref{thm-fOU}, which give  the following result:

\begin{prop} 
The minimal mean squared error in the estimation problem of the fractional 
Ornstein--Uhlenbeck process \eqref{Leq} given observations \eqref{Yt} satisfies  
\begin{equation}\label{Peps}
P_\eps(t) \simeq (\eps/\mu^2)^{\frac{2H}{1+2H}}  
 \frac{\Big(\sin (\pi H) \Gamma(2H+1)\Big)^{\frac 1 {1+2H}}}{\sin \frac \pi {2H+1}}
 \begin{cases}
 \frac{1}{2H+1} & t\in (0,T)\\
 1 & t=T
 \end{cases}\qquad \text{as\ } \eps\to 0.
\end{equation}
\end{prop}

\begin{rem}
The scaling rate $\eps^{\frac{2H}{1+2H}}$ of the estimation error here coincides with the optimal minimax 
rate in the nonparametric estimation problem of deterministic $H$-H\"older signals in the white noise model, 
see \cite{Tsybakov}. 
\end{rem}

\begin{proof}

By scaling property \eqref{Kbeta}, equation \eqref{WHeq} reads
\begin{equation}\label{WHeq1}
\eps h_\eps(u,v) + \int_0^1 \mu^2 T^{2H+1} K_{\beta T}(r,u)   h_\eps(r,v)dr  = \mu^2 T^{2H} K_{\beta T}(u,v), \quad 0\le u\le v \le 1
\end{equation}
where $h_\eps(u,v):= h(uT,vT)$, and 
$$
P_\eps(u) =  \frac \eps {\mu^2} h_\eps(u,u), \quad u\in (0,1).
$$

Expanding the solution of \eqref{WHeq1}  into series of eigenfunctions of $K_{\beta T}$ we get 
$$
h_\eps(u,v) = \sum_{n=1}^\infty \frac{\mu^2T^{2H}}{\eps \lambda_n^{-1} + \mu^2T^{2H+1} }
\varphi_n(u) \varphi_n(v), \quad 0\le u\le v\le 1,
$$
where $\lambda_n$ are the eigenvalues of $K_{\beta T}$. This series is absolutely convergent for any $\eps>0$
and its value diverges to $+\infty$ as $\eps\to 0$. Its first order term asymptotics does not change if the eigenvalues and 
eigenfunctions are replaced with their first order approximations from Theorem \ref{thm-fOU}. 
Let $C:= \sin (\pi H) \Gamma(2H+1)$, then  
\begin{align*}
P_\eps(1) =
& 
\frac \eps {\mu^2} h_\eps(1,1)   =   
\sum_{n=1}^\infty \frac{\eps T^{2H}}{\eps \lambda_n^{-1} + \mu^2T^{2H+1} }\varphi_n^2(1) \simeq \\
&
\sum_{n=1}^\infty \frac{\eps T^{2H}}{\dfrac \eps  C 
 \dfrac{(\pi n)^2+(\beta T)^2}{(\pi n)^{1-2H}}
+ \mu^2T^{2H+1} } (2H+1) \simeq \\
&
(2H+1)\int_1^\infty \frac{\eps T^{2H}}{\dfrac \eps  C 
\big( (\pi x)^{2H+1}+(\beta T)^2 (\pi x)^{2H-1} \big)
+ \mu^2T^{2H+1} } dx \simeq \\
&
 \frac {\eps} {\mu^2T}  
\left(\frac \eps C \frac 1{T^{2H+1}\mu^2}\right)^{-\frac 1{2H+1}}
\frac{2H+1} \pi\int_0^\infty \frac{1 }{y^{2H+1}  + 1   } d y = \\
&
 \frac {\eps} {\mu^2T}  
\left(\frac \eps C \frac 1{T^{2H+1}\mu^2}\right)^{-\frac 1{2H+1}}
\frac 1 {\sin \frac \pi {2H+1}} =
 (\eps/\mu^2)^{\frac {2H}{2H+1}}    
\frac {C^{\frac 1{2H+1}}} {\sin \frac \pi {2H+1}}
\end{align*}
which gives the expression for the filtering error at $t=T$ in \eqref{Peps}. 

Let us now calculate the limit value of $P_\eps(u)$ for $u\in (0,1)$. To this end, note that the contribution of the 
boundary layer term in \eqref{phinfBm} is of order $O(n^{-1})$ and hence 
$$
\varphi_n(u) =  \sqrt{2}\sin  \big(\nu_n u+\phi_H\big) + O(n^{-1}), \quad \text{as\ } n\to\infty,
$$ 
where $\phi_H :=\frac{2H-1} 8 \pi - \arcsin \frac{\ell_H}{\sqrt{1+\ell_H^2}}$. Hence 
\begin{align*}
P_\eps(u) & =  \frac \eps {\mu^2} h_\eps(u,u) =
\sum_{n=1}^\infty \frac{ \eps T^{2H}}{\eps \lambda_n^{-1} + \mu^2T^{2H+1} }
\varphi_n^2(u)= \\
&
\sum_{n=1}^\infty \frac{\eps T^{2H}}{\eps \lambda_n^{-1} + \mu^2T^{2H+1} }
- 
\sum_{n=1}^\infty \frac{ \eps T^{2H}}{\eps \lambda_n^{-1} + \mu^2T^{2H+1} }
\cos    \big(2\nu_n u+2\phi_H\big) := I_1(\eps)+I_2(\eps).
\end{align*}
The first term here differs from the previous case only by constant factor $2H+1$ and hence  
the expression for smoothing error at $t\in (0,T)$ in \eqref{Peps} is obtained, once we show 
that the second term vanishes as $\eps\to 0$. Define $S_0=0$ and 
$$
S_n = \sum_{k=1}^n \cos    \big(2\nu_k u+2\phi_H\big), \quad n\ge 1,
$$
which is a bounded sequence for $u\in (0,1)$. Then  
\begin{align*}
I_2(\eps)
= 
&
\sum_{n=1}^\infty \frac{ \eps T^{2H}}{\eps \lambda_n^{-1} + \mu^2T^{2H+1} }
\big(S_n-S_{n-1}\big) = \\
&
\eps^2 T^{2H}\sum_{n=1}^\infty S_n 
\frac
{
    \lambda_{n+1}^{-1} -\lambda_n^{-1} 
}
{
\big(\eps \lambda_{n+1}^{-1} + \mu^2T^{2H+1}\big)
\big(\eps \lambda_n^{-1} + \mu^2T^{2H+1}\big)
}.
\end{align*}
In view of \eqref{lambda_n_fOU}, for all $n$ large enough
$$
|\lambda_{n+1}^{-1} -\lambda_n^{-1}| \le C_1 n^{2H}  \quad \text{and}\quad \lambda_n^{-1} \ge C_2 n^{2H+1}
$$
with positive constants $C_1$ and $C_2$. Since $S_n$ is bounded, for some constant $C_3$,
\begin{align*}
|I_2(\eps)| \le \, 
&
C_3
\eps^2  \sum_{n=1}^\infty  
\frac
{
    n^{2H}
}
{
\big(\eps n^{2H+1} + 1\big)^2
} \simeq 
C_3
\eps^2  \int_1^\infty  
\frac
{
    x^{2H}
}
{
\big(\eps x^{2H+1} + 1\big)^2
}dx \simeq\\
&
C_3\eps \int_0^\infty  
\frac
{
   y^{2H}
}
{
\big(y^{2H+1} + 1\big)^2
}dy\xrightarrow[\eps\to 0]{}0.
\end{align*}
\end{proof}

\section{Numerical experiments}

Theorems \ref{thm-fOU} and \ref{thm-ifBm} provide an approximation for the eigenvalues and eigenfunctions, which 
is asymptotically exact and therefore more accurate for smaller eigenvalues. By keeping track of all constants 
in the proofs, it is possible to obtain  rough estimates for the residual terms in the suggested formulas; 
however, they are likely to be quite conservative. Numerical experiments, presented in this 
section, indicate that the actual accuracy of our approximation is surprisingly good, already, for relatively small values of $n$.

To solve our eigenproblems numerically we will use the quadrature method, which replaces integration with a numerical approximation.  
Since the kernels under consideration are smooth, it will suffice to work with uniform nodes and weights, in which case the 
eigenproblem \eqref{eigpr} is replaced with the system of linear equations:
$$
 \sum_{i=0}^L K\big(\tfrac i L,\tfrac j L\big)\varphi_i \tfrac 1 L  = \lambda \varphi_j, \quad j=0,...,L.
$$
The eigenvalues $\widehat \lambda_{n,L}$ and the corresponding eigenvectors $\widehat \varphi_{n,L}$, $n=0,...,L$  
for this problem approximate the solutions to the original eigeproblem of interest in the sense: 
$$
\lim_{L\to\infty} \widehat \lambda_{n,L} = \lambda_n\quad \text{and}\quad \lim_{L\to \infty} \widehat\varphi_{n,L}([xL]) = \varphi_n(x).
$$
The convergence can be quantified, under appropriate assumptions on kernel $K$ (see, e.g., \cite{R85}). 
In practical terms, $L$ is chosen, so that its further increase does not improve the accuracy beyond the desired decimal digit.   
We will compare the numerical solutions to the approximations $\widetilde \lambda_n$ and $\widetilde \varphi_n$ 
of the eigenvalues and eigenfunctions,  obtained by truncating the residual terms in our asymptotic formulas.

Our study case is the fractional Ornstein--Uhlenbeck process with $H=\frac 3 4$ and $\beta = -1$. 
A calculation shows that the kernel in \eqref{KstHlarge} can be rewritten in the form  
\begin{align*}
K(s,t)  =  &\;  \phantom{+}   \frac {c_\alpha} {2\beta}e^{\beta(t+s)}
\Big(
\Phi(t, \alpha, \beta)+ \Phi(s, \alpha, \beta)  
\Big) \\
&
-\frac {c_\alpha} {2\beta}   e^{\beta(s-t)} \Big(\Phi(t, \alpha, -\beta)-\Phi(t-s, \alpha, -\beta)\Big)  \\
&
-\frac {c_\alpha} {2\beta}  e^{\beta(t-s)}\Big(\Phi(t-s, \alpha, \beta)+\Phi(s, \alpha, -\beta)\Big), \quad t>s
\end{align*}
where $\alpha = 2-2H$, $c_\alpha = (1-\frac \alpha 2)(1-\alpha)$ and function 
$$
\Phi(t, \alpha,\beta) = \int_0^t e^{-\beta x}x^{-\alpha} dx
$$
is related through a simple scaling to the incomplete Gamma function, routinely available in numerical packages. 

Since $\lambda_n$ rapidly decrease and in view of formula \eqref{lambda_n_fOU}, it will be also interesting to compare the empirical quantity 
$\widehat \nu_{n,L}$, obtained by solving the equation 
\begin{equation}
\label{nonlineq}
\widehat \lambda_{n,L} = \sin (\pi H) \Gamma(2H+1)\frac{\nu^{1-2H}}{\nu^2+\beta^2}
\end{equation}
with the corresponding theoretical approximation $\widetilde \nu_n$, given by  \eqref{fBmnu} after truncating the residual term. 
Transcendental equation \eqref{nonlineq} must also be solved numerically, but its root can be located to within 
{\em any} desired precision and hence this computation does not introduce any significant error. 

\begin{table}[t]
\begin{center}
\resizebox{\textwidth}{!}{%
  \begin{tabular}{ || c | c | c | c | c | c | c | c | c |c|c|}
    \hline
    $n$ & 1 & 2 &  3 & 4 & 5& 6 &7 & 8 & 9 & 10 \\ \hline\hline
     $\widehat \lambda_{n,L}\times 10^3$  & 182.46  &  17.62  &   5.348  & 2.3210 & 1.2519 & 0.7574  & 0.5005 & 0.3495   & 0.2560  & 0.1937 \\ \hline
     $\widehat \nu_{n,L}$  & 1.7133  & 4.8245  &  7.8551  & 11.003  & 14.104 &  17.255 & 20.373 & 23.523  & 26.648  &  29.79 \\ \hline
      $\big|\widehat \lambda_{n,L}/\widetilde \lambda_n-1\big|\times 100\%$ & 23.5 & 9.3 &    2.5 &    2.0 &    0.8  &    0.8  & 
       0.4 &    0.4  &    0.3 &    0.3   \\ \hline
     $\big|\widehat \nu_{n,L}/\widetilde \nu_n-1\big|\times 100\%$ & 14.8  &   4.1  &    1.1  &    0.8 &  0.4 &    0.4 & 0.2 & 0.2 & 0.1 &    0.1   \\ \hline
  \end{tabular}
}
\end{center}
\

\medskip
\caption{\label{table1} Numerical versus asymptotic approximation: 
the numbers in the first two rows are specified up to the stable decimal digit (obtained for $L=10^4$).}
\end{table}

\begin{figure}[b]
%
%
\begin{tikzpicture}[scale =0.4]

\begin{axis}[%
width=6.028in,
height=4.754in,
at={(1.011in,0.642in)},
scale only axis,
xmin=0,
xmax=30,
xlabel={n},
xmajorgrids,
ymin=0,
ymax=0.25,
ymajorgrids,
axis background/.style={fill=white}
]
\addplot [color=black,solid,mark=*,mark options={solid},forget plot]
  table[row sep=crcr]{%
1	0.221114291839364\\
2	0.190746864732153\\
3	0.0797457345057424\\
4	0.0869280760622235\\
5	0.0455831117409247\\
6	0.0555645800500386\\
7	0.0314557349172553\\
8	0.0404239139007672\\
9	0.0236687538112754\\
10	0.0314421690752447\\
11	0.0186753845132586\\
12	0.0254456562093708\\
13	0.0151566030693004\\
14	0.0211199015896497\\
15	0.0125104504146094\\
16	0.017822882942589\\
17	0.0104232348486946\\
18	0.015203978601825\\
19	0.00871546445783622\\
20	0.0130555504288736\\
21	0.00727702534075547\\
22	0.0112468208530032\\
23	0.00603661539190625\\
24	0.00969132694702068\\
25	0.00494602691719592\\
26	0.0083296165083766\\
27	0.00397149385661066\\
28	0.00711946304879518\\
29	0.0030886575587914\\
30	0.00603004613262215\\
};
\end{axis}
\end{tikzpicture}%
\input{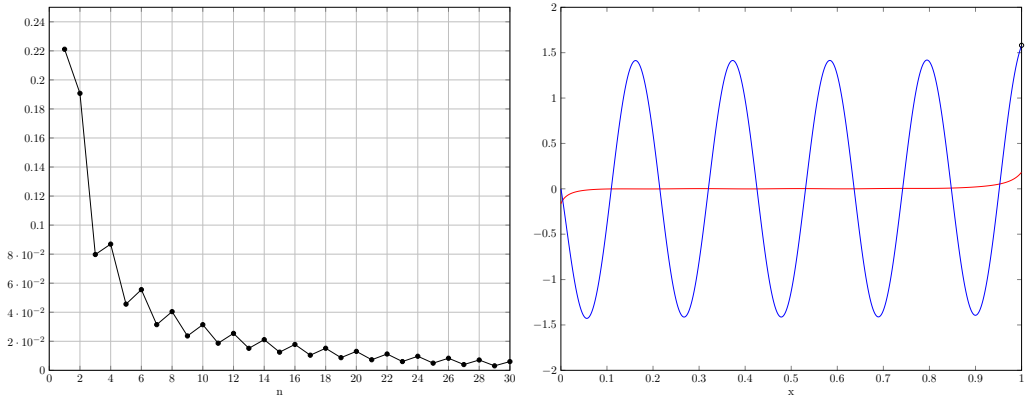}
  \caption{\label{fig-OU-1} Left: approximation error 
  $\widehat \nu_{n,L}-\widetilde \nu_n$ for $L=10^4$. 
  Right: Numerical estimate $\widehat\varphi_{n,L}$ of the eigenfunction (blue) versus the approximation error by the oscillatory term (red) for $n=10$ and $L=10^4$.}
\end{figure}

The obtained numbers are summarized in Table \ref{table1} and illustrated graphically on the left plot of Figure \ref{fig-OU-1},
which gives an idea about the actual magnitude of the $O(n^{-1})$ residual in the approximation \eqref{fBmnu}. 
The right plot of Figure \ref{fig-OU-1} illustrates the numerical estimate $\widehat \varphi_{n,L}(x)$, along with the error of
approximating the eigenfunction only by the oscillatory term from  \eqref{phinfBm}. Up to an $O(n^{-1})$ residual, 
this error coincides with the boundary layer in \eqref{phinfBm}, which pushes $\varphi_n(0)$ to zero and $\varphi_n(1)$ to 
$\sqrt{2H+1}$.

\section{Preliminaries}

\subsection{Frequently used notations} For numerical sequences $a_n$ and $b_n$, we write $a_n\propto b_n$ if 
$a_n=C b_n$ with a constant $C\ne 0$ and $a_n \sim b_n$ and $a_n \simeq b_n$ if 
$a_n \propto b_n (1+o(1))$ and $a_n =  b_n (1+o(1))$ respectively as $n\to\infty$. 
Similarly, $f(x)\propto g(x)$ stands for the equality $f(x)=C g(x)$ with a constant $C\ne 0$, etc.

Our main reference for tools from complex analysis is the text \cite{Gahov}, where the particular form of the 
Riemann boundary value problem used below is detailed in \S 43. Another reference on the subject is the classic book \cite{M46}. 
Unless stated otherwise, standard principle branches of the multivalued complex functions will be used.
We will frequently work with functions, sectionally holomorphic on the complex plane cut along the real line.
For such a function $\Psi$, we denote by $\Psi^+(t)$ and $\Psi^-(t)$ the limits of $\Psi(z)$ as $z$ approaches the point $t$
on the real line from above and below respectively:
$$
\Psi^\pm (t) := \lim_{z\to t^\pm }\Psi(z).
$$

We will frequently use the Sokhotski--Plemelj formula, which asserts that the 
Cauchy-type integral 
$$
\Phi(z):= \frac 1{2\pi i}\int_0^\infty \frac{\phi(\tau)}{\tau-z}d\tau
$$
of H\"older continuous function $\phi$ on $\Real_{>0}$, possibly with integrable singularities at the origin and infinity, 
defines a function, analytic on the cut plane  $\mathbb{C}\setminus \mathbb{R}_{\ge 0}$, whose limits across the real axis satisfy 
$$
\Phi^\pm (t) = \frac 1{2\pi i} \dashint_0^\infty \frac{\phi(\tau)}{\tau-z}d\tau \pm  \frac 1 2 \phi(t), \quad t\in \Real_{>0},
$$
where the dash integral stands for the Cauchy principle value. In particular, $\Phi(z)$ is a solution of the 
Riemann boundary value problem $\Phi^+(t) -\Phi^-(t) = \phi(t)$, $t\in \Real_{>0}$, which vanishes at infinity. 
This fact is the main building block in all boundary problems to be encountered in this paper. 

\subsection{Overview of the proofs}

The proofs below follow the framework, set up in \cite{ChK}, which, in turn, is inspired by approach in
\cite{Ukai}. 
The program is explained in detail in Section 4 of \cite{ChK}, but its implementation depends 
heavily on specificities of the kernel. Complexity of the problem is largely determined  
by the {\em structural} function $\Lambda(z)$, see \eqref{phihat} below, and particularly by the number of its zeros 
and scalability with respect to $\lambda$.
The two processes considered in this paper, the integrated f.B.m. and the fractional Ornstein--Uhlenbeck process, 
differ in both aspects from the simpler situation studied in \cite{ChK} and pose entirely new challenges, 
which is the main focus of this paper on the technical side. 
Whenever possible, we will explicitly omit calculations, which can be done similarly to analogous parts in \cite{ChK}, 
providing exact link to the relevant pages therein. 

Let us now briefly describe the main ideas behind the proofs. 
The basic object is the Laplace transform 
\begin{equation}\label{Laptr}
\widehat \varphi(z) = \int_0^1 \varphi(x)e^{-zx}dx, \quad z\in \mathbb{C},
\end{equation}
where $\varphi(x)$ is a solution to \eqref{eigpr}. 
Using underlying fractional structure of the kernels under consideration, it is possible to derive an expression for $\widehat \varphi(z)$,
which contains singularities. A typical expression has the form (see e.g. \eqref{phiz} and \eqref{phizifbm2} below): 
\begin{equation}
\label{phihat}
\widehat\varphi(z)=  P(z) -Q(z) \frac {\Phi_0(z)+e^{-z}\Phi_1(-z)}{\Lambda(z)} 
\end{equation}
where $P(z)$ and $Q(z)$ are polynomials of a finite degree, 
$\Phi_0(z)$ and $\Phi_1(z)$ are certain functions and $\Lambda(z)$ is given by an explicit formula, such as e.g. \eqref{Lambda}.

Function $\Lambda(z)$ usually have a finite number of zeros $z_1(\lambda),..., z_k(\lambda)$ and a discontinuity along the real line and
$\Phi_0(z)$ and $\Phi_1(z)$ are sectionally holomorphic on the complex plane, cut along the real axis.
Since the Laplace transform is a priori an entire function, all singularities must be removable. Therefore 
the enumerator in \eqref{phihat} must compensate the discontinuity and the zeros of $\Lambda(z)$, thus imposing on 
$\Phi_0(z)$ and $\Phi_1(z)$ two structural conditions:
\begin{equation}\label{cond1}
\lim_{z\to t^+}\frac {\Phi_0(z)+e^{-z}\Phi_1(-z)}{\Lambda(z)} = \lim_{z\to t^-}\frac {\Phi_0(z)+e^{-z}\Phi_1(-z)}{\Lambda(z)}, \quad t\in \Real
\end{equation}
and 
\begin{equation}\label{cond2} 
\Phi_0\big(z_j(\lambda)\big)+e^{-z_j(\lambda)}\Phi_1\big(-z_j(\lambda)\big)=0, \quad j=1,...,k.
\end{equation}

Using various symmetries of the problem, it is possible to rewrite \eqref{cond1} as a more explicit boundary condition 
for the limits of $\Phi_0(z)$ and $\Phi_1(z)$ across the positive real semiaxis (see e.g. \eqref{Hp} below). 
Since the eigenfunction is recovered by computing inverse Laplace transform of \eqref{phihat}, 
the original eigenproblem reduces to finding sectionally holomorphic functions $\Phi_0(z)$ and $\Phi_1(z)$ 
satisfying the following conditions:

\medskip

\begin{enumerate}
\addtolength{\itemsep}{0.7\baselineskip}
\renewcommand{\theenumi}{\roman{enumi}}
\item\label{raz}  limits $\Phi^\pm_0(t)= \lim_{z\to t^\pm}\Phi_0(z)$ and $\Phi^\pm_1(t)= \lim_{z\to t^\pm}\Phi_1(z)$
comply with \eqref{cond1}  

\item\label{dva} the behaviour $\Phi_0(z)$ and $\Phi_1(z)$ at infinity matches the a priori growth estimates, 
determined by polynomials $P(z)$ and $Q(z)$

\item\label{tri} the values of $\Phi_0(z)$ and $\Phi_1(z)$ satisfy the constraints imposed by \eqref{cond2}

\end{enumerate}

\medskip

Finding sectionally holomorphic functions subject to conditions such as  \eqref{raz}-\eqref{dva} is known in complex analysis as 
the Riemann boundary value problem. Remarkably, in our case, its unique solution can be expressed in terms 
of certain integral equations on the real semiaxis (see, e.g., \eqref{qp} below). Along with \eqref{tri}
this furnishes an integro-algberaic system of equations, which gives an equivalent characterisation for the eigenvalues and 
eigenfunctions: any solution of this system corresponds to a solution of the eigenproblem and vice versa. 
Despite its seeming complexity, this system turns out to be more amenable to asymptotic analysis, which is where the 
approximations of Theorems \ref{thm-ifBm} and \ref{thm-fOU}  ultimately come  from. 

On the technical level, the case of the fractional Ornstein--Uhlenbeck process turns out to be less involved than the integrated f.B.m.
Hence we start with the proof of Theorem \ref{thm-fOU} in the next section, deferring the more complicated proof of Theorem \ref{thm-ifBm}
to Section \ref{sec-proof-ifBm}.

\section{Proof of Theorem \ref{thm-fOU}}  

\subsection{The case $H>\frac 1 2$}
It will be convenient to define $\alpha:=2-2H\in (0,1)$, so that \eqref{KstHlarge} reads   
$$
K_\beta(s,t)= \int_0^t \int_0^s e^{\beta(t-v)}e^{\beta(s-u)}c_\alpha |u-v|^{-\alpha}dudv
$$
where $c_\alpha = (1-\frac \alpha 2)(1-\alpha)$. The eigenproblem \eqref{eigpr} takes the form 
\begin{equation}\label{OUeigHlarge}
\int_0^1 \left(\int_0^x \int_0^y e^{\beta(x-u)}e^{\beta(y-v)}c_\alpha |u-v|^{-\alpha}dvdu\right) \varphi(y)dy=\lambda \varphi(x),\quad x\in [0,1].
\end{equation}

\subsubsection{The Laplace transform} 

The starting point of the proof is an expression for the Laplace transform, derived on the basis of \eqref{OUeigHlarge}:

\begin{lem}\label{lem5.1}
Let $(\lambda, \varphi)$ be a solution of \eqref{OUeigHlarge}, then the Laplace transform \eqref{Laptr} satisfies 
\begin{equation}
\label{phiz} 
\widehat\varphi(z) = \widehat \varphi(-\beta)  -\frac {z+\beta}{\Lambda(z)} 
 \Big(
 \Phi_0(z)+e^{-z}\Phi_1(-z)
\Big),
\end{equation}
where 
\begin{equation}
\label{Lambda}
\Lambda(z) = \frac {\Gamma(\alpha)\lambda}{c_\alpha}(z^2-\beta^2) + \int_0^\infty \frac{2 t^{\alpha }}{t^2-z^2} 
 dt 
\end{equation}
and functions $\Phi_0(z)$ and $\Phi_1(z)$
are sectionally holomorphic on the cut plane $\mathbb{C}\setminus \Real_{>0}$.
\end{lem}

\begin{proof}

Taking derivative of both sides of \eqref{OUeigHlarge} we get 
\begin{equation}\label{eigpr_fOU}
\int_0^1 \left( \int_0^{y}   e^{ -\beta v} c_\alpha |  x- v|^{-\alpha} dv\right) e^{\beta   y } \varphi(y)dy
+
\beta\lambda \varphi(x) = \lambda \varphi'(x), \quad x\in [0,1].
\end{equation}
Integrating by parts, the first term can be rewritten as 
$$
\int_0^1 \left( \int_0^{y}   e^{ -\beta v} c_\alpha |  x- v|^{-\alpha} dv\right) e^{\beta   y } \varphi(y)dy=
\int_0^1      c_\alpha |  x- y|^{-\alpha}    \psi(y) dy,
$$
where we defined 
\begin{equation}\label{psidef}
\psi(x) := e^{-\beta x}\displaystyle \int_x^1 e^{\beta r}\varphi(r)dr.
\end{equation}
By definition $\psi(x)$ satisfies 
\begin{equation}
\label{psiphi}
\psi'(x)    +\beta \psi(x)   =- \varphi(x),
\end{equation}
and therefore \eqref{eigpr_fOU} is equivalent to the generalized eigenproblem:  
\begin{equation}\label{geneig}
\begin{aligned}
\int_0^1   &   c_\alpha |  x- y|^{-\alpha}    \psi(y) dy
 = \lambda \Big(\beta^2 \psi(x)-\psi''(x) \Big) 
 , \quad x\in [0,1] \\
&
\psi(1) =  0, \; 
\psi'(0)    +\beta \psi(0)  =0
\end{aligned}
\end{equation}
Plugging the identity  
\begin{equation}\label{Gfla}
|x-y|^{-\alpha} =\frac 1{\Gamma(\alpha)} \int_0^\infty t^{\alpha-1}e^{-t|x-y|}dt,\quad \alpha\in (0,1),
\end{equation}
into \eqref{geneig} gives
\begin{equation}
\label{one}
  \int_0^\infty t^{\alpha-1}   u(x,t) dt
 =\frac {\Gamma(\alpha)\lambda}{c_\alpha}   \Big(\beta^2 \psi(x)-\psi''(x) \Big),
\end{equation}
where we defined 
\begin{equation}
\label{udef}
u(x,t) := \int_0^1 \psi(y)  e^{-t|x-y|}dy. 
\end{equation}
On the other hand, integrating twice by parts gives
\begin{equation}\label{psitagtag}
\begin{aligned}
\widehat \psi''(z) = & \int_0^1 \psi''(x) e^{-zx}dx = \psi'(1)e^{-z}-\psi'(0) +  
z\psi(1)e^{-z}-z\psi(0) + z^2 \widehat \psi(z) = \\
& 
\psi'(1)e^{-z}+(\beta  -z)\psi(0) + z^2 \widehat \psi(z),
\end{aligned}
\end{equation}
where the last equality holds due to the boundary conditions in \eqref{geneig}.
Taking the Laplace transform of \eqref{one} and plugging \eqref{psitagtag} we get 
\begin{equation}
\label{rel1}
 \int_0^\infty t^{\alpha-1}  \widehat u(z,t) dt
 =\frac {\Gamma(\alpha)\lambda}{c_\alpha}   \Big((\beta^2 -z^2)\widehat\psi(z)
 -\psi'(1)e^{-z}-(\beta  -z)\psi(0)  
  \Big).
\end{equation}

A different expression for $\widehat u(z,t)$ can be obtained, using definition \eqref{udef}: taking two derivatives gives
\begin{equation}\label{u2tag}
u''(x,t) =  t^2  u(x,t)-2t  \psi(x),
\end{equation}
subject to the boundary conditions
\begin{equation}\label{utag01}
\begin{aligned}
& u'(0,t)  = \phantom{+} t u(0,t)\\
& u'(1,t)  = -t u(1,t). 
\end{aligned}
\end{equation}
Integrating twice by parts gives
\begin{equation}\label{utagtag}
\begin{aligned}
\widehat u''(z,t) = & \int_0^1 u''(x,t) e^{-zx}dx = u'(1,t)e^{-z}-u'(0,t)+z\widehat u'(z,t) =\\
&
  (z-t) e^{-z} u(1,t) -(z+t) u(0,t) +z^2\widehat u(z,t),
\end{aligned}
\end{equation}
where we used \eqref{utag01}. 
Taking the Laplace transform of \eqref{u2tag} and plugging \eqref{utagtag} we get 
\begin{equation}
\label{rel2}
\widehat u(z,t)  = \frac 1{z-t}  u(0,t)
-\frac 1{z+t}   u(1,t)e^{-z}  
- \frac{2t}{z^2-t^2} \widehat \psi(z).
\end{equation}
Now combining \eqref{rel1} and \eqref{rel2} and rearranging we obtain 
$$
 \widehat \psi(z) 
 =\frac 1{\Lambda(z)} 
 \Big(
 \Phi_0(z)+e^{-z}\Phi_1(-z) 
\Big)
$$
where $\Lambda(z)$ is defined in \eqref{Lambda}  
and 
\begin{equation}\label{Phi0Phi1}
\begin{aligned}
& \Phi_0(z):= -\frac {\Gamma(\alpha)\lambda}{c_\alpha}(\beta  -z)\psi(0)+\int_0^\infty \frac {t^{\alpha-1}}{t-z}  u(0,t)dt\\
&
\Phi_1(z) := -\frac {\Gamma(\alpha)\lambda}{c_\alpha}\psi'(1) +  \int_0^\infty \frac {t^{\alpha-1}}{t-z}   u(1,t)  dt 
\end{aligned}
\end{equation}
Since $\psi(1)=0$, we have
$$
\widehat \psi'(z) =   - \psi(0) + z \widehat \psi(z)
$$
and, on the other hand, by \eqref{psiphi}
$$
\widehat\psi'(z)    +\beta\widehat\psi(z)   =- \widehat\varphi(z).
$$
Combining these two expressions gives 
$$
 \widehat\varphi(z)=\psi(0)- (z +\beta) \widehat\psi(z)   
$$
and, in turn, the expression claimed in \eqref{phiz}.
\end{proof}

\medskip 

The following lemma details the structure of $\Lambda(z)$ and summarizes some of its properties, 
to be used later in the proofs:

\medskip 

\begin{lem}\label{lem4.2} \

\medskip
\noindent 
a) $\Lambda(z)$  admits the expression  
\begin{equation}\label{Lfla}
\Lambda (z) =
\frac {\Gamma(\alpha)\lambda}{c_\alpha}(z^2-\beta^2)
+
 z^{\alpha-1} 
\frac{\pi   }{ \cos \frac{\pi } 2\alpha}
\begin{cases}
e^{\frac{1-\alpha}{2}\pi i} & \arg (z) \in (0,\pi) \\
e^{-\frac{1-\alpha}{2}\pi i } &   \arg(z)\in (-\pi, 0)
\end{cases}
\end{equation}
and has two zeros at $\pm z_0 = \pm i \nu$ with $\nu>0$ satisfying the equation  
\begin{equation}\label{lambdanu}
\lambda 
=
\frac {c_\alpha}{\Gamma(\alpha)} \frac {\pi   } { \cos \frac{\pi } 2\alpha}
\frac{\nu^{\alpha-1}}{\beta^2 +\nu^2}.
\end{equation}

\medskip
\noindent 
b) The limits $\Lambda^\pm(t) = \lim_{z\to t^\pm}\Lambda(z)$ across the real line are given by 
$$
\Lambda^\pm (t) =
 \frac {\Gamma(\alpha)\lambda}{c_\alpha}(t^2-\beta^2 )
+
|t|^{\alpha-1} 
\frac{\pi   }{ \cos \frac{\pi } 2\alpha}
\begin{cases}
e^{\pm \frac{1-\alpha}{2}\pi i} & \quad t>0 \\
e^{\mp\frac{1-\alpha}{2}\pi i } & \quad t<0
\end{cases}
$$
and satisfy the symmetries 
\begin{align}
 \Lambda^+(t) & =\overline{\Lambda^-(t)} \label{conjp}\\
 \frac{\Lambda^+(t)}{\Lambda^-(t)} & =\frac{\Lambda^-(-t)}{\Lambda^+(-t)}   \label{prop}   \\
 \big|\Lambda^+(t)\big| &  =\big|\Lambda^+(-t)\big|   \label{absL}
\end{align}

\medskip
\noindent 
c) The argument $\theta(t):=\arg\{\Lambda^+(t)\}\in (-\pi, \pi]$ is an odd function, $\theta(-t)=-\theta(t)$,
\begin{equation}
\label{thetanu}
\theta(t) = \arctan \frac{
\sin \frac{1-\alpha}{2}\pi 
}
{
\frac{(t/\nu)^2-(\beta/\nu)^2}{1+(\beta/\nu)^2}(t/\nu)^{1-\alpha} 
+
\cos \frac{1-\alpha}{2}\pi  
},\qquad t>0
\end{equation}
continuous on $(0,\infty)$ with $\theta(0+) := \frac{1-\alpha}{2}\pi>0$ and $\theta(\infty) :=\lim_{t\to\infty}\theta(t)=0$.
For all $\nu$ large enough, the scaled function $\theta(u;\nu):=\theta(u\nu)$ satisfies the bound 
\begin{equation}\label{thetabeta}
\Big|\theta(u;\nu) -\theta_0(u)\Big|\le g(u)(\beta/\nu)^2
\end{equation} 
where $g(u)$ does not depend on $\nu$, is continuous on $[0,\infty)$, $g(u)\sim u^{1-\alpha}$ as $u\to 0$ and 
$g(u)\sim u^{\alpha-3}$ as $u\to\infty$ and 
\begin{equation}
\label{theta0}
\theta_0(u) := \lim_{\nu\to\infty}\theta(u\nu) = 
\arctan \frac{
\sin \frac{1-\alpha}{2}\pi 
}
{
u^{3-\alpha} 
+
\cos \frac{1-\alpha}{2}\pi
}.
\end{equation}
For any fixed $\beta\in \Real$, 
\begin{equation}\label{balphadef}
b_\alpha(\beta,\nu) :=\frac{1}{\pi} \int_0^\infty \theta(u;\nu) du \xrightarrow[\nu\to\infty]{} 
\frac{1}{\pi} \int_0^\infty \theta_0(u) du =
\frac{\sin{(\frac{\pi}{3-\alpha}\frac{1-\alpha}{2})}}{\sin{\frac{\pi}{3-\alpha}}}=:b_\alpha,
\end{equation}
and  
\begin{equation}\label{bnu}
\big|b_\alpha(\beta,\nu)-b_\alpha\big|\le C (\beta/\nu)^2,
\end{equation}
with a constant $C$. 

\end{lem} 

\begin{proof}\

\medskip
\noindent
a) 
Standard contour integration gives the identity  
\begin{equation}\label{intfla}
\int_0^{\infty} \frac{t^\alpha}{t^2-z^2}dt  = z^{\alpha-1}\frac 1 2
\frac{\pi   }{ \cos \frac{\pi } 2\alpha}
\begin{cases}
e^{\frac{1-\alpha}{2}\pi i} & \arg (z) \in (0,\pi) \\
e^{-\frac{1-\alpha}{2}\pi i } &   \arg(z)\in (-\pi, 0)
\end{cases}
\end{equation}
and in turn the expression in \eqref{Lfla}. 
To find the roots of $\Lambda(z)$ in the upper half plane,  
let $z= \nu e^{i\omega}$ with $\nu>0$ and $\omega\in (0,\pi)$ so that the 
equation $\Lambda(z)=0$ becomes 
$$
\kappa (\nu^2e^{2i\omega}-\beta^2)
+ 
 \nu^{\alpha-1}e^{i(\omega-\frac \pi 2)(\alpha-1)}  =0,
$$
where we set 
$
\kappa := \displaystyle\frac {\Gamma(\alpha)\lambda}{c_\alpha}\frac{ \cos \frac{\pi } 2\alpha}{\pi   }
$
for brevity. 
This implies 
$$
\kappa  \nu^2\sin 2\omega
+ 
 \nu^{\alpha-1}\sin (\omega-\tfrac \pi 2)(\alpha-1)   =0.
$$
For both $\omega\in (0,\frac \pi 2)$ and $\omega\in (\frac \pi 2 ,\pi)$, the sines are either positive or negative simultaneously and 
hence the equality cannot hold. Therefore the only root in the upper half plane is $i\nu$ with $\nu$ solving the equation \eqref{lambdanu}.
By definition \eqref{Lambda}, all zeros of $\Lambda(z)$ must be conjugates and hence 
the only other root in the lower half plane is $-i\nu$. 

\medskip 
\noindent 
b) All the claims are checked by direct calculations. 

\medskip 
\noindent
c) 
The function $f(u):= (u^2 - (\beta/\nu)^2)u^{1-\alpha}$, $u\in \Real_+$ vanishes at $u=0$ and $u=\beta/\nu$ and has a unique minimum    
$
\min_{u\ge 0}f(u)  = -(\beta/\nu)^{3-\alpha}\frac{2}{3-\alpha}\left(\frac{1-\alpha}{3-\alpha}\right)^{\frac{1-\alpha}{2}}.
$
Therefore for a fixed $\beta$ and all $\nu$ large enough the denominator in \eqref{thetanu} is bounded away from zero uniformly in $u$. 
The rest of the claims readily follow by direct inspection (the value of $\beta_\alpha$ is computed in \cite{ChK}).  

\end{proof}

\subsubsection{Removal of singularities}\label{sec512}

Since the Laplace transform is a priori an entire function, both poles and the discontinuity on the real line in \eqref{phiz} 
must be removable. Therefore the functions $\Phi_0(z)$ and $\Phi_1(z)$ must satisfy  
\begin{equation}
\label{algc}
 \Phi_0(\pm z_0)+e^{\mp z_0}\Phi_1(\mp z_0) = 0
\end{equation}
and 
$$
\lim_{z\to t^+}\frac 1 {\Lambda(z)} \big(e^{-z}\Phi_1(-z)+\Phi_0(z)\big) =
\lim_{z\to t^-}\frac 1 {\Lambda(z)} \big(e^{-z}\Phi_1(-z)+\Phi_0(z)\big).
$$
The latter condition gives  
\begin{align*}
&
\frac {1}{\Lambda^+(t)} 
 \Big(
 \Phi_0^+(t)+e^{-t}\Phi_1(-t)
\Big)=
\frac {1}{\Lambda^-(t)} 
 \Big(
 \Phi_0^-(t)+e^{-t}\Phi_1(-t)
\Big), \quad t>0 \\
&
\frac {1}{\Lambda^+(t)} 
 \Big(
 \Phi_0(t)+e^{-t}\Phi_1^-(-t)
\Big)=
\frac {1}{\Lambda^-(t)} 
 \Big(
 \Phi_0(t)+e^{-t}\Phi_1^+(-t)
\Big), \quad t<0
\end{align*}
which, in view of \eqref{prop}, can be rewritten as
\begin{equation}\label{Hp1}
\begin{aligned}
& 
\Phi_0^+(t) - \frac{\Lambda^+(t)}{\Lambda^-(t)}\Phi_0^-(t) =  e^{-t} \Phi_1(-t) \left(\frac{\Lambda^+(t)}{\Lambda^-(t)}-1\right) \\
&
\Phi_1^+(t) - \frac{\Lambda^+(t)}{\Lambda^-(t)}\Phi_1^-(t) =   e^{-t} \Phi_0(-t) \left(\frac{\Lambda^+(t)}{\Lambda^-(t)}-1\right)
\end{aligned}
\end{equation}
Further, since  $\Lambda^+(t)$ and $\Lambda^-(t)$ are complex conjugates,    
$
\Lambda^+(t)/\Lambda^-(t)=e^{2i\theta(t)}
$
and
$$
\frac{\Lambda^+(t)}{\Lambda^-(t)}-1=e^{2i\theta(t)}-1 = 2i e^{i\theta(t)}\sin\theta(t).
$$
Therefore \eqref{Hp1} takes the  form
\begin{equation}
\label{Hp}
\begin{aligned}
&
\Phi_0^+(t) - e^{2i\theta(t)}\Phi_0^-(t) = 2i  e^{-t} e^{i\theta(t)}\sin\theta(t) \Phi_1(-t)
\\
&
\Phi_1^+(t) - e^{2i\theta(t)}\Phi_1^-(t) = 2i e^{-t}  e^{i\theta(t)}\sin\theta(t) \Phi_0(-t)
\end{aligned}\qquad t>0.
\end{equation}

By definition \eqref{udef} both $tu(0,t)$ and $tu(1,t)$ are bounded and hence the functions in \eqref{Phi0Phi1}
satisfy the following a priori estimates:
\begin{equation}\label{grzero}
 \Phi_1(z)\sim z^{\alpha-1} \quad \text{and}\quad \Phi_0(z)  \sim  z^{\alpha-1} \qquad \text{as}\ z\to 0,
\end{equation}
and 
\begin{equation}
\label{grinf}
\begin{aligned}
\Phi_0(z)  &= 2c_2(\beta-z) + O(z^{-1}) \\
\Phi_1(z)  & =  2c_1 + O(z^{-1})
\end{aligned}
\qquad \text{as\ } z\to \infty
\end{equation}
where we defined 
\begin{equation}
\label{c1c2}
c_1   = -\frac 1 2\frac{\Gamma(\alpha)\lambda}{c_\alpha}\psi'(1)  \quad \text{and}\quad 
c_2  = -\frac 1 2\frac{\Gamma(\alpha)\lambda}{c_\alpha}\psi(0).
\end{equation}

\subsubsection{An equivalent formulation of the eigenproblem}
The Laplace transform of any $\varphi$, which satisfies \eqref{OUeigHlarge}, is given by the formula 
\eqref{phiz}, in which sectionally holomorphic functions $\Phi_0(z)$ and $\Phi_1(z)$ have the growth rates 
\eqref{grinf} and \eqref{grzero} and satisfy the boundary conditions \eqref{Hp} and the constraint \eqref{algc}. 
In this subsection we will establish one-to-one correspondence between such functions and square integrable solutions 
of a certain system of integral equations. This is done by the solution technique of the Riemann boundary
value problems.

Let us start with the homogeneous Riemann problem of finding a function $X(z)$, sectionally holomorphic on the cut plane 
$\mathbb{C}\setminus \Real_{>0}$ and satisfying the boudnary condition
\begin{equation}
\label{Rbvphom}
X^+(t) - e^{2i\theta(t)}X^-(t) =0, \quad t\in \Real_{>0}.
\end{equation}
All such functions can be found by the Sokhotski--Plemelj formula 
\begin{equation}
\label{Xz}
X(z)= z^k X_c(z)=z^k \exp \left(\frac 1 \pi \int_0^\infty\frac{\theta(t)}{t-z}dt\right), \quad z\in \mathbb{C}\setminus \Real_{>0},
\end{equation}
where $k$ is an arbitrary integer to be fixed later. Canonical part $X_c(z)$ satisfies
\begin{equation}\label{Xczinf}
X_c(z) = 1 - z^{-1} \nu b_\alpha(\beta, \nu) + O(z^{-2}) \quad \text{as}\ z\to\infty,
\end{equation}
where $b_\alpha(\beta, \nu)$ is defined in \eqref{balphadef}, and
\begin{equation}
\label{Xcz0}
X_c(z) \sim z^{ \frac{\alpha-1}{2}}  \quad \text{as}\ z\to 0.
\end{equation}

Now define   
\begin{equation}
\label{SDdef}
\begin{aligned}
& S(z):= \frac{\Phi_0(z)+\Phi_1(z)}{2X(z)} \\
& D(z):= \frac{\Phi_0(z)-\Phi_1(z)}{2X(z)}
\end{aligned}
\end{equation}
Combining \eqref{Hp} and \eqref{Rbvphom} shows that these function satisfy the boundary conditions
\begin{equation}\label{bcndSD}
\begin{aligned}
S^+(t)-S^-(t)  &= \phantom{+}2i h(t) e^{-t} S(-t) \\
D^+(t)-D^-(t)  &=  - 2i h(t)e^{-t} D(-t)
\end{aligned}\qquad t>0,
\end{equation}
where  
\begin{equation}\label{hdefine}
h(t):=e^{i\theta(t)}\sin \theta(t)\frac{X(-t)}{X^+(t)}.
\end{equation}
As in \cite{ChK} (see eq. (5.37)) this function satisfies the formula  
\begin{equation}\label{ht}
h(t) = \exp \left(-\frac 1 \pi \int_0^\infty\theta'(s) \log\left| \frac{t+s}{t-s}\right|ds\right)\sin \theta(t)
\end{equation}
and therefore is real valued, Holder continuous on $\Real_{>0}$ with 
$h(0):=\sin \theta(0+)=\sin \frac{1-\alpha}{2} \pi$.

\medskip

Applying the Sokhotski--Plemelj formula to \eqref{bcndSD}, we obtain the following representation
\begin{equation}
\label{SD}
\begin{aligned}
& 
S(z) = \phantom{+}\frac 1 \pi \int_0^\infty \frac{h(t)e^{-t}}{t-z}S(-t) dt + P_S(z) \\
&
D(z) = -\frac 1 \pi \int_0^\infty \frac{h(t)e^{-t}}{t-z}D(-t)dt + P_D(z)
\end{aligned}
\end{equation}
where $P_S(z)$ and $P_D(z)$ are polynomials to be chosen to match the a priori growth of $S(z)$ and $D(z)$ as $z\to\infty$. 
Note that \eqref{SD} requires that $S(-t)$ and $D(-t)$ are integrable near the origin. In view of the a priori estimates 
\eqref{grzero} and \eqref{Xcz0}, the choice of $k$ in \eqref{Xz} is limited by condition $k < \frac {\alpha +1}2$.
In fact, in what follows we will also need $S(-t)$ and $D(-t)$ to be {\em square} integrabile, which implies further restriction 
$k < \frac \alpha 2$. A convenient choice is $k =0$, which we will fix hereafter, that is, set $X(z):=X_c(z)$. 

Since for any numbers $a$, $b$ and $c$
$$
\frac {az+b+O(z^{-1})}{1-cz^{-1}+ O(z^{-2})}=a(z+c) + b + O(z^{-1})\quad \text{as} \ z\to\infty
$$
the a priori estimates \eqref{grinf} and \eqref{Xczinf} imply 
\begin{align*}
S(z) & =  c_2 \big(-z+\beta -\nu b_\alpha(\beta,\nu)\big) + c_1 + O(z^{-1})\\
D(z) & =  c_2 \big(-z+\beta -\nu b_\alpha(\beta,\nu)\big) - c_1 + O(z^{-1})
\end{align*}
This determines the choice of the polynomials in \eqref{SD} 
\begin{align*}
P_S(z) &:=  c_2(-z+\beta-\nu b_\alpha(\beta,\nu))+c_1\\
P_D(z) &:=  c_2(-z+\beta-\nu b_\alpha(\beta,\nu))-c_1, 
\end{align*}
where constants $c_1$ and $c_2$ are defined in \eqref{c1c2}. 
If we now set $z:=-t$ with $t>0$ in \eqref{SD}, the integral equations for $S(-t)$ and $D(-t)$ are obtained:
$$
\begin{aligned}
& 
S(-t) = \phantom{+}\frac 1 \pi \int_0^\infty \frac{h(s)e^{-s}}{s+t}S(-s) ds + c_2\big(t+\beta-\nu b_\alpha(\beta,\nu)\big)+c_1\\
&
D(-t) = -\frac 1 \pi \int_0^\infty \frac{h(s)e^{-s}}{s+t}D(-s)ds + c_2\big(t+\beta-\nu b_\alpha(\beta,\nu)\big)-c_1
\end{aligned}
$$

Consider now the integral equations 
\begin{equation}\label{qp}
p^\pm_j (t)  = \pm \frac 1 \pi \int_0^\infty \frac{h_\beta(s;\nu)e^{-\nu s}}{s+t} p^\pm_j(s)ds+t^j, \quad j\in \{0,1\}
\end{equation}
where $h_\beta(u;\nu):=h(u \nu)$, $u>0$. Below we will argue that, for all sufficiently large $\nu$, 
these equations have unique solutions, such that $p^\pm_0 (t)-1$ and $p^\pm_1(t)-t$ belong to $L^2(0,\infty)$. 
Let us extend the domain of $p^\pm_j(z)$ to the cut plane  by replacing $t$ with $z\in \mathbb{C}\setminus \Real_{\ge 0}$ 
in the right hand side of \eqref{qp}.

Since $S(-t)$ and $D(-t)$ are a priori square integrable near the origin,
by linearity 
\begin{align*}
& 
S(z\nu ) =   
  c_2 \nu p^+_1(-z) + \Big(c_2\big(\beta-\nu b_\alpha(\beta,\nu)\big)+c_1\Big)p^+_0(-z)\\
&
D(z\nu ) =  
 c_2 \nu p^-_1(-z) + \Big(c_2\big(\beta-\nu b_\alpha(\beta,\nu)\big)-c_1\Big)p^-_0(-z)
\end{align*}
Consequently, if we let 
\begin{align*}
&
a_\pm (z):= p^+_0(z)\pm p^-_0(z) \\
&
b_\pm (z):= p^+_1(z)\pm p^-_1(z)
\end{align*} 
by definition \eqref{SDdef}
\begin{equation}\label{Phiba}
\begin{aligned}
\Phi_0(z\nu) & =
c_2 \nu X(z\nu)\Big(b_+(-z) +   \big(  \beta/\nu -  b_\alpha(\beta,\nu)\big)a_+(-z)\Big)  +c_1X(z\nu) a_-(-z)
\\
\Phi_1(z\nu) & = 
c_2 \nu X(z\nu)\Big(b_-(-z) +   \big(\beta/\nu-  b_\alpha(\beta,\nu)\big)a_-(-z)\Big)+c_1X(z\nu) a_+(-z)
\end{aligned}
\end{equation} 
Setting $X_\beta(z;\nu):=X(z\nu)$ and plugging \eqref{algc}, we therefore obtain 
\begin{equation}\label{linsys}
c_2 \nu \xi
 + 
c_1 \eta  =0
\end{equation} 
where 
\begin{equation}\label{xieta}
\begin{aligned}
\xi :=\, & 
e^{i\nu/2}
X_\beta(i;\nu)\Big(b_+(-i) +   \big(  \beta/\nu -  b_\alpha(\beta,\nu)\big)a_+(-i)\Big) + \\
 &  e^{-i\nu/2}
  X_\beta(-i;\nu)\Big(b_-(i) +   \big(\beta/\nu-  b_\alpha(\beta,\nu)\big)a_-(i)\Big)
\\  
\eta :=\,  &
e^{i\nu/2} X_\beta(i;\nu) a_-(-i)+e^{-i\nu/2}  X_\beta(-i;\nu) a_+(i)
\end{aligned}
\end{equation}
Since $c_1$ and $c_2$ are real, nontrivial solutions to \eqref{linsys} are possible if and only if
\begin{equation}
\label{Im}
\Im\{\xi \overline{\eta}\}=0,
\end{equation}
in which case $c_1 = - c_2 \nu \xi /\eta$. Plugging this into \eqref{Phiba} we get
\begin{equation}
\label{Phinorm}
\begin{aligned}
\Phi_0(z) /c_2 \nu & =
 X(z)\Big(b_+(-z/\nu) +   \big(  \beta/\nu -  b_\alpha(\beta,\nu)\big)a_+(-z/\nu)\Big)  -   \frac \xi \eta X(z) a_-(-z/\nu)
\\
\Phi_1(z)/c_2 \nu & = 
  X(z)\Big(b_-(-z/\nu) +   \big(\beta/\nu-  b_\alpha(\beta,\nu)\big)a_-(-z/\nu)\Big)-   \frac \xi  \eta X(z) a_+(-z/\nu)
\end{aligned}
\end{equation}

\pagebreak[3]

To recap, we arrive at the following equivalent formulation of eigenproblem \eqref{OUeigHlarge}:

\begin{lem}\label{lem5.3}
Let $(p^\pm_0, p^\pm_1, \nu)$ with $\nu>0$ be a solution of the system, which consists of the integral equations \eqref{qp} 
and the algebraic equation \eqref{Im}. Let $\varphi$ be defined by the Laplace transform, given by the formula \eqref{phiz}, 
where $\Phi_0(z)$ and $\Phi_1(z)$ are given by \eqref{Phinorm} and let $\lambda$ be defined by \eqref{lambdanu}.
Then the pair $(\lambda, \varphi)$ solves the eigenproblem \eqref{OUeigHlarge}. Conversely, any solution $(\lambda, \varphi)$ of 
\eqref{OUeigHlarge} defines a solution to the above integro-algebraic system. 
\end{lem}

\medskip

The following lemma determines the precise asymptotics of $X_\beta(i;\nu)$ as $\nu\to \infty$, used 
in the calculations to follow:

\begin{lem}\label{lemXbeta}
$$
\arg\big\{X_\beta(i;\nu)\big\} = \frac{1-\alpha}{8}\pi + O(\nu^{-2}) 
\quad \text{and} \quad \big|X_\beta(i;\nu)\big| = \sqrt{\frac{3-\alpha}{2}} + O(\nu^{-2})\quad \nu\to\infty.
$$
\end{lem}

\begin{proof}
The claimed constants are the argument and the absolute value of  (see \eqref{theta0})
$$
X_0(i):= \lim_{\nu\to\infty}X_\beta(i;\nu) = 
\exp \left(\frac 1 \pi \int_0^\infty\frac{\theta_0(u)}{u -i}du\right),
$$ 
computed in Lemma 5.5 \cite{ChK}. The estimates of the residuals follow  from
\eqref{thetabeta}.
\end{proof}

\subsubsection{Properties of the integro-algebraic system}

Solvability of the system, introduced in Lemma \ref{lem5.3} relies on the 
contracting properties of the operator 
$$
(A f)(t) := \frac 1\pi \int_0^\infty \frac{h_\beta(s;\nu)e^{-\nu s}}{s+t}f(s)ds
$$
where $h_\beta(u;\nu):= h(u\nu)$ (see \eqref{ht}):
\begin{equation}
\label{hbetanu}
h_\beta(u;\nu) = 
\exp \left(-\frac 1 \pi  \int_0^\infty\theta'(v;\nu) \log\left| \frac{u +v }{u -v }\right|dv \right)
\sin \theta(u;\nu).
\end{equation}

\begin{lem}\label{lem5.5}
The operator $A$ is a contraction on $L^2(0,\infty)$ for all $\nu$ large enough. More precisely, 
for any $\alpha_0\in (0,1]$ there exists an $\eps>0$ and a constant $\nu'>0$, such that 
$\|A\|\le 1-\eps$ for all $\nu\ge \nu'$ and all $\alpha \in [\alpha_0,1]$. 
\end{lem}

\begin{proof}
A lengthy but otherwise direct calculation  shows that for all $\nu$ large enough and all $\alpha\in [\alpha_0,1]$, 
the exponent in \eqref{hbetanu} is bounded by a continuous function $f(u)$, which does not depend neither on $\alpha$ 
nor on $\nu$ and converges to $1$ as $u\to 0$ and $u\to\infty$. Consequently 
$$
h_\beta(u,\nu)\le f(u) \sin \theta(u;\nu)\le \|f\|_\infty.
$$
%
Further, since $\theta_0(0+) = \frac {1-\alpha}2\pi$, the estimate \eqref{thetabeta} implies that 
on a neighborhood of the origin we have
$$
\sin \theta(u;\nu) \le \frac 1 2+\frac 1 2\sin \frac {1-\alpha_0}2\pi =: 1-3\eps
$$ 
for all sufficiently large $\nu$. Since $f(0)=1$, this also guarantees that $h_\beta(u,\nu)<1-2\eps$ on some neighborhood of the 
origin, for all large $\nu$. But then, since $\sup_{\nu>0}\|h_\beta\|_\infty\le \|f\|_\infty$, a constant $\nu'$ can be chosen so that 
$h_\beta(u,\nu)e^{-\nu u}<1-\eps$ for all $u>0$ and all $\nu\ge \nu'$. The claim now follows by the same calculation as in Lemma 5.6 in \cite{ChK}.

\end{proof}

\pagebreak[3]

The following estimates are the key to asymptotic analysis of the integro-algebraic system:

\begin{lem}\label{lemest}
For any $\alpha_0\in (0,1]$ there exist constants $\nu'$ and $C$, such that for all $\nu\ge \nu'$ and 
all $\alpha\in [\alpha_0,1]$
\begin{align*}
& \big|a_-(\pm i)\big| \le C\nu^{-1}, \quad \big|a_+(\pm i)-2\big|\le C \nu^{-1} \\
& \big|b_-(\pm i)\big| \le C\nu^{-2}, \quad \big|b_+(\pm i)\mp 2i\big|\le C\nu^{-2} 
\end{align*}
and for all $\tau>0$
\begin{align*}
& \big|a_-(\tau)\big| \le C\nu^{-1}\tau^{-1}, \quad \big|a_+(\tau)-2\big|\le C \nu^{-1}\tau^{-1} \\
& \big|b_-(\tau)\big| \le C\nu^{-2}\tau^{-1}, \quad \big|b_+(\tau)\mp 2\tau\big|\le C\nu^{-2} \tau^{-1}
\end{align*}
\end{lem}

\begin{proof}
As shown in the proof of the previous lemma, $h_\beta(u;\nu)$ is bounded by a constant, which depends only on $\alpha_0$, for all 
sufficiently large $\nu$. With this estimate, the bounds are obtained as in Lemma 5.7 in \cite{ChK}.
\end{proof}

\subsubsection{Inversion of the Laplace transform}

The following lemma derives an expression for the eigenfunctions in terms of solutions to the integro-algebraic system of 
equations introduced in Lemma \ref{lem5.3}:

\begin{lem}\label{lemeigf}
Let $(\Phi_0,\Phi_1,\nu)$ satisfy the integro-algebraic system from Lemma \ref{lem5.3}. Then 
the corresponding eigenfunction $\varphi$, given by the Laplace transform \eqref{phiz}, satisfies   
\begin{align}\label{phifla}
&
\varphi(x)  = 
- \nu^{3-\alpha}\frac{ \cos \frac{\pi } 2\alpha}{\pi   }   2
\Re\bigg\{e^{i\nu x}\Phi_0(i\nu )\frac { 1-i(\beta/\nu) }{  
\frac{2}{(\beta/\nu)^2 +1}  -\alpha+1 }\bigg\}+\\
\nonumber
&
 \nu^{3-\alpha}\frac{ \cos \frac{\pi } 2\alpha}{\pi   }  \frac 1 \pi\int_0^\infty \frac{  \sin\theta_\beta(u;\nu)}{\gamma_\beta(u;\nu)}\left( e^{-(1-x)u\nu} \big(u+\tfrac \beta \nu\big) \Phi_1(-u\nu)
 -
 e^{-u\nu x} \big(u-\tfrac \beta \nu \big)\Phi_0(-u\nu) \right)du,
\end{align}
where $\gamma_\beta(u;\nu)$ is defined in \eqref{gammabeta} below.
Moreover, 
\begin{equation}\label{flaphi}
\begin{aligned}
\int_0^1 e^{\beta x}\varphi(x)dx &= -\nu^{3-\alpha}\frac {\cos \frac \pi 2 \alpha}{\pi}  2 c_2  \big(1+(\beta/\nu)^2\big) \\
\varphi(1) &= -\nu^{3-\alpha}\frac {\cos \frac \pi 2 \alpha}{\pi} 2 c_2\nu \frac \xi\eta \big(1+(\beta/\nu)^2\big)
\end{aligned}
\end{equation}

\end{lem}

\begin{proof}
Since $\widehat \varphi(z)$ is an entire function, the inversion of the Laplace transform \eqref{phiz} 
can be carried out by integration on the imaginary axis:
\begin{equation}
\label{phixa}
\begin{aligned}
\varphi(x) & =  -\frac 1{2\pi i}\lim_{R\to\infty}\int_{-iR}^{iR} 
\left( 
   \frac {z+\beta}{\Lambda(z)}
 \Phi_0(z)+ \frac {z+\beta}{\Lambda(z)} e^{-z}\Phi_1(-z)
 -\psi(0)
\right)e^{zx}dz \\
& =
-\frac 1{2\pi i}\lim_{R\to\infty}\int_{-iR}^{iR} \big(f_0(z) +  f_1(z)\big)dz
\end{aligned}
\end{equation}
where we defined 
$$
f_0(z) = e^{zx}\left((z+\beta)\frac {\Phi_0(z)}{\Lambda(z)}   -\psi(0)\right)\quad \text{and}\quad 
 f_1(z) = e^{(x-1)z} (z+\beta) \frac {\Phi_1(-z)}{\Lambda(z)}.
$$
Computing the contour integral as in the proof of Lemma 5.8 in \cite{ChK}, we get 
\begin{equation}\label{intff}
\begin{aligned}
\int_{-i\infty}^{i\infty} \big(f_1(z)+f_0(z)\big)dz =\, & 2\pi i \Big(\Res(f_0,z_0) +  \Res(f_0,-z_0)\Big) +\\
&
\int_0^\infty \big(f_1^+(t)-f_1^-(t)\big)dt
+\int_0^\infty \big(f_0^-(-t)-f_0^+(-t)\big)dt.
\end{aligned}
\end{equation}
Using the symmetries \eqref{conjp} and \eqref{absL} and the definition of $\theta(t)$, we have
\begin{align*}
&
f_1^+(t)-f_1^-(t)   
= -e^{(x-1)t} (t+\beta) \Phi_1(-t)\frac{2i \sin\theta(t)}{\gamma(t)} \\
&
f_0^-(-t)-f_0^+(-t) 
=-e^{-tx} (-t+\beta)\Phi_0(-t)\frac{2i \sin\theta(t)}{\gamma(t)}
\end{align*}
where $\gamma(t)=|\Lambda^+(t)|$, and hence 
\begin{align*}
\varphi(x) = &- \Res\big(f_0,z_0\big) -  \Res\big(f_0,-z_0\big) +\\
&
 \frac 1 \pi \int_0^\infty \frac{  \sin\theta(t)}{\gamma(t)}\left( e^{-(1-x)t} (t+\beta) \Phi_1(-t)-
 e^{-tx} (t-\beta)\Phi_0(-t) \right)dt.
\end{align*}
The residues can be readily computed:   
\begin{align*}
&
\Res\big(f_0,z_0\big)  = 
e^{i\nu x} (i\nu +\beta)\frac {\Phi_0(i\nu )}{\Lambda'(i\nu ) }  =
e^{i\nu x}\Phi_0(i\nu ) \frac  { \cos \frac{\pi } 2\alpha}{\pi   }\nu^{3-\alpha} \frac { 1-i(\beta/\nu) }{  
\frac{2}{(\beta/\nu)^2 +1}  -\alpha+1 } \\
&
\Res\big(f_0,-z_0\big)  =  e^{-i\nu x} (-i\nu +\beta)\frac {\Phi_0(-i\nu )}{\Lambda'(-i\nu)}=
e^{-i\nu x}\Phi_0(-i\nu )\frac  { \cos \frac{\pi } 2\alpha}{\pi   }\nu^{3-\alpha}\frac {1  +i(\beta/\nu)}{
\frac{2 }{(\beta/\nu)^2 +1}   -\alpha+1}
\end{align*}
and therefore 
$$
\Res\big(f_0,z_0\big)+\Res\big(f_0,-z_0\big) = 2\nu^{3-\alpha}\frac  { \cos \frac{\pi } 2\alpha}{\pi   }  
\Re\left\{e^{i\nu x}\Phi_0(i\nu )\frac { 1-i(\beta/\nu) }{  
\frac{2}{(\beta/\nu)^2 +1}  -\alpha+1 }\right\}. 
$$
Plugging this back we get \eqref{phifla}, where 
\begin{equation}
\label{gammabeta}
\gamma_\beta(u;\nu)=\nu^{1-\alpha}\frac{ \cos \frac{\pi } 2\alpha}{\pi   } \big|\Lambda^+(u\nu)\big|= \left| 
\frac{u^2-(\beta/\nu)^2 }{(\beta/\nu)^2 +1}
+
u^{\alpha-1} 
e^{ \frac{1-\alpha}{2}\pi i}\right|.
\end{equation}
Equations  \eqref{flaphi} are obtained in view of \eqref{c1c2}, \eqref{lambdanu} and \eqref{linsys}. 
\end{proof}

\subsubsection{Asymptotic analysis}
Lemma \ref{lem5.3} reduces the eigenproblem \eqref{OUeigHlarge} to integro--algebraic system of equations. 
The following lemma derives exact asymptotics for the algebraic part of its solution:

\begin{lem}\label{lem5.8}
The integro-algebraic system of Lemma \ref{lem5.3} has countably many solutions, which can be enumerated so that 
\begin{equation}
\label{nuas}
\nu_n = \pi \Big(n+\frac 1 2\Big) -\frac{1-\alpha}{4}\pi + \arcsin{\frac{b_\alpha}{\sqrt{1+b^2_\alpha}}} + n^{-1} r_n(\alpha), \quad n\to \infty
\end{equation}
where the residual $r_n(\alpha)$ is bounded, uniformly in $n\in \mathbb{N}$ and $\alpha \in [\alpha_0, 1]$ for any $\alpha_0\in (0,1]$.
\end{lem}

\begin{proof} 
The arguments parallel those of Lemma 5.9 in \cite{ChK}.
Plugging the estimates from Lemma \ref{lemest} and Lemma \ref{lemXbeta} and the estimate \eqref{bnu} into the definition 
\eqref{xieta}, we can write 
$$
\xi \overline{\eta} = 4 \frac{3-\alpha}{2}\sqrt{1+b_\alpha^2}\exp \left\{i\Big(\nu+ \frac{1-\alpha}{4}\pi-\pi+\arg\big\{i+b_\alpha\big\}\Big)\right\}
\big(1+R(\nu)\big)
$$
where $R(\nu)$ is a function satisfying the bound $|R(\nu)|\le C_1 \nu^{-1}$ with a constant $C_1$, depending only on $\alpha_0$. 
Hence the equation \eqref{Im} reads
\begin{equation}
\label{Imexplicit}
\nu+ \frac{1-\alpha}{4}\pi-\pi+\arg\big\{i+b_\alpha\big\}-\pi n+\arctan \frac{\Im\{R(\nu\}}{1+\Re\{R(\nu)\}}=0, \quad n\in \mathbb{Z}.
\end{equation}
This enumerates all possible solutions of the integro-algebraic system from Lemma \ref{lem5.3}. Obviously, $\nu$ is positive for all 
$n$ greater than some integer. Note that at this point it is not yet clear whether there is a solution for any such $n$. 

Existence of the unique solution can be argued for all $n$ large enough as follows. 
Recall that by Lemma \ref{lem5.5} the integral operator in the right hand side of equations \eqref{qp} is contracting in $L_2(0,1)$ 
for all $\nu$ large enough.  
A lengthy but otherwise direct calculation also shows that $|R'(\nu)|\le C_2\nu^{-1}$ with a constant $C_2$.
Hence for any sufficiently large $n$, the system consisting of the integral equations \eqref{qp} and the algebraic equation \eqref{Imexplicit}, has the unique solution, given by the fixed point iterations. 
The asymptotics \eqref{nuas} now follows from \eqref{Imexplicit}, since 
$\arg\{i+b_\alpha\}=\frac \pi 2 -\arcsin \frac{b_\alpha}{\sqrt{1+b_\alpha^2}}$.

\end{proof}

The corresponding asymptotic approximation of the eigenfunctions is obtained using Lemma \ref{lemeigf}:
\begin{lem} 
Under the enumeration, introduced by Lemma \ref{lem5.8}, the eigenfunctions admit the approximation:
\begin{multline}\label{phias}
\varphi_n(x) =   \sqrt{2}
\cos \Big(\nu_n x + \frac {1-\alpha}8 \pi + 
\frac \pi 2 -\arcsin \frac{b_\alpha}{\sqrt{1+b_\alpha^2}}
\Big)
\\
+
 \frac {\sqrt{3-\alpha}} \pi\int_0^\infty \rho_0(u)
 \Big(    -e^{-u\nu_n x}  \frac{u-b_\alpha}{\sqrt{1+b_\alpha^2}} -(-1)^ne^{-(1-x)u\nu_n} \Big)du + n^{-1}r_n(x),
\end{multline}
where the residual $r_n(x)$ is uniformly bounded in both $n\in \mathbb{N}$ and $x\in [0,1]$ and  
\begin{equation}\label{rhofOU}
\rho_0(u)=\frac{  \sin\theta_0(u)}{\gamma_0(u)}X_0(-u).
\end{equation}
Moreover,
\begin{equation}
\label{psias}
\varphi_n(1) \propto   
- (-1)^n \sqrt{3-\alpha}  \big(1+O(n^{-1})\big)
\quad \text{and}\quad 
\int_0^1 e^{\beta x}\varphi_n(x)dx  \propto   
- \sqrt{\frac{3-\alpha}{1+b_\alpha^2}} \nu_n^{-1}
\end{equation}
and 
\begin{equation}\label{Dima}
\int_0^1 \varphi_n(x)dx  \propto- \sqrt{\frac{3-\alpha}{1+b_\alpha^2}} \nu_n^{-1}.
\end{equation}
\end{lem}

\begin{proof}

Let 
$
\gamma_0(u):= \big|u +u^{\alpha-2} e^{ \frac{1-\alpha}{2}\pi i}\big|,
$
then by \eqref{gammabeta} 
$$
\big|\gamma_\beta(u;\nu)-u\gamma_0(u)\big|\le 2(\beta/\nu)^2(u^2+1).
$$
Along with \eqref{thetabeta}, formula \eqref{phifla} gives:
\begin{multline*}
\varphi_n(x) \propto  - \frac 2 {3-\alpha}\Re\Big\{e^{i\nu_n x}\Phi_0(i\nu_n )\Big\}\\
+
 \frac 1 \pi\int_0^\infty \frac{  \sin\theta_0(u)}{\gamma_0(u)}\left( e^{-(1-x)u\nu_n}   \Phi_1(-u\nu_n)-
 e^{-u\nu_n x}  \Phi_0(-u\nu_n) \right)du + n^{-1}r_n(x)
\end{multline*}
where the residual $r_n(x)$ is bounded in both $n\in \mathbb{N}$ and $x\in [0,1]$. Now plugging the estimates 
from Lemma \ref{lemest} and  Lemma \ref{lemXbeta} and \eqref{bnu}  into the expressions \eqref{Phinorm} and normalizing 
to the unit $L^2(0,1)$ norm, as in (5.52) in \cite{ChK},   gives \eqref{phias}.
Formulas \eqref{flaphi} give \eqref{psias} after the same normalizing.

Asymptotics \eqref{Dima} is obtained by integrating \eqref{phifla}:  
a direct calculation shows that 
$$
\int_0^1 \varphi_n(x)dx = C \nu_n^{-1}\big(1+O(\nu_n^{-1})\big)\quad n\to\infty,
$$
where $C\nu_n^{-1}$ coincides with the integral of the expression in \eqref{phias} without the residual. 
Since this expression does not depend on $\beta$, the constant factor $C$ must coincide with that obtained 
for $\beta = 0$. In other words, the sequence of integrals $\int_0^1 \varphi_n(x)dx$ for the fractional Ornstein--Uhlenbeck 
process and the f.B.m. has the exactly same leading order asymptotics. The exact constant in \eqref{Dima} can therefore be 
taken from (5.53) in \cite{ChK}. 

\end{proof}

\subsubsection{Enumeration alignment} 

The enumeration introduced in Lemma \ref{lem5.8} may not coincide with the {\em natural} enumeration, which puts the eigenvalues 
into decreasing order. Note that when the expression \eqref{nuas} is plugged into \eqref{lambdanu} the emerging sequence 
of $\lambda_n$'s in our enumeration is strictly decreasing; hence starting from some index it can differ from the natural enumeration 
only by a finite shift. To identify this shift we can use the calibration procedure, based on continuity of the 
spectrum with respect to $\alpha$ and the known asymptotics \eqref{OUeig} for the standard Ornstein--Uhlenbeck process,
corresponding to $\alpha=1$. The precise details are the same as in Section 5.1.7. in \cite{ChK} and  
the formulas \eqref{nuas} and \eqref{phias}-\eqref{psias} should be shifted by one: replacing $n$ with $n-1$ and 
$\alpha$ with $2-2H$ the expressions claimed in Theorem \ref{thm-fOU} are obtained. 
%
%
%
%

\subsection{The case $H<\frac 1 2$} In this case the covariance function has the form \eqref{OUKst} 
and the eigenproblem reads:
$$
\int_0^1 \left(
\int_0^x e^{\beta(x-u)}\frac d{du}\int_0^y e^{\beta(y-v)}C_\alpha|u-v|^{1-\alpha}\sign(u-v) dvdu
\right)\varphi(y)dy = \lambda \varphi(x),
$$
where $C_\alpha:=1-\frac \alpha 2$.  Taking the derivative of both sides gives 
$$
\int_0^1 \left(
  \frac d{dx}\int_0^y e^{\beta(y-v)}C_\alpha|x-v|^{1-\alpha}\sign(x-v) dv \right)\varphi(y)dy
  +
\beta \lambda \varphi(x) = \lambda \varphi'(x). 
$$
This can be rewritten as 
$$
-\frac d{dx}\int_0^1 \frac{d}{dy} \left(\int_y^1 e^{\beta r}\varphi(r)dr\right)\left(
  \int_0^y e^{-\beta v}C_\alpha|x-v|^{1-\alpha}\sign(x-v) dv \right)dy
  +
\beta \lambda \varphi(x) = \lambda \varphi'(x),
$$
and integrating by parts we get
$$
 \frac d{dx}\int_0^1  C_\alpha|x-y|^{1-\alpha}\sign(x-y) \psi(y) dy + \beta \lambda\varphi(x) = \lambda \varphi'(x),
$$
where $\psi(x)$ is defined as in \eqref{psidef}. Plugging in the identity \eqref{psiphi}
we obtain the generalized eigenproblem (cf. \eqref{geneig}):  
$$
\begin{aligned}
\frac d{dx}\int_0^1  &  C_\alpha|x-y|^{1-\alpha}\sign(x-y) \psi(y) dy    = 
 \lambda\Big(\beta^2 \psi(x)-\psi''(x)\Big) 
 , \quad x\in [0,1] \\
&
\psi(1) =  0, \; 
\psi'(0)    +\beta \psi(0)  =0
\end{aligned}
$$
From here on, the proof proceeds as in the case $H>\frac 12$.

\section{Proof of Theorem \ref{thm-ifBm}}\label{sec-proof-ifBm}

For the integrated f.B.m. with covariance function \eqref{KifBm}, eigenproblem \eqref{eigpr} reads
\begin{equation}\label{eigifBm}
\int_0^1 \left(\int_0^y\int_0^x K(u,v) dudv\right) \varphi(y)dy = \lambda \varphi(x)\quad x\in [0,1],
\end{equation}
where $K(u,v) :=\tfrac 1 2\left(u^{2-\alpha}+v^{2-\alpha}-|v-u|^{2-\alpha}\right)$ and  $\alpha := 2-2H\in (0,2)$.

\subsection{The case $H<\frac 1 2$}

\subsubsection{The Laplace transform} 

Our starting point is again a suitable expression for the Laplace transform:

\begin{lem}
Let $(\lambda,\varphi)$ be a solution of  \eqref{eigifBm}, then the Laplace transform of $\varphi$ satisfies 
\begin{equation}
\label{phizifbm2} 
\widehat\varphi(z)=  \widehat \varphi(0)+z\frac d {dz}\widehat \varphi(z)_{\big|z=0}-\frac {e^{-z}\Phi_1(-z)+\Phi_0(z)}{\Lambda(z)} 
\end{equation}
where functions $\Phi_0(z)$ and $\Phi_1(z)$
are sectionally holomorphic on the cut plane $\mathbb{C}\setminus \Real_{>0}$ and 
\begin{equation}\label{lambdaifbm2}
\Lambda(z) := \frac {\lambda\Gamma(\alpha-1)}{1-\frac{\alpha}{2}}z^2 - \int_0^\infty \frac{2 t^{\alpha-2 }}{z^2-t^2}.
\end{equation}
\end{lem}

\begin{proof}
Taking derivative of \eqref{eigifBm} and changing the order of integration we get 
\begin{equation}
\label{eq:7.4}
\int_0^1 K(x,y) \int_y^1 \varphi(u)du dy=\lambda \varphi^{\prime}(x), \quad x\in [0,1].
\end{equation}
Define 
$
\psi(x) = \displaystyle\int_x^1 \int_y^1 \varphi(u) du dy,
$
then integration by parts gives 
\begin{align*}
\int_0^1 K(x,y) \int_y^1 \varphi(u)du dy= & -\int_0^1 K(x,y) \psi^{\prime}(y)dy= \\
&
\int_0^1 \big(1-\tfrac{\alpha}{2}\big)(y^{1-\alpha}+\sign{(x-y)}\lvert x-y\rvert^{1-\alpha})\psi(y)dy.
\end{align*}
Since $\varphi(x)=\psi^{\prime\prime}(x)$ equation \eqref{eq:7.4} reads 
$$
 (1-\tfrac{\alpha}{2}) \int_0^1 \big(y^{1-\alpha}+\sign{(x-y)}\lvert x-y\rvert^{1-\alpha}\big)\psi(y)dy=\lambda \psi^{(3)}(x), \quad x\in [0,1],
$$
and, taking another derivative we arrive at the generalized eigenproblem
\begin{equation}\label{geneig_ifBm2}
\begin{aligned}
\big(1-\tfrac{\alpha}{2}\big) \frac{d}{dx} \int_0^1 \sign{(x-y)}\lvert x-y\rvert^{1-\alpha}\psi(y)dy=\lambda \psi^{(4)}(x), \quad x\in [0,1] \\
\psi(1)=0, \quad \psi'(1)=0, \quad \psi''(0)=0, \quad \psi^{(3)}(0)=0.
\end{aligned}
\end{equation}

Using the identity \eqref{Gfla}, with $\alpha$ replaced by $\alpha-1\in (0,1)$, the expression in the left hand side can be rewritten 
in the form: 
\begin{multline*}
\frac{d}{dx} \int_0^1 \sign{(x-y)}\lvert x-y\rvert^{1-\alpha}\psi(y)dy = \\ 
\frac{1}{\Gamma(\alpha-1)} \frac{d}{dx} \int_0^\infty t^{\alpha-2}\left(\int_0^1 \sign{(x-y)} e^{-t|x-y|}\psi(y)dy\right) dt
\end{multline*}
If we now define 
\begin{equation}\label{uudef}
u(x,t) := \int_0^1 \sign{(x-y)} e^{-t|x-y|} \psi(y) dy\quad \text{and}\quad  u_0(x) := \int_0^{\infty} t^{\alpha-2} u(x,t) dt,
\end{equation}
equation \eqref{geneig_ifBm2} becomes
\begin{equation} \label{u0psi}
u'_0(x)=\frac{\lambda \Gamma(\alpha-1)}{1-\frac{\alpha}{2}}\psi^{(4)}(x), \quad x\in [0,1].
\end{equation}
Taking the Laplace transform and plugging the boundary conditions from \eqref{geneig_ifBm2}
we get 
\begin{equation}\label{hatutag}
\begin{aligned}
\widehat u'_0(z) = \,
&
\frac{\lambda \Gamma(\alpha-1)}{1-\frac{\alpha}{2}} \int_0^1 e^{-zx} \psi^{(4)}(x) dx=\\
&
\frac{\lambda \Gamma(\alpha-1)}{1-\frac{\alpha}{2}}
\Big( e^{-z}\psi^{(3)}(1)+z e^{-z}\psi''(1)-z^2 \psi'(0)-z^3 \psi(0) +z^4 \widehat \psi(z) \Big).
\end{aligned}
\end{equation}
Another expression for $\widehat u'_0(z)$ can be obtained, using definitions \eqref{uudef} directly. Taking two derivatives of the expression 
for $u(x,t)$ gives the equation 
\begin{equation}\label{uxttt}
u''(x,t) = 2\psi'(x)+t^2 u(x,t)
\end{equation}
with the boundary conditions 
\begin{align*}
u'(0,t) = & \phantom{+} t u(0,t) + 2\psi(0) \\
u'(1,t) = & - t u(1,t).
\end{align*}
Integrating by parts twice and plugging these conditions yields 
\begin{align*}
\widehat u''(z,t) =\, &  \int_0^1 u''(x,t) e^{-zx}dx=u'(1,t) e^{-z} - u'(0,t) + z u(1,t)e^{-z} -z u(0,t) +z^2 \widehat u(z,t) = \\
&
\big(z- t\big) u(1,t) e^{-z} -  \big(z+t\big) u(0,t) - 2\psi(0)    +z^2 \widehat u(z,t).
\end{align*}
Combining this with the Laplace transform of \eqref{uxttt} gives 
$$
\widehat u(z,t)= \frac {2z}{ z^2  -t^2} \widehat \psi (z)   -\frac 1{z+t} u(1,t) e^{-z} +  \frac 1 {z-t} u(0,t).
$$
Multiplying by $t^{\alpha-2}$ and integrating we obtain 
$$
\widehat u_0(z) =  
  \int_0^\infty \frac {t^{\alpha-2}} {z-t} u(0,t)dt-e^{-z} \int_0^\infty \frac {t^{\alpha-2}}{z+t} u(1,t) dt  
+ z\widehat \psi (z) \int_0^\infty\frac {2t^{2-\alpha}}{ z^2  -t^2}dt,
$$
and, since $\widehat u'_0(z)=u_0(1)e^{-z} -u_0(0)+z \widehat u_0(z)$, 
$$
\widehat u'_0(z) = \int_0^\infty \frac{t^{\alpha-1}}{z-t}u(0,t)dt+e^{-z} \int_0^\infty \frac{t^{\alpha-1}}{z+t}u(1,t)dt+z^2\widehat \psi(z) \int_0^\infty \frac{2t^{\alpha-2}}{z^2-t^2}dt. 
$$
Combining this with \eqref{hatutag} and rearranging we get 
$$
z^2 \widehat \psi(z)=-\frac {e^{-z}\Phi_1(-z)+\Phi_0(z)}{\Lambda(z)},
$$  
where $\Lambda(z)$ is given in \eqref{lambdaifbm2} and 
\begin{equation}\label{phi01ifbm2}
\begin{aligned}
\Phi_0(z) &:= -\frac {\lambda\Gamma(\alpha-1)}{1-\frac{\alpha}{2}} \psi'(0)z^2 - \frac {\lambda\Gamma(\alpha-1)}{1-\frac{\alpha}{2}} \psi(0)z^3+\int_0^\infty \frac{t^{\alpha-1}}{t-z}u(0,t)dt\\
\Phi_1(z) &:= \frac {\lambda\Gamma(\alpha-1)}{1-\frac{\alpha}{2}} \psi^{(3)}(1) - \frac {\lambda\Gamma(\alpha-1)}{1-\frac{\alpha}{2}} \psi''(1)z-\int_0^\infty \frac{t^{\alpha-1}}{t-z}u(1,t)dt.
\end{aligned}
\end{equation}
The expression \eqref{phizifbm2} follows since 
$
\widehat \varphi(z) = \widehat \psi''(z)= -\psi'(0)-\psi(0)z + z^2 \widehat \psi(z).
$

\end{proof}

\pagebreak[3]

The next lemma  details the structure of $\Lambda(z)$:

\begin{lem}  \label{lem5.2}
\

\medskip
\noindent
a) Function $\Lambda(z)$ admits the expression  
\begin{equation}\label{Lambdazz}
\Lambda (z) = \frac {\lambda\Gamma(\alpha)}{|c_\alpha|} z^2 - \frac{\pi   }{ |\cos \frac{\pi } 2\alpha|} z^{\alpha-3} 
\begin{cases}
e^{\frac{1-\alpha}{2}\pi i}  & \arg (z) \in (0,\pi) \\
e^{-\frac{1-\alpha}{2}\pi i }  &   \arg(z)\in (-\pi, 0)
\end{cases}
\end{equation}
where $c_\alpha=(1-\frac{\alpha}{2})(1-\alpha)$. It has two zeros at $\pm z_0 = \pm \nu i$ with $\nu\in \Real_{>0}$ given by
\begin{equation}\label{lambdanuifbm2}
\nu^{\alpha-5} =   \frac{\lambda\Gamma(\alpha)}{|c_\alpha|}\frac{|\cos \frac{\pi } 2\alpha|}{\pi}.
\end{equation}

\medskip 
\noindent
b) The limits of $\Lambda(z)$ across the real axis are given by 
$$
\Lambda^\pm (t) =
 \frac {\lambda\Gamma(\alpha)}{|c_\alpha|} t^2
\mp
|t|^{\alpha-3} 
\frac{\pi   }{ |\cos \frac{\pi } 2\alpha|}
\begin{cases}
e^{\frac{1\mp\alpha}{2}\pi i} & \quad t>0 \\
e^{-\frac{1\mp\alpha}{2}\pi i } & \quad t<0
\end{cases}
$$
and satisfy symmetries \eqref{conjp}-\eqref{absL}. 

\medskip 
\noindent 
c)  The argument $\theta(t):=\arg\{\Lambda^+(t)\}\in (-\pi, \pi]$ is an odd function $\theta(-t)=-\theta(t)$
$$
\theta(t)= \arctan \frac
{
 -\sin \frac{1-\alpha}{2}\pi  
}
{
(t/\nu)^{5-\alpha} -    \cos \frac{1-\alpha}{2}\pi   
},\qquad t>0
$$
decreasing continuously from $\theta(0+) := \frac{3-\alpha}{2}\pi$ to $0$ as $t\to\infty$. Function $\theta_0(u):=\theta(u\nu)$ satisfies 
\begin{equation}\label{bifbm2}
\begin{aligned}
b_0 &:=\frac{1}{\pi}\int_0^\infty \theta_0(t) dt = \frac{\sin\big(\frac{\pi}{2}\frac{3-\alpha }{5-\alpha}\big)}{\sin\frac{\pi}{5-\alpha}}\\
b_1 &:=\frac{1}{\pi}\int_0^\infty t \theta_0(t)dt =\frac{1}{2}\\ 
b_2 &:=\frac{1}{\pi}\int_0^\infty t^2 \theta_0(t) dt = \frac{1}{3} \frac{\sin\big(\frac{3\pi}{2}\frac{3-\alpha}{5-\alpha}\big)}{\sin\frac{3\pi}{5-\alpha}}\\
\end{aligned}
\end{equation}
and consequently 
\begin{equation}\label{inttheta}
\frac{1}{\pi} \int_0^\infty \frac{\theta(s)}{s-z}ds = -\frac{\nu b_0}{z}-\frac{\nu^2 b_1}{z^2}-\frac{\nu^3 b_2}{z^3} + O(z^{\alpha-5}), \quad z \to\infty, \quad z \in \mathbb{C}\setminus \Real_{>0}.
\end{equation}
\end{lem}

\begin{proof} 
The proof is similar to that of Lemma \ref{lem4.2}: identity \eqref{intfla}
with $\alpha$ replaced by $\alpha-2$ gives expression \eqref{Lambdazz} and, in turn, the formulas in b) and c).  
By a change of variables the integrals in \eqref{bifbm2} reduce to 
$$
\frac{1}{\pi}\int_0^\infty u^k \theta_0(u)du=
\frac{\big|\cos \frac{\pi}{2}\alpha \big|^{\frac{k+1}{5-\alpha}}}{\pi (k+1)} \int_0^\infty \frac{s^{\frac{k+1}{5-\alpha}}}{1+\big(\cot \frac{3-\alpha}{2}\pi +s\big)^2} ds
$$
and the formulas claimed in \eqref{bifbm2} are obtained by appropriate contour integration. 
Asymptotics \eqref{inttheta} holds by virtue of the elementary formula 
$$
\frac 1 {s-z} = -\frac 1 z -\frac s {z^2} -\frac {s^2} {z^3} + \frac {s^3}{z^3}\frac 1 {s-z},
$$
since $\theta(t) \sim t^{\alpha-5}$ as $t\to\infty$ and 
$$
\int_0^\infty t^k \theta(t)dt = \nu^{k+1}\int_0^\infty  u^k \theta_0(u)du, \quad k=0,1,2.
$$

\end{proof}

\subsubsection{Removal of singularities}

As in Section \ref{sec512}, removal of singularities in \eqref{phizifbm2} imposes the boundary condition 
$$
\begin{aligned}
&
\Phi_0^+(t) - e^{2i\theta(t)}\Phi_0^-(t) = 2i  e^{-t} e^{i\theta(t)}\sin\theta(t) \Phi_1(-t)
\\
&
\Phi_1^+(t) - e^{2i\theta(t)}\Phi_1^-(t) = 2i e^{-t}  e^{i\theta(t)}\sin\theta(t) \Phi_0(-t)
\end{aligned}\qquad t>0.
$$
 and 
\begin{equation}\label{algalgc}
 \Phi_0(\pm z_0)+e^{\mp z_0}\Phi_1(\mp z_0) = 0.
\end{equation}
It follows from \eqref{phi01ifbm2} that 
\begin{equation}
\label{phi01growthifbm2}
\begin{aligned}
\Phi_0(z) &=  2k_2 z^2 + 2k_3 z^3 + O(z^{\alpha-2}) \\
\Phi_1(z) &= 2k_0 + 2k_1 z + O(z^{\alpha-2}) 
\end{aligned}
\qquad\text{as\ } z\to \infty ,
\end{equation}
where we defined the constants
\begin{equation}\label{k1234}
\begin{aligned}
k_0 & = \phantom{+}\frac 1 2\frac {\lambda\Gamma(\alpha)}{ |c_\alpha|}\psi^{(3)}(1), & k_1 & =- \frac 1 2\frac {\lambda\Gamma(\alpha)}{ |c_\alpha|} \psi''(1)  \\
k_2  & = - \frac 1 2\frac {\lambda\Gamma(\alpha)}{ |c_\alpha|} \psi'(0),   & k_3 &=- \frac 1 2\frac {\lambda\Gamma(\alpha)}{ |c_\alpha|}  \psi(0).
\end{aligned}
\end{equation}
Integrating \eqref{u0psi} we get $2k_0=u_0(1)-u_0(0)$ and  hence \eqref{phi01ifbm2} read
\begin{align*}
\Phi_0(z) &  = \phantom{+}u_0(0)+2k_2 z^2 + 2 k_3 z^3 + z \int_0^\infty \frac{t^{\alpha-2}}{t-z} u(0,t) dt\\
\Phi_1(z) &  = -u_0(0)+2k_1 z - z \int_0^\infty \frac{t^{\alpha-2}}{t-z} u(1,t) dt
\end{align*}
Consequently,   
\begin{equation}\label{Phi_near_origin}
 \Phi_0(z)= u_0(0) + O(z^{\alpha-1}) \quad \text{and}  \quad  \Phi_1(z) = -u_0(0) + O(z^{\alpha-1}) \quad \text{as}\  z \to 0. 
\end{equation}

\subsubsection{An equivalent formulation of the eigenproblem}
The suitable solution of the homogeneous Riemann boundary value problem 
$$
X^+(t) - e^{2i\theta(t)}X^-(t) =0, \quad t\in \Real_{>0}.
$$ 
in this case has the form   
\begin{equation}\label{Xz_ifBm_smallH}
X(z)= z X_c(z)=z\exp \left(\frac 1 \pi \int_0^\infty\frac{\theta(t)}{t-z}dt\right), \quad z\in \mathbb{C}\setminus \Real_{>0}.
\end{equation}
Factor $z$ in front of the exponential is fixed here to guarantee that functions  
\begin{equation}
\label{SDdef-ifBm}
\begin{aligned}
& S(z):= \frac{\Phi_0(z)+\Phi_1(z)}{2X(z)} \\
& D(z):= \frac{\Phi_0(z)-\Phi_1(z)}{2X(z)}
\end{aligned}
\end{equation}
are square integrable at the origin, when restricted to the negative real semiaxis. 
This is indeed the case in view of a priori estimates \eqref{Phi_near_origin} and asymptotics \eqref{Xz_near_zero} of $X(z)$
at the origin, derived along with other useful properties in the following lemma:
 
\begin{lem} 
Function $X(z)$ defined in \eqref{Xz_ifBm_smallH} satisfies 
\begin{equation}\label{Xz_near_zero}
X(z) \sim z^{\frac{\alpha-1}{2}} \quad \text{as}\ z\to 0
\end{equation}
and 
\begin{equation}\label{Xzgrowth}
X(z) \simeq z - \nu b_0 +  \Big( \frac{1}{2} b_0^2-b_1 \Big)   \frac{\nu^2}{z}
-  \Big( \frac{1}{6} b_0^3-b_0b_1 + b_2\Big)\frac{\nu^3}{z^2}  \quad \text{as}\ z \to \infty,
\end{equation}
where $b_{j,\alpha}$ are the constants defined in \eqref{bifbm2}. 
Moreover, $X_0(z):=X_c(\nu z)$ satisfies 
\begin{equation}\label{X0ifBm2}
\arg\big\{X_0(i)\big\} = \frac {3-\alpha}8\pi\quad \text{and}\quad |X_0(i)| = \sqrt{\frac{5-\alpha}{2}}.
\end{equation}

\end{lem}

\begin{proof}
The growth estimate \eqref{Xz_near_zero} holds since  
$$
X_c (z) = \exp \left(\frac 1 \pi \int_0^\infty\frac{\theta(t)}{t-z}dt\right) \sim z^{-\theta(0+)/\pi}  \quad\text{as\ } z\to 0
$$
and $\theta(0+)= \frac {3-\alpha}{2}\pi$ (see Lemma \ref{lem5.2} (c)).  
The asymptotics at infinity is obtained from \eqref{inttheta} and the Taylor expansion of exponential. 
Formulas \eqref{X0ifBm2} follow from the identities    
$$
\arg\{X_c(i\nu)\} = \frac {\theta_0(0+)} 4
\quad
\text{and}
\quad 
|X_c(i\nu)|^2 =  \frac{|c_\alpha|}{\lambda \Gamma(\alpha)} \lim_{z\to z_0} 
\frac{z^4 \Lambda(z)}{ z^2-z_0^2  }
$$
proved in Lemma 5.5 \cite{ChK}.  

\end{proof}

Due to a priori estimates \eqref{phi01growthifbm2} and \eqref{Xzgrowth}, functions $S(z)$ and $D(z)$ 
satisfy equations (c.f. \eqref{SD})
$$
\begin{aligned}
& 
S(z) = \phantom{+}\frac 1 \pi \int_0^\infty \frac{h(t)e^{-t}}{t-z}S(-t) dt + P_S(z) \\
&
D(z) = -\frac 1 \pi \int_0^\infty \frac{h(t)e^{-t}}{t-z}D(-t)dt + P_D(z)
\end{aligned}
$$
with $h(t)$ defined as in \eqref{hdefine} and polynomials, whose degrees do not exceed 2
$$ 
P_S(z) = l_0 + l_1 z +l_2 z^2  \quad \text{and} \quad
P_D(z) = m_0 + m_1 z +m_2 z^2.
$$
By definition \eqref{SDdef-ifBm} and estimate \eqref{phi01growthifbm2}  
\begin{equation}\label{matchme1}
S(z) X(z) =   \frac 1 2 \Big(\Phi_0(z) +   \Phi_1(z)\Big) = k_0 +  k_1 z + k_2 z^2 +  k_3 z^3  + o(1), \quad z\to\infty.
\end{equation}
On the other hand,  \eqref{SD} implies 
\begin{equation}\label{matchme2}
S(z) X(z) = \Big(P_S(z) + k_S z^{-1} + o(z^{-1})\Big)X(z), \quad z\to\infty,
\end{equation}
where we defined 
\begin{equation}\label{kSdef}
k_S  :=  \lim_{z\to\infty} \frac{z}{\pi}\int_0^\infty \frac{h(t)e^{-t}}{t-z}S(-t)dt.
\end{equation}
Plugging expansion \eqref{Xzgrowth} into \eqref{matchme2} and matching the powers with \eqref{matchme1}, 
we obtain the relations 
\begin{equation}\label{lm123a}
\begin{aligned}
&
l_2 =k_3 
\\
&
l_1  =k_2+\nu k_3 b_0  
\\
&
l_0  =k_1+\nu k_2 b_0 +\nu^2 k_3 \sigma_1 
\\
& k_S 
=   k_0+\nu k_1 b_0 +\nu^2 k_2 \sigma_1 +\nu^3 k_3\sigma_2,
\end{aligned}
\end{equation}
where we defined  
$$
\sigma_1 = \frac{1}{2}  b_0^2 + b_1\quad \text{and}\quad 
\sigma_2 = \frac{1}{6}  b_0^3 + b_0 b_1 + b_2.
$$
Analogous calculations for $D(z)X(z)$ give 
\begin{equation}\label{lm123b}
\begin{aligned} 
& 
m_2   = k_3 \\ 
&
m_1   = k_2 + \nu k_3   b_0
\\
&
m_0   =  -  k_1+ k_2 \nu b_0+ \nu^2 k_3  \sigma_1 
\\
&
k_D    = - k_0 -  k_1\nu b_0  + k_2 \nu^2\sigma_1 +k_3 \nu^3\sigma_2,
\end{aligned}
\end{equation}
with the constant   
\begin{equation}\label{kDdef}
k_D   :=-\lim_{z\to\infty} \frac{z}{\pi}\int_0^\infty \frac{h(t)e^{-t}}{t-z}D(-t)dt.
\end{equation}

Consider now the integral equations 
\begin{equation}\label{p_pm_j}
p^\pm_j (t) = \pm \frac 1 \pi \int_0^\infty \frac{h_0(s)e^{-\nu s}}{s+t} p^\pm_j(s)ds+t^j, \quad t>0, \quad j\in \{0,1,2\},
\end{equation}
where $h_0(s):= h(s\nu)$ with $h(s)$ being defined as in \eqref{hdefine}.
As in Lemma \ref{lem5.5} the operator in the right hand side is contracting on $L^2(0,\infty)$ for all $\nu$ 
large enough. Consequently these equations have unique solutions, such that functions $p^\pm_j(t)-t^j$ belong to $L^2(0,\infty)$.
Since $S(-t)$ and $D(-t)$ are a priori square integrable at the origin, by linearity 
\begin{align}
\nonumber 
S(z\nu) =\, &  l_0 p^+_0(-z) - l_1 \nu p^+_1(-z) + l_2 \nu^2 p^+_2(-z)=\\
&
\nonumber
\Big(k_1+\nu k_2 b_0 +\nu^2  k_3 \sigma_1\Big) p^+_0(-z) 
-  \Big(\nu k_2+\nu^2 k_3 b_0 \Big) p^+_1(-z) +  \nu^2 k_3 p^+_2(-z) \\
\label{Sifbm2}
D( z\nu) =\, &
m_0 p^-_0(-z) - m_1\nu p^-_1(-z) +m_2 \nu^2p^-_2(-z) =\\
&
\nonumber
\Big(
-  k_1+ \nu  k_2 b_0+ \nu^2 k_3  \sigma_1
\Big) p^-_0(-z) - \Big(\nu k_2 + \nu^2 k_3   b_0\Big) p^-_1(-z) +\nu^2k_3 p^-_2(-z)
\end{align}
where we substituted \eqref{lm123a}-\eqref{lm123b} and extended the domain of $p^\pm_j(z)$ to the cut plane 
by replacing $t$ with $z\in \mathbb{C}\setminus \Real_{<0}$ in \eqref{p_pm_j}. 
Combining the definitions of $S(z)$ and $D(z)$  with \eqref{Sifbm2} we get  
\begin{equation}\label{phi01ifbm22}
\begin{aligned}
\Phi_0(\nu z) /X(\nu z) &= k_1\xi_1(-z) +\nu k_2 \xi_2(-z) +\nu^2 k_3 \xi_3(-z)  \\
\Phi_1(\nu z)/X(\nu z)  & =  k_1\eta_1(-z) +\nu k_2\eta_2(-z) +\nu^2k_3 \eta_3(-z)
\end{aligned}
\end{equation}
where   $a^\pm_j(z) := p^+_j(z)-p^-_j(z)$ and 
\begin{equation}\label{xieta_ifBm}
\begin{aligned}
\xi_1(z) & := a^-_0(z) & \eta_1(z) & :=  a^+_0(z) \\
\xi_2(z) & :=   b_0  a^+_0(z) -    a^+_1(z) & \eta_2(z) & :=   b_0 a^-_0(z) -   a^-_1(z) \\
\xi_3(z) & := \sigma_1  a^+_0(z) -  b_0 a^+_1(z)+  a^+_2(z)
&
\eta_3(z) & := \sigma_1 a^-_0(z) -  b_0 a^-_1(z)+  a^-_2(z).
\end{aligned}
\end{equation}
Now plugging these expressions into \eqref{algalgc} we obtain 
\begin{equation}
\label{k123}
  k_1\gamma_1+\nu  k_2 \gamma_2+\nu^2 k_3 \gamma_3  = 0
\end{equation}
with  
\begin{equation}\label{gammaphi}
\gamma_j   =  \eta_j(i)+e^{\phi_\nu i}\xi_j(-i)\quad \text{and} \quad \phi_\nu  = \nu  + 2\arg\{X(i\nu)\}.
\end{equation} 
Since $k_j$'s are real and $\gamma_j$'s 
have nontrivial imaginary parts, \eqref{k123} furnishes two equations with real coefficients. 
An additional third equation can be obtained as follows. By definitions \eqref{kSdef} and \eqref{kDdef}, we have  
\begin{equation}\label{kSkD}
\begin{aligned}
k_S = & - \nu l_0 c_{S,0} + \nu^2 l_1 c_{S,1} - l_2 \nu^3 c_{S,2} \\
k_D = & -\nu m_0 c_{D,0}+\nu^2  m_1c_{D,1}-\nu^3 m_2 c_{D,2}
\end{aligned}
\end{equation}
where   
\begin{align*}
c_{S,j} & := - \lim_{z\to\infty} \frac{z}{\pi}\int_0^\infty \frac{h_0(t)e^{-\nu t}}{t-z} p^+_j (t)dt=
 \frac{1}{\pi}\int_0^\infty  h_0(t)e^{-\nu t}  p^+_j (t)dt \\
c_{D,j} & := \lim_{z\to\infty} \frac{z}{\pi}\int_0^\infty \frac{h_0(t)e^{-\nu t}}{t-z}p^-_j(t)dt = 
 -\frac{1}{\pi}\int_0^\infty  h_0(t)e^{-\nu t}  p^-_j(t)dt.
\end{align*}
Plugging \eqref{kSkD} into the equations \eqref{lm123a} and   \eqref{lm123b} we obtain 
$$
\begin{aligned}
k_0 +\nu k_1 \big(b_0 +   c_{S,0}\big)  +\nu^2 k_2  &\big(\sigma_1 +    b_0 c_{S,0} -     c_{S,1}  \big)
+ \\
&
\nu^3 k_3\big(\sigma_2+    \sigma_1 c_{S,0}  -  b_0   c_{S,1} +   c_{S,2} \big)=0 \\  
k_0 +  \nu k_1 \big(b_0  + c_{D,0} \big)   +\nu^2 k_2   &\big(  - \sigma_1 -  b_0c_{D,0}+  c_{D,1} \big)
 + \\
&
\nu^3  k_3 \big( - \sigma_2 -   \sigma_1 c_{D,0}   +   b_0c_{D,1} - c_{D,2} \big) = 0
\end{aligned}
$$
which upon subtraction yield 
\begin{multline}\label{addeq}
 k_1 \big( c_{S,0} - c_{D,0} \big) 
+\nu k_2   \big(2\sigma_1 +    b_0 (c_{S,0}+  c_{D,0}) -     (c_{S,1}    +  c_{D,1})\big)+  \\
\nu^2 k_3\big(2\sigma_2+    \sigma_1 (c_{S,0} +   c_{D,0})
 -  b_0   (c_{S,1}+c_{D,1}) +   (c_{S,2}    + c_{D,2}) \big)  
=0 
\end{multline}
Thus  \eqref{k123} and \eqref{addeq} form a system of three linear equations for $k_1,\nu k_2, \nu^2 k_3$, whose coefficients are real 
valued and are functions of $\nu$. Letting $M(\nu)$ be the matrix of coefficients, this system admits a nontrivial solution if and only if $\nu$ 
satisfies the algebraic equation 
\begin{equation}
\label{detMeq}
\det\{M(\nu)\}=0.
\end{equation}

In summary, we arrive at the following equivalent formulation of the eigenproblem:

\begin{lem}\label{lem6.11}
Let $(p^\pm_0, p^\pm_1, p^\pm_2, \nu)$ with $\nu>0$ be a solution of the system, which consists of the integral equations
\eqref{p_pm_j}  and the algebraic equations \eqref{detMeq}. Let $\varphi$ be defined by the Laplace transform, given by the 
formula \eqref{phizifbm2}, where $\Phi_0(z)$ and $\Phi_1(z)$ are given by \eqref{phi01ifbm22} and let $\lambda$ be defined 
by \eqref{lambdanuifbm2}. Then the pair $(\lambda, \varphi)$ solves the eigenproblem \eqref{eigifBm}. Conversely, any 
solution $(\lambda, \varphi)$ of  \eqref{eigifBm} defines a solution to the above integro-algebraic system. 
\end{lem}

\subsubsection{Properties of the integro-algebraic system} 
As mentioned above, equations \eqref{p_pm_j} have unique solutions, such that $p^\pm_j(t)-t^j$ belong to $L^2(0,\infty)$. 
Asymptotic analysis of the integro-algebraic system of Lemma \ref{lem6.11} is based on the following estimates, derived as in Lemma 5.7 \cite{ChK}:

\begin{lem}\label{lem6.12}
For any $\alpha_0 \in (1,2)$ there exist constants $\nu'$ and $C$, such that for all $\nu \geq \nu'$, $\alpha \in [1,\alpha_0]$ 
the following estimates hold:
\begin{align*}
& \big|p^\pm_j(i)-i^j\big| \leq C \nu^{-(j+1)}\\
& \big|p^\pm_j(\tau)-\tau^j\big| \leq C \nu^{-(j+1)} \tau^{-1},\quad \tau>0\\
& |c_{S,j}|\vee |c_{D,j}| \le C \nu^{-(j+1)}.
\end{align*}
\end{lem}
 
\subsubsection{Inversion of the Laplace transform}

The eigenfunctions are recovered from the solution of the integro-algebraic system by inversion of the Laplace transform:  

\begin{lem} 
Let $(\Phi_0,\Phi_1,\nu)$ satisfy the integro-algebraic system introduced in Lemma \ref{lem6.11}, 
then the pair $(\lambda, \varphi)$ with $\lambda$ defined by the formula \eqref{lambdanuifbm2} and the function 
\begin{equation}\label{eigfunifBm}
\begin{aligned}
\varphi(x) =  & 
-
\frac 1 \nu \frac{|c_\alpha|}{\lambda \Gamma(\alpha)} \frac 2 {5   -   \alpha}
\Re\left\{ e^{i\nu x} \Phi_0(i\nu)   \frac { 1}{ i }  
\right\}
+ \\
& \phantom{+} \frac 1 \nu \frac {|c_\alpha|}{\lambda\Gamma(\alpha)} \frac  1  \pi \int_{0}^\infty \frac{ \sin \theta_0(t)}{ \gamma_0(t) }\left( e^{-t\nu (1-x)}\Phi_1(-t\nu ) +e^{-t\nu x}  \Phi_0(-t\nu )\right) dt 
\end{aligned}
\end{equation}
where $\gamma_0(u)=\big|u^2-u^{\alpha-3}e^{\frac{1-\alpha}{2}\pi i}\big|$, solves the eigenproblem \eqref{eigifBm} with $\alpha \in (1,2)$.
\end{lem}

\begin{proof}
As in the proof of Lemma \ref{lemeigf} eigenfunction $\varphi(x)$ satisfies \eqref{phixa}, this time 
with 
$$
f_1(z):= e^{z(x-1)}\frac { \Phi_1(-z)}{\Lambda(z)} \quad \text{and} \quad 
f_0(z):= e^{zx}\left(
\psi^{\prime}(0)+\psi(0)z+\frac { \Phi_0(z)}{\Lambda(z)}
\right).
$$
Equation \eqref{intff} holds with 
\begin{align*}
f_1^+(t) -f_1^-(t) 
&
= -e^{-t(1-x)}\Phi_1(-t) \frac{2i\sin \theta(t)}{\gamma(t)} \\
f_0^-(-t)-f_0^+(-t)
&
=- e^{-tx}  \Phi_0(-t)\frac{2i \sin \theta(t)}{\gamma(t)}
\end{align*}
and $\gamma(t)=|\Lambda^+(t)|$.
The residues are computed, using the explicit expression for $\Lambda(z)$ from Lemma \ref{lem5.2}\,(a):
\begin{align*}
\Res\big(f_0, z_0\big)  
=\, & e^{i\nu x} \Phi_0(i\nu) \frac {|c_\alpha|}{\lambda\Gamma(\alpha)} \frac { 1}{\nu i } \frac 1 {5-\alpha}
\\
\Res\big(f_0, -z_0\big) 
=\, 
& 
  e^{-i\nu x} \Phi_0(-i\nu ) \frac{|c_\alpha|}{\lambda \Gamma(\alpha)} \frac {1 }{ -\nu i}\frac 1 {5 -\alpha}
\end{align*}
and hence
$$
\Res\big(f_0, z_0\big)+ \Res\big(f_0, -z_0\big)  =
\frac 2 \nu \frac 1 {5   -   \alpha}\frac{|c_\alpha|}{\lambda \Gamma(\alpha)}\Re\left\{ e^{i\nu x} \Phi_0(i\nu)   \frac { 1}{ i }  
\right\}.
$$
Assembling all parts together we obtain formula \eqref{eigfunifBm}. 
\end{proof}

\subsubsection{Asymptotic analysis}

The following lemma determines asymptotics of the algebraic part of solutions to the system from Lemma \ref{lem6.11}:

\begin{lem}\label{lem7.7} 
The integro-algebraic system, introduced in Lemma \ref{lem6.11}, has countably many solutions, which can be enumerated so that 
\begin{equation}\label{asym2}
\nu_n=\pi (n-1) -\frac{3-\alpha}{4}\pi+\arctan  \Delta_\alpha + r_n(\alpha)n^{-1}  \quad \text{as}\ n\to\infty
\end{equation}
where 
\begin{equation}\label{Delta_agr1}
\Delta_\alpha=\frac{\frac{1}{3} b_0^3 - b_2}{\frac{1}{4}+\frac{1}{2}b_0^2+b_2b_0-\frac{1}{12}b_0^4}
\end{equation}
and the residual $r_n(\alpha)$ is bounded uniformly in $n$ and $\alpha \in [1,\alpha_0]$ for any $\alpha_0 \in (1,2)$.
\end{lem}

\begin{proof}
Definition \eqref{xieta_ifBm} and the estimates from Lemma \ref{lem6.12} imply  
\begin{align*}
\xi_1(i) & \simeq 0 & \eta_1(i) & \simeq   2 \\
\xi_2(i) & \simeq   2 b_0    -    2i & \eta_2(i) & \simeq   0\\
\xi_3(i) & \simeq \sigma_1  2 -  b_0 2i   -2
&
\eta_3(i) & \simeq 0
\end{align*}
where $\simeq$ stands for equality up to $O(\nu^{-1})$ residual, uniform with respect to $\alpha\in [1,\alpha_0]$. 
Hence $\gamma_j$'s from \eqref{gammaphi} satisfy 
\begin{align*}
\gamma_1   &\simeq   2  \\
\gamma_2   &\simeq  2 e^{\phi_\nu i}  ( b_0    +     i) \\
\gamma_3   &\simeq   2 e^{\phi_\nu i}(\sigma_1  -1  +  b_0  i   )  
\end{align*}
where $\phi_\nu$ is defined in \eqref{gammaphi} and, in view of \eqref{addeq} and \eqref{k123} the  matrix in \eqref{detMeq} 
satisfies   
\begin{equation}\label{Mnu2}
M(\nu) \simeq 
2\begin{pmatrix}
1  &  b_0 \cos \phi_\nu - \sin\phi_\nu &  (\sigma_1  -1) \cos \phi_\nu   -   b_0  \sin\phi_\nu
\\
0 & \cos \phi_\nu + b_0 \sin\phi_\nu  &     b_0 \cos \phi_\nu  +(\sigma_1  -1)\sin\phi_\nu    \\
0 &   \sigma_1  &  \sigma_2
\end{pmatrix}
\end{equation}
Consequently  
$$
\det\{M(\nu)\} \simeq \big(\cos \phi_\nu + b_0 \sin\phi_\nu\big)\sigma_2-
\big( b_0 \cos \phi_\nu  +(\sigma_1  -1)\sin\phi_\nu\big) \sigma_1.
$$
Hence the root of \eqref{detMeq} satisfies 
\begin{align*}
\tan\phi_\nu &
\simeq \frac {\sigma_2 -b_0\sigma_1}
{
 (\sigma_1  -1)\sigma_1   -  b_0\sigma_2 
}
=
\frac 
{
-\frac{1}{3}  b_0^3   + b_2  
}
{
  \frac{1}{12}  b_0^4+ b_1^2 -  \frac{1}{2}  b_0^2 - b_1     - b_2b_0
} \\
&
=\frac 
{
 \frac{1}{3}  b_0^3   - b_2  
}
{
 \frac 1 4 +  \frac{1}{2}  b_0^2 + b_2b_0 -\frac{1}{12}  b_0^4 
} =: \Delta_\alpha.
\end{align*}
A lengthy but otherwise direct calculation shows that the residual in this equality is differentiable with respect to $\nu$ and its derivative 
is less than 1 in magnitude for all $\nu$ large enough. Hence for all sufficiently large integer $n$ the integro-algebraic 
system has the unique solution, obtained through fixed-point iterations, and  its algebraic part $\nu_n$ satisfies 
$$
\nu_n  = \pi n -2\arg\{i X_0(i)\} + \arctan \Delta_\alpha + r_n n^{-1}\quad n\to\infty
$$  
where $r_n$ is a sequence, uniformly bounded in $n$ and $\alpha\in [1,\alpha_0]$. Asymptotics \eqref{asym2} 
now follows from \eqref{X0ifBm2}.   
\end{proof}

The following lemma derives asymptotic approximation for the eigenfunctions: 

\begin{lem}\label{lem6.88}
Under the enumeration, introduced by Lemma \ref{lem7.7}, the eigenfunctions admit the approximation:
\begin{multline}\label{phinflaifBm2}
\varphi_n(x) =  
\sqrt{2} \cos \Big(\nu_n x + \frac {3-\alpha}8\pi - \arctan \Delta_\alpha\Big)
 \\       
-\frac  {\sqrt{5-\alpha}}  \pi \int_{0}^\infty \rho_0(t)\left( 
Q_0(t) e^{-t\nu_n x} 
+ (-1)^n Q_1(t) e^{-t\nu_n (1-x)} \right) dt + r_n(x)n^{-1}
\end{multline}
where the residual $r_n(x)$ is uniformly bounded in both $n\in \mathbb{N}$ and $x\in [0,1]$ and 
\begin{equation}
\label{rho0ifBm2}
\rho_0(t) = \frac{ \sin \theta_0(t)}{ \gamma_0(t) }t \exp \left(\frac 1 \pi \int_0^\infty\frac{\theta_0(s)}{s +t }ds \right)
\end{equation}
and 
\begin{equation}\label{QifBm2}
\begin{aligned}
Q_0(t) := & \frac {\Delta_\alpha}
{ 
 \sqrt{ \Delta_\alpha^2 +1} 
}
\frac {\sigma_1 }{ \sigma_2  -b_0\sigma_1      }
\Big(      \frac{\sigma_2}{\sigma_1}b_0 -  \sigma_1
+  \big(b_0    -   \frac{\sigma_2}{\sigma_1}\big) t 
    -   t^2 
\Big) \\
Q_1(t) := & 1
\end{aligned}
\end{equation} 
Moreover,
\begin{equation}
\label{funphin}
\begin{aligned}
\varphi_n(1)  &= 
 -  (-1)^n   \sqrt{5-\alpha} \big( 1+O(n^{-1})\big)\\
\int_0^1\varphi_n(x)dx  & =  
\nu_n^{-1}   \sqrt{5   -   \alpha} \frac {\Delta_\alpha}
{ 
 \sqrt{ \Delta_\alpha^2 +1} 
}
\frac {\sigma_2 }{ \sigma_2  -b_0\sigma_1      } \big( 1+O(n^{-1})\big).
\end{aligned}
\end{equation}
\end{lem}

\begin{proof}
The asymptotic structure of matrix $M(\nu)$ in  \eqref{Mnu2} implies the following  
relations between the coefficients $k_1$, $k_2\nu_n$ and $k_3\nu^2_n$:
\begin{equation}
\begin{aligned}\label{k12eq}
k_1 & \simeq -\big(b_0 \cos \phi_{\nu_n} - \sin\phi_{\nu_n}\big)k_2\nu_n 
-\big((\sigma_1  -1) \cos \phi_{\nu_n}   -   b_0  \sin\phi_\nu\big)k_3\nu_n^2  \\
k_2   \nu_n  &\simeq  -\frac{\sigma_2}{\sigma_1} k_3 \nu_n^2
\end{aligned}
\end{equation} 
Plugging these expressions and the estimates from Lemma \ref{lem6.12} into \eqref{phi01ifbm22} gives
\begin{equation}
\label{Phi0inu}
\Phi_0(i\nu_n)    \simeq 2 X(i\nu_n)\Big(   (b_0   +i)k_2\nu_n 
 + (\sigma_1     -1 +   b_0 i) k_3  \nu_n^2 
 \Big) = 2 i\nu_n X_c(i\nu_n)\zeta   k_3 \nu^2_n  
\end{equation}
where we used the equality $X(i\nu_n)=i\nu_n X_c(i\nu_n)$ and defined 
$$
\zeta :=     -b_0\frac{\sigma_2}{\sigma_1}   
 + \sigma_1     -1 +  i\Big( b_0  - \frac{\sigma_2}{\sigma_1}    \Big).
$$
The argument and the absolute value of this constant are given by   
$$
\arg\{\zeta\}=
  -\arctan \Delta_\alpha
 \quad \text{and}\quad 
|\zeta|^2 = 
 \Big(\frac 1 {\Delta_\alpha^2}+1\Big)\frac {\big(\sigma_2  -b_0\sigma_1     \big)^2} {\sigma_1^2}.
$$
Similarly we have
\begin{equation}
\label{Phi01inu}
\begin{aligned}
\Phi_0(-\nu_n t)  &\simeq  -2\nu_n t X_c(-t\nu_n)\Big( -\frac{\sigma_2}{\sigma_1}    (b_0  -    t)
+ (\sigma_1   -  b_0  t+   t^2)
\Big) k_3 \nu_n^2 \\
\Phi_1(-\nu_n t)  & \simeq -2\nu_n t X_c(-t\nu_n) k_1   
\end{aligned}
\end{equation}
where the residuals are bounded uniformly with respect to $t$ by Lemma \ref{lem6.12}. 
Combining the equations in \eqref{k12eq} and using the definition of $\Delta_\alpha$, we also have      
\begin{align*}
k_1   & \simeq  
\Big(
\big(b_0 \cos \phi_{\nu_n} - \sin\phi_{\nu_n}\big)
 \frac{\sigma_2}{\sigma_1}   
-\big((\sigma_1  -1) \cos \phi_{\nu_n}   -   b_0  \sin\phi_\nu\big) 
\Big)k_3\nu_n^2=\\
& =
-(-1)^n\frac 1{\sqrt{1+\Delta_\alpha^2}}\Big(
\big(b_0   - \Delta_\alpha\big)
 \frac{\sigma_2}{\sigma_1}   
-\big((\sigma_1  -1)    -   b_0   \Delta_\alpha\big) 
\Big)k_3\nu_n^2 =
-(-1)^n |\zeta| k_3\nu_n^2.
\end{align*}

Expression  \eqref{phinflaifBm2} is now obtained by plugging \eqref{Phi0inu}-\eqref{Phi01inu} and  \eqref{X0ifBm2} into \eqref{eigfunifBm} 
and normalizing by the factor
$$
C_n := - 2\nu^2_n k_3 \frac{|c_\alpha|}{\lambda \Gamma(\alpha)}  \frac {|\zeta|} {\sqrt{5   -   \alpha}}.
$$
Asymptotic formulas \eqref{funphin} follow by normalizing expressions \eqref{k1234} by the same factor.
\end{proof}  
\subsubsection{Enumeration alignment}

The enumeration, introduced in Lemma \ref{lem7.7}, may differ from the {\em natural} enumeration, which puts all the eigenvalues 
into increasing order, only by a constant shift. This shift can be identified by the calibration procedure, 
based on continuity of the spectrum, similar to Section 5.1.7. of \cite{ChK}. Since for the integrated Brownian motion, 
corresponding to $\alpha=1$, the sequence $\nu_n$ is asymptotic to $\pi (n-\frac 1 2)$, the asymptotics in \eqref{asym2}
should be shifted by $\pi$ and
the formulas in (2) and (3) of Theorem \ref{thm-ifBm} are obtained by the corresponding adjustment of the expressions 
from Lemma \ref{lem6.88}. 

\subsection{The case $H>\frac 1 2$}

For $\alpha\in (0,1)$, corresponding to $H\in (\frac 1 2,1)$, the proof is done completely differently, since 
as we will see below, structural function $\Lambda(z)$ in this case has more roots than before.  

\subsubsection{The Laplace transform} 

\begin{lem}
Let $(\lambda, \varphi)$ be a solution of \eqref{eigifBm}, then the Laplace transform of $\varphi$ satisfies 
\begin{equation}
\label{phizifbm} 
\widehat\varphi(z)= \widehat \varphi(0)+z\frac d {dz}\widehat \varphi(z)_{\big|z=0} -\frac {e^{-z}\Phi_1(-z)+\Phi_0(z)}{\Lambda(z)} 
\end{equation}
where 
functions $\Phi_0(z)$ and $\Phi_1(z)$
are sectionally holomorphic on the cut plane $\mathbb{C}\setminus \Real_{>0}$ and
\begin{equation}
\label{lambdaifbm}
\Lambda(z) := \frac {\lambda\Gamma(\alpha)}{c_\alpha}z^2 -\frac{1}{z^2} \int_0^\infty \frac{2 t^{\alpha }}{t^2-z^2} 
 dt. 
\end{equation} 
\end{lem}

\begin{proof}
As in the case $H<\frac 1 2$, the function 
$
\psi(x) = \int_x^1 \int_y^1 \varphi(u) du dy,
$
satisfies \eqref{geneig_ifBm2}.
Setting $c_{\alpha}=(1-\frac{\alpha}{2})(1-\alpha)$ and interchanging derivative and integration we arrive at the following 
generalized eigenproblem:
\begin{equation}\label{geneig_ifBm}
\begin{aligned}
&
c_\alpha \int_0^1 |x-y|^{-\alpha} \psi(y) dy=\lambda \psi^{(4)}(x), \quad x\in [0,1] \\
&
\psi(1)=0, \quad \psi'(1)=0, \quad \psi''(0)=0, \quad \psi^{(3)}(0)=0.
\end{aligned}
\end{equation}
Define 
$$
u(x,t) = \int_0^1  e^{-t|x-y|}\psi(y) dy\quad \text{and}\quad
u_0(x) = \int_0^{\infty} t^{\alpha-1}u(x,t)dt,
$$ 
then, plugging the identity \eqref{Gfla} into \eqref{geneig_ifBm} we get
$$
\frac{c_\alpha}{\Gamma(\alpha)} u_0(x) = \lambda \psi^{(4)}(x), \quad x\in [0,1] 
$$
and therefore, in view of the boundary conditions in \eqref{geneig_ifBm},
\begin{equation}
\label{u0zpsiz}
\begin{aligned}
\widehat u_0(z) = & \frac{\lambda \Gamma(\alpha)}{c_\alpha} \int_0^1 e^{-zx} \psi^{(4)}(x) dx=\\
 &\frac{\lambda \Gamma(\alpha)}{c_\alpha}\Big( e^{-z}\psi^{(3)}(1)+z e^{-z}\psi''(1)-z^2 \psi'(0)-z^3 \psi(0) +z^4 \widehat \psi(z) \Big).
\end{aligned}
\end{equation}
An additional relation between $\widehat{u}_0(z)$ and $\widehat{\varphi}(z)$ is obtained as in the proof Lemma \ref{lem5.1}
(c.f. \eqref{rel2}):
$$
\widehat u_0(z)=  \int_0^\infty \frac{t^{\alpha-1}}{z-t}u(0,t)dt-e^{-z} \int_0^\infty \frac{t^{\alpha-1}}{z+t}u(1,t)dt-\widehat \psi(z) \int_0^\infty \frac{2t^{\alpha}}{z^2-t^2}dt. 
$$
Combining this with \eqref{u0zpsiz} and rearranging we obtain the expression 
$$
z^2 \widehat \psi(z)=-\frac {e^{-z}\Phi_1(-z)+\Phi_0(z)}{\Lambda(z)},
$$  
where $\Lambda(z)$ is defined in \eqref{lambdaifbm} and 
\begin{equation}\label{phi01ifbm}
\begin{aligned}
\Phi_0(z) := &-\frac {\lambda\Gamma(\alpha)}{c_\alpha} \psi'(0)z^2 - \frac {\lambda\Gamma(\alpha)}{c_\alpha} \psi(0)z^3+\int_0^\infty \frac{t^{\alpha-1}}{t-z}u(0,t)dt\\
\Phi_1(z) := &\phantom{+} \frac {\lambda\Gamma(\alpha)}{c_\alpha} \psi^{(3)}(1) - \frac {\lambda\Gamma(\alpha)}{c_\alpha} \psi''(1)z+\int_0^\infty \frac{t^{\alpha-1}}{t-z}u(1,t)dt.
\end{aligned}
\end{equation}
Formula  \eqref{phizifbm} follows since 
$$
\widehat \varphi(z) = \widehat \psi''(z)= -\psi'(0)-\psi(0)z + z^2 \widehat \psi(z).
$$
\end{proof}

Next lemma details the structure of $\Lambda(z)$: 

\begin{lem}  \label{lem62}
\

\medskip

\noindent
a) $\Lambda(z)$ admits the expression  
\begin{equation}\label{Lambdazifbm}
\Lambda (z) = \frac {\lambda\Gamma(\alpha)}{c_\alpha} z^2 - \frac{\pi   }{ \cos \frac{\pi } 2\alpha} z^{\alpha-3} 
\begin{cases}
e^{\frac{1-\alpha}{2}\pi i}  & \arg (z) \in (0,\pi) \\
e^{-\frac{1-\alpha}{2}\pi i }  &   \arg(z)\in (-\pi, 0)
\end{cases}
\end{equation} 
and has six zeros 
$$
\pm z_0 = \pm i\nu , \quad \pm z_+ = \pm \nu e^{\frac{\pi}{2}\frac{1-\alpha}{5-\alpha}i}, \quad 
\pm z_- = \pm \nu e^{\frac{\pi}{2}\frac{9-\alpha}{5-\alpha}i}
$$
where $\nu$ is given by 
\begin{equation}\label{lambdanuifbm}
\nu^{\alpha-5}=
\frac{\lambda \Gamma(\alpha)}{c_\alpha} \frac {\cos \frac{\pi } 2\alpha} {\pi }.
\end{equation}

\medskip 
\noindent 
b) The limits of $\Lambda(z)$ across the real line are given by 
\begin{equation}\label{Lpmt_ifBm}
\Lambda^\pm (t) =
 \frac {\lambda\Gamma(\alpha)}{c_\alpha} t^2
\mp
|t|^{\alpha-3} 
\frac{\pi   }{ \cos \frac{\pi } 2\alpha}
\begin{cases}
e^{\frac{1\mp\alpha}{2}\pi i} & \quad t>0 \\
e^{-\frac{1\mp\alpha}{2}\pi i } & \quad t<0
\end{cases}
\end{equation}
and satisfy the symmetries \eqref{conjp}-\eqref{absL}.

\medskip 
\noindent c) The argument 
$
\theta(t) := \arg\{\Lambda^+(t)\} \in (-\pi,\pi]
$
is an odd function $\theta(t)=-\theta(-t)$,
$$
\theta(t) = \arctan \frac
{
-
\sin \frac{1-\alpha}{2}\pi  
}
{
(t/\nu)^{5-\alpha}
- 
\cos \frac{1-\alpha}{2}\pi 
},\quad t>0
$$
increasing continuously from $\theta(0+)=-\frac {1+\alpha}2\pi$ to $0$ as $t\to\infty$. The rescaled function $\theta_0(u):=\theta(u\nu)$ satisfies 
\begin{equation}\label{bifbm}
b_{k,\alpha}:=\frac{1}{\pi}\int_0^\infty u^k \theta_0(u)du=-\frac{1}{k+1} \frac{\sin\big((k+1)\frac{1+\alpha}{2}\frac{\pi}{5-\alpha}\big)}{\sin\big((k+1)\frac{\pi}{5-\alpha}\big)}\quad k=0,1,2
\end{equation}
and the following asymptotics holds:
$$
\frac{1}{\pi} \int_0^\infty \frac{\theta(s)}{s-z}ds = -\frac{\nu b_0}{z}-\frac{\nu^2 b_1}{z^2}-\frac{\nu^3 b_2}{z^3}
+O(z^{-4}), \quad z \to\infty, \quad z \in \mathbb{C}\setminus \Real_{>0}.
$$ 

\end{lem}

\begin{proof} 
\

\medskip
\noindent 
a) 
Formula \eqref{Lambdazifbm} follows from definition \eqref{lambdaifbm} and identity \eqref{intfla}.  
Note that conjugate of any zero of $\Lambda(z)$ is also a zero, hence it is enough to locate zeros only in the upper half plane. 
To this end let $z=\nu e^{i \omega}$ with $\nu>0$ and $\omega \in (0,\pi)$, then 
equating \eqref{Lambdazifbm} to zero gives
$$
\frac{\lambda \Gamma(\alpha)}{c_\alpha} \frac {\cos \frac{\pi } 2\alpha} {\pi } \nu^{5-\alpha}=
\exp \left(  \frac{1-\alpha}{2}\pi i+(\alpha-5)\omega i\right).
$$
Obviously, any solution must satisfy \eqref{lambdanuifbm} and 
$$
\frac{1-\alpha}{2}\pi +(\alpha-5)\omega =2\pi k\quad \text{for some \ } k\in \mathbb{Z}. 
$$ 
The only values of $k$, for which 
$$
\omega=\frac \pi 2 \frac{1-\alpha-4 k}{5-\alpha}  \in (0,\pi),
$$
are $0$, $-1$ and $-2$, corresponding to three zeros in the upper half plane:
$$
z_+ = \nu e^{\frac{\pi}{2}\frac{1-\alpha}{5-\alpha}i}, \quad z_0 = \nu i, \quad  z_- = \nu e^{\frac{\pi}{2}\frac{9-\alpha}{5-\alpha}i}.
$$
The zeros in the lower half plane are the conjugates 
$$
\overline{z_0} = -z_0, \quad \overline{z_+} =  -z_-, \quad \overline{z_-} = - z_+.
$$

\medskip
\noindent
b) All the formulas  are obtained from \eqref{Lambdazifbm} by direct calculations.  

\medskip
\noindent
c) The expression for $\theta(t)$ follows from \eqref{Lpmt_ifBm}. The integrals in \eqref{bifbm}  reduce to 
\begin{equation*}
\frac{1}{\pi}\int_0^\infty u^k \theta_0(u)du=-\frac{(\cos(\frac{\pi}{2}\alpha))^{\frac{k+1}{5-\alpha}}}{\pi (k+1)} \int_0^\infty \frac{s^{\frac{k+1}{5-\alpha}}}{1+(\cot(\frac{\pi}{2}(1+\alpha))+s)^2} ds
\end{equation*}
by a change of variable and the claimed formulas are obtained by appropriate contour integration. 
\end{proof}

\subsubsection{Removal of singularities}

Since the Laplace transform is an entire function, removal of poles  in \eqref{phizifbm}
gives  
\begin{equation} \label{algcifbm}
\begin{aligned} 
&
\Phi_0(\pm z_0)+e^{\mp z_0}\Phi_1(\mp z_0)  = 0 \\
&
\Phi_0(\pm z_+)+e^{\mp z_+}\Phi_1(\mp z_+)  = 0 \\
&
\Phi_0(\pm z_-)+e^{\mp z_-}\Phi_1(\mp z_-)  = 0 
\end{aligned} 
\end{equation}
and removal of discontinuity on the real line as in Section \ref{sec512} yields the boundary conditions, c.f. \eqref{Hp}:  
$$
\begin{aligned}
&
\Phi_0^+(t) - e^{2i\theta(t)}\Phi_0^-(t) = 2i  e^{-t} e^{i\theta(t)}\sin\theta(t) \Phi_1(-t)
\\
&
\Phi_1^+(t) - e^{2i\theta(t)}\Phi_1^-(t) = 2i e^{-t}  e^{i\theta(t)}\sin\theta(t) \Phi_0(-t)
\end{aligned}\qquad t>0.
$$
Since $tu(0,t)$ and $tu(1,t)$ are bounded functions, it follows from \eqref{phi01ifbm} that 
\begin{equation}
\label{phi01growthifbm}
\begin{aligned}
\Phi_0(z) =\, &  2k_5 z^2 + 2k_6 z^3 + O(z^{-1})\\
\Phi_1(z) =\, & 2k_3 + 2k_4 z + O(z^{-1})
\end{aligned} \qquad \text{as\ } z\to \infty
\end{equation}
where we defined
\begin{equation}\label{k1to6}
\begin{aligned}
k_5 &=-\frac {\lambda\Gamma(\alpha)}{2 c_\alpha} \psi'(0) , \quad k_6=- \frac {\lambda\Gamma(\alpha)}{2 c_\alpha} \psi(0) \\
k_3 &= \frac {\lambda\Gamma(\alpha)}{2 c_\alpha} \psi^{(3)}(1), \quad k_4=- \frac {\lambda\Gamma(\alpha)}{2 c_\alpha} \psi''(1)
\end{aligned}
\end{equation}
Also we have   
\begin{equation}\label{phi01estifBmzero}
\Phi_0(z)\sim z^{\alpha-1}\quad \text{and}  \quad  \Phi_1(z) \sim z^{\alpha-1}\quad \text{as}\ z\to 0. 
\end{equation}

\subsubsection{An equivalent formulation of the eigenproblem}

The appropriate solution of the homogeneous Riemann boundary value problem 
$$
X^+(t) - e^{2i\theta(t)}X^-(t) =0, \quad t\in \Real_{>0},
$$
in this case is given by the Sokhotski--Plemelj formula 
\begin{equation}\label{Xz_ifbm}
X(z)= \frac 1 z X_c(z)=
\frac{1}{z} \exp \left(\frac 1 \pi \int_0^\infty\frac{\theta(t)}{t-z}dt\right), \quad z\in \mathbb{C}\setminus \Real_{>0}.
\end{equation}
Factor $1/z$ in front of the exponential is chosen to guarantee square integrability of the functions 
\begin{equation}\label{SzDz}
\begin{aligned}
& S(z):= \frac{\Phi_0(z)+\Phi_1(z)}{2X(z)} \\
& D(z):= \frac{\Phi_0(z)-\Phi_1(z)}{2X(z)}
\end{aligned}
\end{equation}
at the origin, in view of  estimates \eqref{phi01estifBmzero} and \eqref{Xzeroifbm} below: 
 
\begin{lem} 
The function defined in \eqref{Xz_ifbm} satisfies 
\begin{equation}\label{Xzeroifbm}
X(z) \sim z^{\frac{\alpha-1}{2}}   \quad \text{as\ } z \to 0
\end{equation}
and 
\begin{equation}\label{Xinfifbm}
X(z) =    \frac{1}{z}   -  b_0\frac{\nu}{z^2} +  \Big( \frac{1}{2} b_0^2-b_1 \Big)\frac{\nu^2}{z^3}    -  \Big( \frac{1}{6} b_0^3-b_0b_1 + b_2\Big)\frac{\nu^3}{z^4} + O(z^{-5}), \    z \to \infty.
\end{equation}
Moreover, $X_0(z):=X_c(\nu z)$ satisfies 
\begin{equation}\label{X0iifBm}
\arg\big\{X_0(i)\big\} = -\frac {1+\alpha}8\pi\quad \text{and}\quad |X_0(i)| = \sqrt{\frac{5-\alpha}{8}}
 \frac{1}{\cos \frac{\pi}{2}\frac{1-\alpha}{5-\alpha}}.
\end{equation}
\end{lem}

\begin{proof}

Estimate \eqref{Xzeroifbm} is valid, since  
$
X_c (z) \sim z^{-\theta(0+)/\pi}
$
with $\theta(0+)=-\frac {1+\alpha}{2}\pi$ (see Lemma \ref{lem62} (c)).  
The asymptotics at infinity is obtained by \eqref{bifbm} and the Taylor expansion of exponential. 
The formulas in \eqref{X0iifBm} follow from the identities    
$$
\arg\{X_0(i)\} = \frac {\theta_0(0+)} 4
\quad
\text{and}
\quad 
|X_0(i)|^2 =  \frac{c_\alpha}{\lambda \Gamma(\alpha)} \lim_{z\to z_0} 
\frac{z^4 \Lambda(z)}{(z^2-z_0^2) (z^2-z_+^2) (z^2-z_-^2)}
$$
proved as in Lemma 5.5 \cite{ChK}.  
\end{proof}

Being integrable at the origin, functions $S(z)$ and $D(z)$ satisfy, c.f. \eqref{SD}:
$$
\begin{aligned}
& 
S(z) = \phantom{+}\frac 1 \pi \int_0^\infty \frac{h(t)e^{-t}}{t-z}S(-t) dt + P_S(z) \\
&
D(z) = -\frac 1 \pi \int_0^\infty \frac{h(t)e^{-t}}{t-z}D(-t)dt + P_D(z)
\end{aligned}
$$
where polynomials are chosen to match a priori growth of $S(z)$ and $D(z)$ at infinity,
determined by \eqref{phi01growthifbm} and \eqref{Xinfifbm}: 
\begin{equation*}
\begin{aligned}
& 
P_S(z) = k_1 + l_1 z +l_2 z^2 + l_3 z^3 + l_4 z^4  \\
&
P_D(z) = k_2 + m_1 z +m_2 z^2 + m_3 z^3 + m_4 z^4
\end{aligned}
\end{equation*}
Here  $k_1$ and $k_2$ are arbitrary and the rest of the constants are related to the previously introduced quantities through 
matching the powers in \eqref{SzDz}:   
\begin{align*}
l_4 & = k_6 & 
m_4 & = k_6 \\
l_3 & = k_5 +  b_0\nu k_6 & 
m_3 & = k_5 +  b_0\nu k_6 \\
l_2 & = k_4 +  b_0\nu k_5 +  \nu^2 \sigma_1k_6 & 
m_2 & = -k_4 +  b_0\nu k_5 +  \nu^2 \sigma_1k_6 \\
l_1 & = k_3 +  b_0 \nu k_4 +  \nu^2 \sigma_1k_5+ \nu^3 \sigma_2k_6 & 
m_1 & = -k_3 -  b_0 \nu k_4 +  \nu^2 \sigma_1 k_5+ \nu^3 \sigma_2 k_6
\end{align*}
where we defined 
\begin{align*}
& \sigma_1 = \frac{1}{2}  b_0^2 + b_1\\
& \sigma_2 = \frac{1}{6}  b_0^3 + b_0 b_1 + b_2.
\end{align*}

Consider now the integral equations
\begin{equation}\label{pqifbm}
p^\pm_j (t)  = \pm \frac 1 \pi \int_0^\infty \frac{h_0(s)e^{-\nu s}}{s+t} p^\pm_j(s)ds+t^j, \quad t>0, \quad j\in \{0,1,2,3,4\},
\end{equation}
where $h_0(s):= h(s\nu)$ with $h(s)$ being defined as in \eqref{hdefine}. As in Lemma \ref{lem5.5}, the integral operator in the 
right hand side is a contraction on $L^2(0,\infty)$. Consequently equations \eqref{pqifbm} have unique solutions, such that functions 
$p^\pm_j(t)-t^j$ belong to $L^2(0,\infty)$. 
Since $S(-t)$ and $D(-t)$ are square integrable at the origin, by linearity we have
\begin{equation}\label{Sifbm}
\begin{aligned}
S(z\nu)   =\;  & 
k_1 p^+_0(-z) -  k_3\nu p^+_1(-z)  + \\
&
 k_4 \nu^2 \Big(   -     b_0  p^+_1(-z)+p^+_2(-z)\Big)  +\\
&
k_5 \nu^3 \Big( -    \sigma_1  p^+_1(-z) +   b_0    p^+_2(-z)-   p^+_3(-z) \Big) +\\
& 
k_6 \nu^4
\Big( -   \sigma_2  p^+_1(-z) +    \sigma_1  p^+_2(-z)
-    b_0    p^+_3(-z) + p^+_4(-z)\Big)
\end{aligned}
\end{equation}
and 
\begin{equation}\label{Difbm}
\begin{aligned}
D(z\nu)  =\; & k_2 p^-_0(-z)+k_3 \nu p^-_1(-z) +\\
&  
k_4 \nu^2 \Big( b_0    p^-_1(-z)   -  p^-_2(-z)\Big) + \\
& 
k_5 \nu^3\Big(-  \sigma_1   p^-_1(-z)  +  b_0    p^-_2(-z)-    p^-_3(-z) \Big)  + \\
&
 k_6 \nu^4\Big(   -   \sigma_2 p^-_1(-z)+   \sigma_1 p^-_2(-z)  - b_0    p^-_3(-z) +p^-_4(-z)
\Big)
\end{aligned}
\end{equation}
where the domain of $p^\pm_j(z)$ is extended to the cut plane by replacing $t$ with $z\in \mathbb{C}\setminus \Real_{<0}$ in \eqref{pqifbm}.
Now plugging \eqref{Sifbm} and \eqref{Difbm} into  definition \eqref{SzDz} and letting  
$
a^\pm_j(z) := p^+_j(z)\pm p^-_j(z) 
$
we obtain  
\begin{equation}\label{phi0ifbm2}
\frac{\Phi_0(z\nu)}{X(z\nu) }= 
k_1 \xi_1(-z) + k_2 \xi_2(-z)    +k_3 \nu \xi_3(-z) + k_4 \nu^2 \xi_4(-z)    + k_5 \nu^3 \xi_5(-z) + k_6 \nu^4\xi_6(-z)
\end{equation}
\begin{equation}\label{phi1ifbm2}
\frac {\Phi_1(z \nu )}{X(z \nu)}   = 
k_1 \eta_1(-z) +k_2 \eta_2(-z) +k_3 \nu \eta_3(-z)+ k_4 \nu^2 \eta_4(-z) + k_5 \nu^3 \eta_5(-z) +k_6 \nu^4\eta_6(-z)
\end{equation}
where  
\begin{equation}\label{xiifBm}
\begin{aligned}
\xi_1(z) & := \phantom{+} p^+_0(z) \\
\xi_2(z) & := \phantom{+} p^-_0(z) \\
\xi_3(z) & := -a^-_1(z) \\
\xi_4(z) & :=  -     b_0 a^-_1(z)  + a^-_2(z) \\
\xi_5(z) & :=   -    \sigma_1 a^+_1(z)   +   b_0  a^+_2(z)-   a^+_3(z) \\
\xi_6(z) & := -   \sigma_2 a^+_1(z) +    \sigma_1  a^+_2(z) -    b_0    a^+_3(z) + a^+_4(z)
\end{aligned}
\end{equation}
and
\begin{equation}\label{etaifBm}
\begin{aligned}
\eta_1(z) & := \phantom{+} p^+_0(z)\\
\eta_2(z) & := -p^-_0(z) \\
\eta_3(z) & := -   a^+_1(z) \\
\eta_4(z) & := -     b_0  a^+_1(z)   +a^+_2(z) \\
\eta_5(z) & :=  -    \sigma_1  a^-_1(z) +   b_0    a^-_2(z)   -  a^-_3(z) \\
\eta_6(z) & :=  -   \sigma_2  a^-_1(z)  +    \sigma_1  a^-_2(z)-    b_0   a^-_3(z) + a^-_4(z).
\end{aligned}
\end{equation}
In terms of the objects, introduced above, conditions \eqref{algcifbm} take the form of the system of linear equations  
\begin{equation}\label{alg2cifbm}
\begin{aligned}
& k_1 \gamma_{1,1} + k_2 \gamma_{1,2} + k_3 \nu \gamma_{1,3} + k_4 \nu^2 \gamma_{1,4} + k_5 \nu^3 \gamma_{1,5} + k_6 \nu^4 \gamma_{1,6} = 0\\
& k_1 \gamma_{2,1} + k_2 \gamma_{2,2} + k_3 \nu \gamma_{2,3} + k_4 \nu^2 \gamma_{2,4} + k_5 \nu^3 \gamma_{2,5} + k_6 \nu^4 \gamma_{2,6} = 0\\
& k_1 \gamma_{3,1} + k_2 \gamma_{3,2} + k_3 \nu \gamma_{3,3} + k_4 \nu^2 \gamma_{3,4} + k_5 \nu^3 \gamma_{3,5} + k_6 \nu^4 \gamma_{3,6} = 0\\
\end{aligned}
\end{equation}
where we defined 
\begin{equation}
\label{gammaij}
\begin{aligned}
& 
\gamma_{1,j} := \xi_j(-i)+ e^{-i \nu }\frac{X(-\nu i)}{X(\nu i)}\eta_j(i ) \\
& 
\gamma_{2,j} :=  \xi_j(-e^{\phi i})
+e^{ -  \nu   e^{\phi i}}\frac {X(-\nu e^{\phi i})}{X(\nu e^{\phi i})}\eta_j(e^{\phi i}) \\
&
\gamma_{3,j} := \eta_j( -e^{-\phi i})+e^{ -  \nu   e^{-\phi i}}\frac{X(-\nu e^{-\phi i})}{X(\nu e^{-\nu \phi i})}\xi_j( e^{-\phi i})
\end{aligned}
\end{equation}
and    
$
\phi = \displaystyle \frac{\pi}{2}\frac{1-\alpha}{5-\alpha}.
$
Since $k_1, k_2, ...,k_6$ and $\nu$ are real, the system \eqref{alg2cifbm} comprise of six equations with real coefficients. 
Let $M(\nu)$ denote the matrix of these coefficients: 
\begin{equation}\label{Mij}
\begin{aligned}
M_{1,j}(\nu) = \Re\{\gamma_{1,j}\},\quad & M_{2,j}(\nu) = \Im\{\gamma_{1,j}\} \\
M_{3,j}(\nu) = \Re\{\gamma_{2,j}\},\quad & M_{4,j}(\nu) = \Im\{\gamma_{2,j}\} \\
M_{5,j}(\nu) = \Re\{\gamma_{3,j}\},\quad & M_{6,j}(\nu) = \Im\{\gamma_{3,j}\}
\end{aligned}
\end{equation}
Nontrivial solutions are possible if and only if
\begin{equation}
\label{algcondifBm}
\det\{M(\nu)\}=0.
\end{equation}
and thus we arrive at the following equivalent formulation for the eigenproblem:

\medskip 

\begin{lem}\label{lem6.4}
Let $(p^\pm_0, ..., p^\pm_4, \nu)$ with $\nu>0$ be a solution of the system, which consists of the integral equations \eqref{pqifbm} 
and the algebraic equations \eqref{algcondifBm}. Let $\varphi$ be defined by the Laplace transform, given by the formula \eqref{phizifbm}, 
where $\Phi_0(z)$ and $\Phi_1(z)$ are given by \eqref{phi0ifbm2}-\eqref{phi1ifbm2} and let $\lambda$ be defined 
by \eqref{lambdanuifbm}.
Then the pair $(\lambda, \varphi)$ solves eigenproblem \eqref{eigifBm}. Conversely, any solution $(\lambda, \varphi)$ of 
\eqref{eigifBm} defines a solution to the above integro-algebraic system. 
\end{lem}

\subsubsection{Properties of the integro-algebraic system}
  
As mentioned above, equations \eqref{pqifbm} have unique solutions, such that $p^\pm_j(t)-t^j$ belong to $L^2(0,\infty)$.
The following lemma derives several estimates useful in asymptotic analysis of the integro-algebraic 
system of Lemma \ref{lem6.4}: 

\begin{lem} \label{lem6.5}
For any $\alpha_0 \in (0,1)$ there exist constants $\nu'$ and $C$, such that for all $\nu \geq \nu'$ and
$\alpha \in [\alpha_0,1]$ 
\begin{equation}\label{pqasymptotic}
\big|p^\pm_j(z/\nu)-(z/\nu)^j\big| \leq C \nu^{-(j+1)}, \quad 0 \leq j \leq 4,\quad z\in \{\pm z_0, \pm z_+, \pm z_-\}
\end{equation}
and for all $\tau>0$
\begin{equation}\label{ppmtau}
\big|p^\pm_j(\tau)-\tau^j\big| \leq C \nu^{-(j+1)} \tau^{-1}, \quad 0 \leq j \leq 4.
\end{equation}

\end{lem}

\begin{proof}
The proof is analogous to Lemma 5.7 in \cite{ChK}.
\end{proof}

\subsubsection{Inversion of the Laplace transform}

The eigenfunctions are recovered by inverting the Laplace transform  \eqref{phizifbm}:

\begin{lem}\label{lem-phix}
Let $(\Phi_0, \Phi_1, \nu)$ satisfy the integro-algebraic system introduced in Lemma \ref{lem6.4}, then the function 
\begin{subequations}
\begin{align}\label{phixifBm-a}
\varphi(x)= & 
  \frac 1 \nu \frac{c_\alpha}{\lambda \Gamma(\alpha)}\frac 2 {5   -   \alpha}\Re\bigg\{ e^{i\nu x} \Phi_0(i\nu)   i \bigg\} + \\
\label{phixifBm-b} &  
\frac 1 {\nu}\frac {c_\alpha}{\lambda\Gamma(\alpha)} \frac 1{\pi} 
\int_{0}^\infty  \frac{\sin \theta_0(u)}{  \gamma_0(u)} \left( e^{-u\nu(1-x)}\Phi_1(-u\nu)+e^{-u\nu x}  \Phi_0(-u\nu) \right) du 
+   
\\
\label{phixifBm-c} &
\frac 1 \nu \frac{c_\alpha}{\lambda \Gamma(\alpha)}\frac 2 {5   -   \alpha}\Re
\bigg\{
e^{-\nu e^{\phi i} x} \Phi_0(-\nu e^{\phi i})  e^{-\phi i}   
-  
e^{-\nu e^{\phi i}(1-x)}\Phi_1(-\nu e^{\phi i})   e^{-\phi i}  
\bigg\}
\end{align}
\end{subequations}
with  
$
\gamma_0(u)  
:=
\big|u^2-u^{\alpha-3}  e^{\frac{1-\alpha}{2}\pi i}  \big|
$
and 
$
\phi = \displaystyle \frac{\pi}{2}\frac{1-\alpha}{5-\alpha},
$
solves eigenproblem \eqref{eigifBm}. 
\end{lem}

\begin{proof}
With the singularities being removed,  the expression \eqref{phizifbm} is an entire function and therefore the inversion can be carried out on the imaginary axis: 
\begin{align*}
\varphi(x) = & -\frac 1{2\pi i}\int_{-iR}^{iR} \left(
 \psi^{\prime}(0)+\psi(0)z+\frac { \Phi_0(z)}{\Lambda(z)}+\frac {e^{-z}\Phi_1(-z)}{\Lambda(z)} 
\right)e^{zx}dz =\\
& -\frac 1 {2\pi i} \int_{-i\infty}^{i\infty} \big(f_1(z)+f_0(z)\big)dz,
\end{align*}
where 
$$
f_1(z):= e^{z(x-1)}\frac { \Phi_1(-z)}{\Lambda(z)} \quad \text{and} \quad 
f_0(z):= e^{zx}\left(
\psi^{\prime}(0)+\psi(0)z+\frac { \Phi_0(z)}{\Lambda(z)}
\right).
$$
\begin{figure}
\begin{tikzpicture}
\draw[help lines,->] (-3,0) -- (3,0) coordinate (xaxis);
\draw[help lines,->] (0,-3) -- (0,3) coordinate (yaxis);


\path[draw,line width=0.8pt,
decoration={markings,
mark=at position 0.75cm with {\arrow[line width=1pt]{<}},
mark=at position 3.5cm with {\arrow[line width=1pt]{<}},
mark=at position 5.7cm with {\arrow[line width=1pt]{<}}
},
postaction=decorate] (0,0.1) -- (2,0.1) arc (0:90:2) -- (0,1.2) arc (90:-90:0.1) -- (0,0.1);

\path[draw,line width=0.8pt,
decoration={markings,
mark=at position 0.75cm with {\arrow[line width=1pt]{>}},
mark=at position 3.5cm with {\arrow[line width=1pt]{>}},
mark=at position 5.7cm with {\arrow[line width=1pt]{>}}
},
postaction=decorate] (0,-0.1) -- (2,-0.1) arc (0:-90:2) -- (0,-1.0) arc (90:-90:0.1) -- (0,-0.1);

\path[draw,line width=0.8pt,
decoration={markings,
mark=at position 0.75cm with {\arrow[line width=1pt]{<}},
mark=at position 3.5cm with {\arrow[line width=1pt]{<}},
mark=at position 5.7cm with {\arrow[line width=1pt]{<}}
},
postaction=decorate] (0,0.1) -- (-2,0.1) arc (180:90:2) ; 

\path[draw,line width=0.8pt,
decoration={markings,
mark=at position 0.75cm with {\arrow[line width=1pt]{>}},
mark=at position 3.5cm with {\arrow[line width=1pt]{>}},
mark=at position 5.7cm with {\arrow[line width=1pt]{>}}
},
postaction=decorate] (0,-0.1) -- (-2,-0.1) arc (180:275:2) ; 

\draw (0,1.1) circle (1pt);
\draw (0,-1.1) circle (1pt);
\draw (-0.89,0.65) circle (1pt);
\draw (0.89,0.65) circle (1pt);
\draw (0.89,-0.65) circle (1pt);
\draw (-0.89,-0.65) circle (1pt);

\node[below] at (xaxis) {$\mathrm{Re}(z)$};
\node[left] at (yaxis) {$\mathrm{Im}(z)$};
\node at (-0.3, 1.1) {$z_0$};
\node at (-0.45, -1.1) {$-z_0$};
\node at (-1.1,0.8) {$z_-$};
\node at (1.1,0.8) {$z_+$};
\node at (-1.1,-0.8) {$-z_+$};
\node at (1.1,-0.8) {$-z_-$};
\node at (1.8,1.8) {$C_R^+$};
\node at (1.8,-1.8) {$C_R^-$};
\node at (-1.8,1.8) {${^+}C_R$};
\node at (-1.8,-1.8) {${^-}C_R$};
\node at (1.5,0.3) {$\Sigma_+$};
\node at (1.5,-0.35) {$\Sigma_-$};
\node at (-1.5,0.3) {${_+}\Sigma$};
\node at (-1.5,-0.35) {${_-}\Sigma$};
\node at (0.45, 1.5) {$L^+_{\delta,R}$};
\node at (0.45, -1.5) {$L^-_{\delta,R}$};
\end{tikzpicture}
  \caption{\label{fig1} Integration contour, used for the Laplace transform inversion: the outer circular arcs are of radius $R$ and the half circles around the poles
  have radius $\delta$ }
\end{figure}
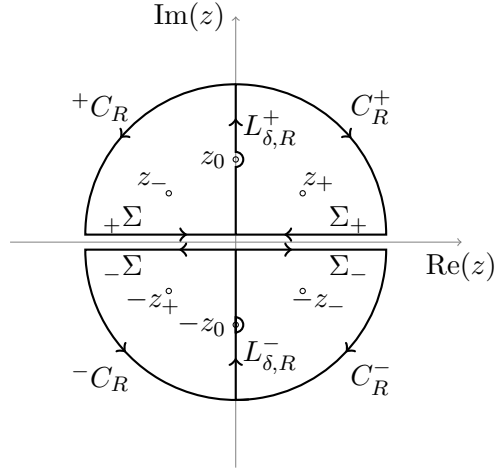

Integrating $f_1(z)$ and $f_0(z)$ over the contours on Figure \ref{fig1} in the right and left half-planes respectively gives 
\begin{align*}
&
\int_{L^+_{\delta,R}} f_1(z)dz + \int_{C_R^+} f_1(z)dz + \int_{\Sigma_+}f_1(z)dz = -2\pi i\Res(f_1, z_+) \\
&
\int_{L^-_{\delta,R}} f_1(z)dz + \int_{\Sigma_-}f_1(z)dz + \int_{C_R^-}f_1(z)dz = - 2\pi i\Res(f_1,- z_-)
\end{align*}
and 
\begin{align*}
& 
\int_{L^+_{\delta,R}}f_0(z)dz +\int_{{^+}C_R}f_0(z)dz + \int_{{_+}\Sigma}f_0(z)dz =  2\pi i\Res(f_0, z_0)+2\pi i\Res(f_0, z_-)\\
&
\int_{L^-_{\delta,R}}f_0(z)dz + \int_{{_-}\Sigma}f_0(z)dz +\int_{{^-}C_R}f_0(z)dz  = 2\pi i\Res(f_0, -z_0)+2\pi i\Res(f_0, -z_+)
\end{align*}

Taking $\delta \to 0$ and $R\to\infty$ and applying Jordan's lemma we get 
\begin{align*}
&
 \int_{0}^{i\infty} f_1(z)dz  - \int_{0}^\infty f_1^+(t)dt = -2\pi i\Res(f_1, z_+) \\
&
\int_{-i\infty}^0 f_1(z)dz + \int_{0}^\infty f_1^-(t)dt    = - 2\pi i\Res(f_1, -z_-)
\end{align*}
and 
\begin{align*}
& 
 \int_{0}^{i\infty} f_0(z)dz   + \int_{-\infty}^0f_0^+(t)dt =  2\pi i\Res(f_0, z_0)+2\pi i\Res(f_0, z_-)\\
&
\int_{-i\infty}^0 f_0(z)dz - \int_{-\infty}^0f_0^-(t)dt    = 2\pi i\Res(f_0, -z_0)+2\pi i\Res(f_0, -z_+)
\end{align*}
Summing up all these equations yields 
\begin{align*}
&
\frac 1{2\pi i} \int_{-i\infty}^{i\infty} \big( f_1(z)+f_2(z)\big)dz  = \\
&  \frac 1{2\pi i} \int_{0}^\infty \big(f_1^+(t) -f_1^-(t)\big)dt  
 +
\frac 1{2\pi i} \int_{ 0}^\infty \big(f_0^-(-t)-f_0^+(-t)\big)dt+\\
 &
 \Res\big(f_0, z_0\big)+ \Res\big(f_0, -z_0\big) + \Res\big(f_0, z_-\big)+ \Res\big(f_0, -z_+\big) 
  - \Res\big(f_1, z_+\big) -  \Res\big(f_1, -z_-\big).
\end{align*}
Further by symmetries \eqref{conjp}-\eqref{absL} 
$$
\begin{aligned}
&
f_1^+(t)  -f_1^-(t) 
= -e^{-t(1-x)}\Phi_1(-t) \frac{2i\sin \theta(t)}{\gamma(t)} \\
&
f_0^-(-t)  -f_0^+(-t)
=- e^{-tx}  \Phi_0(-t)\frac{2i \sin \theta(t)}{\gamma(t)}
\end{aligned}
\qquad t>0
$$
where we defined $\gamma(t)=|\Lambda^+(t)|$.
The residues can be computed using expression \eqref{Lambdazifbm}:
\begin{align*}
&
\Res\big(f_0, z_0\big) 
= e^{i\nu x} \Phi_0(i\nu) \frac {c_\alpha}{\lambda\Gamma(\alpha)} \frac { 1}{\nu i } \frac 1 {5-\alpha}
\\
&
\Res\big(f_0, -z_0\big) 
= 
  e^{-i\nu x} \Phi_0(-i\nu ) \frac{c_\alpha}{\lambda \Gamma(\alpha)} \frac {1 }{ -\nu i}\frac 1 {5 -\alpha}
\\
&
\Res\big(f_0, z_-\big) 
= 
e^{-\nu e^{ -\phi i } x} \Phi_0(-\nu e^{ -\phi i }) \frac {c_\alpha}{\lambda\Gamma(\alpha)} 
\frac { 1}{-\nu  e^{ -\phi i } }\frac 1 { 5    -   \alpha }
\\
&
\Res\big(f_0, -z_+\big) 
= 
e^{-\nu e^{\phi i} x} \Phi_0(-\nu e^{\phi i})\frac{c_\alpha}{\lambda \Gamma(\alpha)} \frac { 1}{-\nu e^{\phi i}} \frac 1{5   -   \alpha}
\\
&
\Res\big(f_1, z_+\big) 
=
e^{\nu e^{\phi i}(x-1)}\Phi_1(-\nu e^{\phi i})\frac {c_\alpha}{\lambda\Gamma(\alpha)}\frac { 1}{\nu e^{\phi i} } 
\frac 1 {5   -   \alpha}
\\
&
\Res\big(f_1, -z_-\big) 
=
e^{\nu e^{ -\phi i }(x-1)} \Phi_1(-\nu e^{ -\phi i })\frac {c_\alpha} {\lambda\Gamma(\alpha)}\frac {1}{\nu  e^{ -\phi i }}
\frac 1 {5  -   \alpha}
\end{align*}
and hence
\begin{align*}
&  \Res\big(f_0, z_0\big)+ \Res\big(f_0, -z_0\big)+ \Res\big(f_0, z_-\big)+ \Res\big(f_0, -z_+\big) 
  - \Res\big(f_1, z_+\big) -  \Res\big(f_1, -z_-\big) =\\
&  
\frac 2 \nu \frac 1 {5   -   \alpha}\frac{c_\alpha}{\lambda \Gamma(\alpha)}\Re\left\{
 e^{i\nu x} \Phi_0(i\nu)   \frac { 1}{ i }  
-
e^{-\nu e^{\phi i} x} \Phi_0(-\nu e^{\phi i}) \frac { 1}{  e^{\phi i}}  
+  
e^{\nu e^{\phi i}(x-1)}\Phi_1(-\nu e^{\phi i}) \frac { 1}{  e^{\phi i} } 
\right\}.
\end{align*}
Assembling all parts together we obtain the expression in \eqref{phixifBm-a}-\eqref{phixifBm-c}.  
\end{proof}

\subsubsection{Asymptotic analysis}

The following lemma determines the asymptotics of algebraic part of the solutions to \eqref{algcondifBm} under a particular enumeration: 

\begin{lem} \label{lem6.15}
The integro-algebraic system from Lemma \ref{lem6.4} has countably many solutions, which can be enumerated so that 
\begin{equation}
\label{nunifBm}
\nu_n = \pi n +  \pi + \frac {1+\alpha}4\pi  + \arctan \Delta_\alpha + \frac{r_n(\alpha)}{n}, \quad n\to\infty
\end{equation}
where $\Delta_\alpha$ is given by \eqref{Deltaalpha} below and the residual $r_n(\alpha)$ is bounded uniformly in $n$ and $\alpha \in [\alpha_0,1]$ for any $\alpha_0 \in (0,1)$.
\end{lem}

\begin{proof}
Formula \eqref{nunifBm} is obtained by asymptotic analysis of the equation \eqref{algcondifBm}, using the estimates of
Lemma \ref{lem6.5}. Let us first find the leading asymptotics of $\xi_j(i)$'s and $\eta_j(i)$'s, 
defined in \eqref{xiifBm}-\eqref{etaifBm}:
\begin{equation}\label{xijati}
\begin{aligned}
\xi_1(i) & \simeq  1   & \eta_1(i) & \simeq  1 \\
\xi_2(i) & \simeq  1   & \eta_2(i) & \simeq - 1  \\
\xi_3(i) & \simeq 0 
& \eta_3(i) & \simeq -   2i \\
\xi_4(i) & \simeq  0 
& \eta_4(i) & \simeq -   2i  b_0     - 2  \\
\xi_5(i) & 
\simeq 2i (1 -     \sigma_1)    -2   b_0     
& \eta_5(i) & \simeq  0     \\
\xi_6(i) & 
\simeq  2i( b_0     -   \sigma_2) + 2 (1-    \sigma_1)  \qquad   
&
\eta_6(i) & \simeq  0
\end{aligned}
\end{equation}
where $\simeq$ stands for equality up to $O(\nu^{-1})$ term, uniform over $\alpha\in [\alpha_0,1]$. Further by \eqref{gammaij} 
\begin{align*}
\gamma_{1,1}  &\simeq 1+ c_\nu-i s_\nu  \\
\gamma_{1,2}  &\simeq 1- c_\nu + i s_\nu   \\
\gamma_{1,3}  &\simeq  -2  c_\nu  i -2  s_\nu   \\
\gamma_{1,4}  &\simeq              - 2 (c_\nu  +           s_\nu   b_0) + i2(s_\nu -   c_\nu   b_0 ) \\
\gamma_{1,5}  &\simeq - 2i (1 -     \sigma_1)    -2   b_0    \\
\gamma_{1,6}  &\simeq - 2i( b_0     -   \sigma_2) + 2 (1-    \sigma_1)   
\end{align*}
where we used the notations
$$
c_\nu = \cos \big(\nu +  2\arg\big\{X(\nu i)\big\}\big)
\quad \text{and}\quad 
s_\nu =  \sin \big(\nu +  2\arg\big\{X(\nu i)\big\}\big).
$$
Similarly, by \eqref{xiifBm}-\eqref{etaifBm} and \eqref{pqasymptotic}
$$
\begin{aligned}
\xi_1(-e^{\phi i}) & \simeq 1 & \eta_1(-e^{-\phi i}) & \simeq \phantom{+}1\\
\xi_2(-e^{\phi i}) & \simeq 1 & \eta_2(-e^{-\phi i}) & \simeq  -1\\
\xi_3(-e^{\phi i}) & \simeq 0 & \eta_3(-e^{-\phi i}) & \simeq  2e^{-\phi i}\\
\xi_4(-e^{\phi i}) & \simeq 0 & \eta_4(-e^{-\phi i}) & \simeq     2  b_0  e^{-\phi i}    + 2 e^{-2\phi i}\\
\xi_5(-e^{\phi i}) & =  2    \sigma_1  e^{ \phi i}    +  2 b_0   e^{2\phi i} +   2 e^{ 3\phi i} 
&
\eta_5(-e^{-\phi i}) & \simeq 0 \\
\xi_6(-e^{\phi i}) & = 2   \sigma_2  e^{\phi i}  +  2  \sigma_1    e^{2\phi i}  
+    2 b_0    e^{3\phi i}  + 2 e^{4\phi i} \quad
& 
\eta_6(-e^{-\phi i}) & \simeq 0.
\end{aligned}
$$
Since $\phi\in (0, \frac \pi 2)$, we have 
$$
\gamma_{2,j} \simeq  \xi_j(-e^{\phi i})\quad \text{and}\quad 
\gamma_{3,j} \simeq  \eta_j( -e^{-\phi i}),
$$
and, collecting all parts together, the leading term asymptotics of \eqref{Mij} reads 
\begin{equation}\label{Mnu}
M(\nu)\simeq
\begin{pmatrix}
  1+ c_\nu & 1- c_\nu  &  -2  s_\nu  &  - 2 (c_\nu  +           s_\nu   b_0) &    -2   b_0 &    2 (1-    \sigma_1)\\
 -s_\nu  &  s_\nu & -2  c_\nu    &  \phantom{+}2(s_\nu -   c_\nu   b_0 )  &   2  (\sigma_1-1)  &   2 (\sigma_2 - b_0) \\
 1 & 1 & 0 & 0 &      2A_1   &  2 B_1
 \\
 0& 0& 0& 0 & 2 A_2 & 2B_2 \\
 1 & -1  & \phantom{+} 2 c^{(1)}  & \phantom{+}2(b_0   c^{(1)}    +  c^{(2)})  & 0 & 0 \\
 0 & 0 & -2 s^{(1)} & -2(b_0   s^{(1)}  +  s^{(2)})  & 0&   0  
\end{pmatrix}
\end{equation}
where we defined 
\begin{align*}
A_1 & = \sigma_1    c^{(1)}    +   b_0     c^{(2)}        +c^{(3)}  &
B_1 & = \sigma_2   c^{(1)}  +    \sigma_1     c^{(2)}  +    b_0       c^{(3)}  +   c^{(4)} \\
A_2 & = \sigma_1    s^{(1)}    +   b_0     s^{(2)}        +s^{(3)} &
B_2 & = \sigma_2   s^{(1)}  +    \sigma_1     s^{(2)}  +    b_0       s^{(3)}  +  s^{(4)}
\end{align*}
and 
$$
c^{(j)} = \cos(j \phi)\quad \text{and}\quad s^{(j)} = \sin (j \phi).
$$
A lengthy but otherwise straightforward calculation yields the following asymptotic 
expression for the determinant of this matrix:
\begin{align*}
\det\{M(\nu)\} \simeq & - c_\nu s^{(2)}\big(A_1B_2-A_2B_1\big) + c_\nu s^{(2)} \big(A_2(\sigma_1-1)-B_2b_0\big)
\\
& +s_\nu s^{(2)} \big(A_2(\sigma_2-b_0)-B_2(\sigma_1-1)\big) 
\end{align*}
and therefore the equation $\det\{M(\nu)\}=0$ becomes  
\begin{equation}
\label{tnu}
\tan\Big(\nu +  2\arg\big\{X(\nu i)\big\}\Big)  = \Delta_\alpha + R(\nu) 
\end{equation}
with 
\begin{equation}\label{Deltaalpha}
\Delta_\alpha = \frac
{
 A_1B_2-A_2B_1 -     A_2(\sigma_1-1)+B_2b_0 
}
{
  A_2(\sigma_2-b_0)-B_2(\sigma_1-1) 
}.
\end{equation}
Here $|R(\nu)|\vee |R'(\nu)|\le C\nu^{-1}$ for all $\nu$ large enough with some constant $C$, depending only on $\alpha_0$.
Hence for any sufficiently large integer $n$, fixed point iterations produce the unique solution to the integro-algebraic system, 
with $\nu_n$ satisfying  asymptotics \eqref{nunifBm}, where we used the expression \eqref{X0iifBm}.

\end{proof}

The next lemma derives asymptotic expression for the eigenfunctions: 

\begin{lem} 
Under the enumeration, introduced by Lemma \ref{lem6.15}, the unit norm 
eigenfunctions admit the approximation:
\begin{subequations}
\begin{align} 
\label{phinx-a}
\varphi_n&(x)  =
\sqrt{2} \cos \Big(\nu_n x + \frac {3-\alpha}8\pi - \arctan \Delta_\alpha\Big)  \\
&
\label{phinx-b}
-
\frac{\sqrt{5   -   \alpha}}{\pi }
\int_{0}^\infty  \rho_0(u)
 \left(  Q_0(u)  e^{-u\nu_n x}  - (-1)^n Q_1(u) e^{-u\nu_n(1-x)}\right) du \\
&  
\label{phinx-c}
+
C_0 e^{- c \nu_n x}   \cos\Big(  s \nu_n x  + \varkappa_0 \Big)
+
C_1 e^{- c \nu_n (1-x)}   \cos\Big(  s \nu_n (1-x)  + \varkappa_1 \Big) +n^{-1} r_n(x)
\end{align}
\end{subequations}
with residual $r_n(x)$, bounded uniformly in both $n\in \mathbb{N}$ and $x\in [0,1]$. Here function
$\rho_0(u)$ is defined in \eqref{rhorho}, polynomials $Q_0(u)$ and $Q_1(u)$ are given in \eqref{Q0Q1}, 
$$
c := \cos \frac{\pi}{2}\frac{1-\alpha}{5-\alpha} \quad\text{and} \quad s :=\sin \frac{\pi}{2}\frac{1-\alpha}{5-\alpha},
$$ 
and the amplitudes $C_0$ and $C_1$ and phases $\varkappa_0$ and $\varkappa_1$ are constants, which depend only on $\alpha$. 
Moreover, the eigenfunctions satisfy 
\begin{equation}\label{phinend}
\varphi_n(1)  = (-1)^n   \sqrt{5-\alpha} \big( 1+O(n^{-1})\big)
\end{equation}
and 
\begin{equation}\label{phinave}
\int_0^1\varphi_n(x)dx   =  
\nu_n^{-1}   \sqrt{5   -   \alpha} \,{C}   \big( 1+O(n^{-1})\big)
\end{equation}
with explicit constant $C$, defined in \eqref{tildeC}.
\end{lem}

\begin{proof}
All the formulas are derived from Lemma \ref{lem-phix}, expressions \eqref{phi0ifbm2}-\eqref{phi1ifbm2} and 
the relations between coefficients $k_j$'s, which follow from the equations defined by matrix \eqref{Mnu}. 
More precisely, equation corresponding to the fourth row of \eqref{Mnu} implies
$$
k_5 \simeq -  \frac{B_2}{A_2} k_6\nu_n,
$$
and hence the third row gives 
$$
k_1 + k_2 \simeq  
2\Big(A_1    \frac{B_2}{A_2}   -    B_1 \Big)k_6 \nu_n^4.
$$
The sixth row implies that 
$$
 k_3   \simeq - \Big(   b_0      +\frac {    s^{(2)} }{s^{(1)}}\Big) k_4 \nu_n,  
$$
and hence by the fifth row
$$
 k_1  -k_2    \simeq 2 \Big(\frac {    s^{(2)} }{s^{(1)}} c^{(1)}        -  c^{(2)}\Big) k_4 \nu_n^2.   
$$
Plugging these expressions into the second row 
we obtain  
$$
\Big(
 -s_{\nu_n}\frac {    s^{(2)} }{s^{(1)}} c^{(1)}  + s_{\nu_n}  c^{(2)}     
     +\frac {    s^{(2)} }{s^{(1)}}c_{\nu_n}      
 +     s_{\nu_n}          
\Big)k_4 + \left(
-     (\sigma_1-1)    \frac{B_2}{A_2}  +    (\sigma_2 - b_0)
\right) k_6 \nu_n^2\simeq0.
$$
Equation \eqref{tnu} implies $c_{\nu_n} \simeq (-1)^n /\sqrt{1+ \Delta_\alpha^2}$
and using trigonometric identities $s^{(2)} = 2s^{(1)} c^{(1)}$ and $c^{(2)} = 2 \big(c^{(1)}\big)^2-1$ we obtain 
\begin{equation}\label{A4}
 k_4 \simeq - (-1)^n\frac 1{2c^{(1)}}    \sqrt{A_3^2+ B_3^2} \, k_6 \nu_n^2 
\end{equation}
where $A_3$ and $B_3$ denote the expressions in the numerator and denominator of \eqref{Deltaalpha}. 
The first row does not impose any further constraint on the coefficients since  
$\det\{M(\nu_n)\}=0$. 

Let us first consider the oscillatory term \eqref{phixifBm-a}. 
Plugging asymptotics \eqref{xijati} and the above relations between the coefficients into  \eqref{phi0ifbm2} gives
\begin{align*}
\frac{\Phi_0(i\nu_n)}{X(i\nu_n) } 
\simeq\; &
k_1  + k_2   - k_5 \nu_n^3 \big(   2i (1 -     \sigma_1)    +2   b_0\big) 
+ k_6 \nu_n^4 \big(- 2i( b_0     -   \sigma_2) + 2 (1-    \sigma_1) \big) \simeq \\
&
2\big( A_3  + i B_3\big) k_6 \nu_n^4.
\end{align*}
Hence expression in \eqref{phixifBm-a} satisfies 
\begin{align}
&
\nonumber
\frac 1 {\nu_n} \frac{c_\alpha}{\lambda \Gamma(\alpha)}\frac 2 {5   -   \alpha}
\Re\bigg\{ e^{i\nu_n x} \Phi_0(i\nu_n)   i  \bigg\} \simeq 
k_6 \nu_n^3  \frac{c_\alpha}{\lambda \Gamma(\alpha)}\frac 4 {5   -   \alpha}
\Re\bigg\{ e^{i\nu_n x}  \big( A_3  + i B_3\big)  X(i\nu_n)  i  \bigg\} \simeq \\
&
\label{oscterm}
k_6 \nu_n^2  \frac{c_\alpha}{\lambda \Gamma(\alpha)}\sqrt{\frac 2 {5   -   \alpha}}\sqrt{A_3^2+B_3^2}  
 \frac{1}{c^{(1)}}
\cos\Big(  \nu_n x   + \arctan \frac{B_3}{A_3} -\frac {1+\alpha}8\pi  \Big) 
\end{align}
where we used the formulas from \eqref{X0iifBm}. Since  
$A_3/B_3 =\Delta_\alpha$ and  $\arctan x + \arctan 1/x =   \pi/ 2$, we have
$$
\arctan \frac{B_3}{A_3} -\frac {1+\alpha}8\pi =  \frac {3-\alpha}8\pi-\arctan \Delta_\alpha.
$$
Thus normalizing \eqref{oscterm} by the constant factor 
\begin{equation}\label{Cn}
C_n := 
k_6 \nu_n^2  \frac{c_\alpha}{\lambda \Gamma(\alpha)}\sqrt{\frac 1 {5   -   \alpha}}\sqrt{A_3^2+B_3^2}  \frac{1}{c^{(1)}}
\end{equation}
we obtain the cosine term claimed in \eqref{phinx-a}.

The approximation for \eqref{phixifBm-b} is obtained similarly: plugging the estimates from 
\eqref{ppmtau} into \eqref{phi0ifbm2} yields
\begin{align*}
&
\frac{\Phi_0(-u\nu_n)}{X(-u\nu_n) }
\simeq
k_1   + k_2  
- k_5 \nu_n^3 \Big( \sigma_1 2u   -   b_0  2u^2+   2u^3\Big)
- k_6 \nu_n^4\Big( \sigma_2 2u - \sigma_1  2u^2 +    b_0    2u^3 - 2 u^4\Big) \\
&
\simeq
2\bigg(
A_1    \frac{B_2}{A_2}   -    B_1  
+  \Big(\frac{B_2}{A_2}\sigma_1     -   \sigma_2\Big)  u
+\Big(\sigma_1 -  \frac{B_2}{A_2}   b_0\Big)  u^2 
+ \Big(\frac{B_2}{A_2}      -    b_0  \Big)   u^3 
 + u^4 
\bigg) k_6 \nu_n^4 
\end{align*}
and 
\begin{align*}
\frac {\Phi_1(-u \nu_n )}{X(-u \nu_n)}   
\simeq \;  &
k_1   - k_2   - k_3 \nu_n 2u + k_4 \nu_n^2 \big(-b_0  2u +2u^2\big)  \simeq 
\\
& - (-1)^n
 \bigg(
 \frac {    s^{(2)} }{s^{(1)}} c^{(1)}        -  c^{(2)}     +  \frac {    s^{(2)} }{s^{(1)}}   u 
  + u^2 
\bigg)  \frac 1{ c^{(1)}}    \sqrt{A_3^2+ B_3^2} k_6 \nu_n^4.
\end{align*}
 
Hence the integral term   \eqref{phixifBm-b} contributes 
\begin{align*}
&
\frac 1 {\nu_n}\frac {c_\alpha}{\lambda\Gamma(\alpha)} \frac 1{\pi} 
\int_{0}^\infty  \frac{\sin \theta_0(u)}{  \gamma_0(u)} \left( e^{-u\nu_n(1-x)}\Phi_1(-u\nu_n)+e^{-u\nu_n x}  \Phi_0(-u\nu_n) \right) du \simeq \\
&
-\frac 1 {\nu_n^2}\frac {c_\alpha}{\lambda\Gamma(\alpha)} \frac 1{\pi} 
\int_{0}^\infty  \frac{\sin \theta_0(u)}{  \gamma_0(u)} X_0(-u)\frac 1{u}
\left(  \frac{\Phi_0(-u\nu_n)}{X(-u\nu_n)} e^{-u\nu_n x}  + \frac{\Phi_1(-u\nu_n)}{X(-u\nu_n)}e^{-u\nu_n(1-x)}\right) du
\end{align*}
which after normalizing by factor \eqref{Cn} becomes \eqref{phinx-b} with 
\begin{equation}\label{rhorho}
\rho_0(u):= \frac{\sin \theta_0(u)}{  \gamma_0(u)} \frac 1{u} X_0(-u)
\end{equation}
and 
\begin{align}
\nonumber
Q_0(u) = &
\frac {2 c^{(1)}} {
    \sqrt{A_3^2+B_3^2}   
}
\Big(
A_1    \frac{B_2}{A_2}   -    B_1  
+  \Big(\frac{B_2}{A_2}\sigma_1     -   \sigma_2\Big)  u
+\Big(\sigma_1 -  \frac{B_2}{A_2}   b_0\Big)  u^2 
+ \Big(\frac{B_2}{A_2}      -    b_0  \Big)   u^3 
 + u^4 
\Big)
\\
\label{Q0Q1}
Q_1(u) = &
 \frac {    s^{(2)} }{s^{(1)}} c^{(1)}       -  c^{(2)}   +  \frac {    s^{(2)} }{s^{(1)}}   u   + u^2  
\end{align}
The last term \eqref{phinx-c} is deduced from \eqref{phixifBm-c} along the same lines. In principle, closed form formulas 
can be obtained for the amplitudes $C_0$ and $C_1$ and the phases $\varkappa_0$ and $\varkappa_1$; the emerging expressions are cumbersome
and will be omited.

Finally, by \eqref{k1to6} and  \eqref{A4}  
$$
\varphi_n(1) =   \psi''_n(1)=-k_4 \frac {2 c_\alpha}{\lambda_n\Gamma(\alpha)}\simeq
  (-1)^n\frac 1{2c^{(1)}}  \sqrt{A_3^2+B_3^2}   \frac {2 c_\alpha}{\lambda_n\Gamma(\alpha)} k_6 \nu_n^2
$$
and after normalizing by  factor \eqref{Cn} we get \eqref{phinend}. 
Similarly,   
$$
\int_0^1 \varphi_n(x)dx = -\psi'(0) = k_5 \frac {2 c_\alpha}{\lambda_n\Gamma(\alpha)} \simeq
-  \frac{B_2}{A_2}    \frac {2 c_\alpha}{\lambda_n\Gamma(\alpha)} k_6\nu_n 
$$
and after normalising by \eqref{Cn} we get \eqref{phinave} with the constant 
\begin{equation}\label{tildeC}
C :=  -   \frac{B_2}{A_2} \frac {2   c^{(1)}}{\sqrt{A_3^2+B_3^2}}.
\end{equation}

\end{proof}

The seemingly different expressions for $\Delta_\alpha$, obtained for $\alpha \in (0,1)$ 
and $\alpha\in (1,2)$, in fact, coincide:

\begin{lem}
The expressions for $\Delta_\alpha$ in \eqref{Delta_agr1} and \eqref{Deltaalpha} are equal  
for all $\alpha\in (0,2)\setminus \{1\}$.
\end{lem}

\begin{proof}
The claimed equality follows if we show that
\begin{multline}\label{eqeq}
  \Big(\frac 1 3 d_0^3-d_2\Big)\Big(A_2(\sigma_2-b_0)-B_2(\sigma_1-1)\Big)=\\
\Big(\frac 1 4+\frac 1 2 d_0^2+d_2 d_0-\frac1 {12} d_0^4\Big)
\Big(
A_1 B_2-A_2 B_1-A_2(\sigma_1-1)+B_2 b_0
\Big).
\end{multline}
Define the quantities:
$$
c_k := \cos k \frac \pi 2\frac {1-\alpha}{5-\alpha} \quad\text{and}\quad
 s_k := \sin  k\frac \pi 2\frac {1-\alpha}{5-\alpha}  , \quad k =1,2,3,4
$$
and 
$$
C_k := \cos k \frac \pi 2\frac {2}{5-\alpha} \quad\text{and}\quad
S_k := \sin  k\frac \pi 2\frac {2}{5-\alpha}  , \quad k =1,2,3,4.
$$
Then for $k=0,1,2$, 
\begin{align*}
b_k =  & \frac{1}{k+1}
\frac
{
\sin (k+1)\frac \pi 2 \left(\frac {1-\alpha}{5-\alpha} -\frac {2}{5-\alpha}\right)
}
{
\sin (k+1)\frac \pi 2 \frac {2} {5-\alpha} 
} 
=  
\frac{1}{k+1}\frac 1{S_{k+1}}\Big(
s_{k+1}
C_{k+1} 
-S_{k+1} 
c_{k+1}
\Big)
\end{align*}
and
$$
d_0 = 
\frac{
s_1
C_1
+
c_1
S_1
}
{
S_1
} \quad \text{and}\quad 
d_2 = \frac 1 3 
\frac
{
s_3
C_3
+
c_3
S_3
}
{
S_3
}.
$$
After the multiplication of \eqref{eqeq} by $(S_1S_2S_3)^6$, the left and right hand sides,
which we now denote by $E_\ell$ and $E_r$ turn into polynomials with respect to $c_k$'s, $s_k$'s, $C_k$'s and $S_k$'s. 
By the basic trigonometry $c_k$'s and $s_k$'s can be expressed in terms of $c_1$ and $s_1$:
$$
\begin{aligned}
&
s_2 = 2 s_1c_1   \\
&
s_3 = -4s_1^3+3s_1   \\
&
c_4 = 2c_1c_3-c_2 
\end{aligned}
\quad 
\begin{aligned}
  & c_2 = 2c_1^2-1 \\
  & c_3 = 4 c_1^3-3 c_1 \\
  & 
s_4 = 2c_1s_3-s_2
\end{aligned}
$$
The same is true for $C_k$'s and $S_k$'s. Plugging these identities, computer aided symbolic computation produces the following formula:
\begin{align}
&\label{brbr}
E_\ell -E_r =- 2\, \cos x  (\cos y )^4{\Big( (\cos x)^2 - 1\Big)}^2\, 
 \Big((\cos y)^2 - 1\Big)^6\,  \Big(4\, (\cos y)^2 - 1\Big)^5\, \times\\
&\nonumber
\Big( 
\sin(2y) 
+  \sin (2  x + 2  y )  
- 2\sin (2  x + 4  y ) 
+  \sin (2  x + 6  y ) 
+  \sin (4  x + 6  y )  
-  \sin (4  x + 8  y )
\Big)
\end{align}
where 
$$
x = \frac\pi 2\frac{1-\alpha}{5-\alpha}\quad\text{and}\quad 
y =  \frac \pi 2\frac 2 {5-\alpha}.
$$
Note that 
\begin{align*}
& \sin\!\left(2\, y\right) = \sin\!\left(    \pi  \frac 2 {5-\alpha}\right) \\
&
\sin\!\left(2\, x + 2\, y\right)=
\sin\!\left( \pi  \frac{3-\alpha}{5-\alpha} \right)\\
&
\sin\!\left(2\, x + 4\, y\right)=
\sin\!\left(\pi\frac{5-\alpha}{5-\alpha}  \right)=0\\
&
\sin\!\left(2\, x + 6\, y\right) = 
\sin\!\left( \pi  \frac{7-\alpha}{5-\alpha} \right)=
-\sin\!\left( \pi  \frac{2}{5-\alpha}\right)
  \\
&
\sin\!\left(4\, x + 6\, y\right)=
\sin\!\left(2 \pi \frac{1-\alpha}{5-\alpha} + 3  \pi  \frac 2 {5-\alpha}\right)=
\sin\!\left(2 \pi \frac{4-\alpha}{5-\alpha}\right) \\
&
\sin\!\left(4\, x + 8\, y\right) = 
\sin\!\left(2\pi  \frac{1-\alpha}{5-\alpha} + 2  \pi  \frac 4 {5-\alpha}\right)=
\sin\!\left(2\pi  \frac{5-\alpha}{5-\alpha} \right)=0.
\end{align*}
and hence the expression in the last brackets in \eqref{brbr} becomes
\begin{align*}
&
 \sin\!\left(    \pi  \frac 2 {5-\alpha}\right) 
+ \sin\!\left( \pi  \frac{3-\alpha}{5-\alpha} \right) 
-\sin\!\left( \pi  \frac{2}{5-\alpha}\right)
+ \sin\!\left(2 \pi \frac{4-\alpha}{5-\alpha}\right) =\\
&
2 \sin\!\left( \frac \pi 2  \frac{3-\alpha}{5-\alpha} +   \pi \frac{4-\alpha}{5-\alpha}\right) 
  \cos\!\left(\frac \pi 2  \frac{3-\alpha}{5-\alpha} -   \pi \frac{4-\alpha}{5-\alpha}\right) =\\
&
2 \sin\!\left( \frac \pi 2  \frac{11-3\alpha}{5-\alpha}  \right) 
  \cos\!\left(\frac \pi 2  \frac{5-\alpha }{5-\alpha}  \right)  =0.
\end{align*}

\end{proof}

\subsubsection{Enumeration alignment}
Similarly to the case $H<\frac 1 2$, the alignment between the enumeration, introduced in Lemma \ref{lem6.15}, 
and the enumeration which puts the eigenvalues in decreasing order, is achieved by shifting expression 
\eqref{nunifBm} for $\nu_n$ by $-2\pi$.


\begin{thebibliography}{10}

\bibitem{Br03a}
Jared~C. Bronski.
\newblock Asymptotics of {K}arhunen-{L}oeve eigenvalues and tight constants for
  probability distributions of passive scalar transport.
\newblock {\em Comm. Math. Phys.}, 238(3):563--582, 2003.

\bibitem{Br03b}
Jared~C. Bronski.
\newblock Small ball constants and tight eigenvalue asymptotics for fractional
  {B}rownian motions.
\newblock {\em J. Theoret. Probab.}, 16(1):87--100, 2003.

\bibitem{CKM03}
Patrick Cheridito, Hideyuki Kawaguchi, and Makoto Maejima.
\newblock Fractional {O}rnstein-{U}hlenbeck processes.
\newblock {\em Electron. J. Probab.}, 8:no. 3, 14, 2003.

\bibitem{ChK}
P.~Chigansky and M.~Kleptsyna.
\newblock Exact asymptotics in eigenproblems for fractional brownian covariance
  operators.
\newblock to appear in Stochastic Process. Appl., arXiv preprint 1601.05715,
  2017.

\bibitem{ChKM2}
P.~Chigansky, M.~Kleptsyna, and D.~Marushkevych.
\newblock On the eigenproblem for {G}aussian bridges.
\newblock arXiv preprint 1706.09298, 2017.

\bibitem{Gahov}
F.~D. Gakhov.
\newblock {\em Boundary value problems}.
\newblock Dover Publications, Inc., New York, 1990.
\newblock Translated from the Russian, Reprint of the 1966 translation.

\bibitem{GHT03}
F.~Gao, J.~Hannig, and F.~Torcaso.
\newblock Integrated {B}rownian motions and exact {$L_2$}-small balls.
\newblock {\em Ann. Probab.}, 31(3):1320--1337, 2003.

\bibitem{LS2}
Robert~S. Liptser and Albert~N. Shiryaev.
\newblock {\em Statistics of random processes. {II}}, volume~6 of {\em
  Applications of Mathematics (New York)}.
\newblock Springer-Verlag, Berlin, expanded edition, 2001.
\newblock Applications, Translated from the 1974 Russian original by A. B.
  Aries, Stochastic Modelling and Applied Probability.

\bibitem{LP04}
Harald Luschgy and Gilles Pag{\`e}s.
\newblock Sharp asymptotics of the functional quantization problem for
  {G}aussian processes.
\newblock {\em Ann. Probab.}, 32(2):1574--1599, 2004.

\bibitem{M82}
B.~B. Mandelbrot.
\newblock On an eigenfunction expansion and on fractional {B}rownian motions.
\newblock {\em Lett. Nuovo Cimento (2)}, 33(17):549--550, 1982.

\bibitem{MvN68}
Benoit~B. Mandelbrot and John~W. Van~Ness.
\newblock Fractional {B}rownian motions, fractional noises and applications.
\newblock {\em SIAM Rev.}, 10:422--437, 1968.

\bibitem{M46}
N.~I. Muskhelishvili.
\newblock {\em Singular integral equations}.
\newblock Dover Publications, Inc., New York, 1992.
\newblock Boundary problems of function theory and their application to
  mathematical physics, Translated from the second (1946) Russian edition and
  with a preface by J. R. M. Radok, Corrected reprint of the 1953 English
  translation.

\bibitem{N09b}
A.~I. Nazarov.
\newblock On a family of transformations of {G}aussian random functions.
\newblock {\em Teor. Veroyatn. Primen.}, 54(2):209--225, 2009.

\bibitem{NN04ptrf}
A.~I. Nazarov and Ya.~Yu. Nikitin.
\newblock Exact {$L_2$}-small ball behavior of integrated {G}aussian processes
  and spectral asymptotics of boundary value problems.
\newblock {\em Probab. Theory Related Fields}, 129(4):469--494, 2004.

\bibitem{NN04tpa}
A.~I. Nazarov and Ya.~Yu. Nikitin.
\newblock Logarithmic asymptotics of small deviations in the {$L_2$}-norm for
  some fractional {G}aussian processes.
\newblock {\em Teor. Veroyatn. Primen.}, 49(4):695--711, 2004.

\bibitem{N09}
Alexander~I. Nazarov.
\newblock Exact {$L_2$}-small ball asymptotics of {G}aussian processes and the
  spectrum of boundary-value problems.
\newblock {\em J. Theoret. Probab.}, 22(3):640--665, 2009.

\bibitem{PT17}
V.~Pipiras and M.S. Taqqu.
\newblock {\em Long-Range Dependence and Self-Similarity}.
\newblock Cambridge Series in Statistical and Probabilistic Mathematics.
  Cambridge University Press, 2017.

\bibitem{R85}
H.-J. Reinhardt.
\newblock {\em Analysis of approximation methods for differential and integral
  equations}, volume~57 of {\em Applied Mathematical Sciences}.
\newblock Springer-Verlag, New York, 1985.

\bibitem{Tsybakov}
Alexandre~B. Tsybakov.
\newblock {\em Introduction to nonparametric estimation}.
\newblock Springer Series in Statistics. Springer, New York, 2009.
\newblock Revised and extended from the 2004 French original, Translated by
  Vladimir Zaiats.

\bibitem{Ukai}
Seiji Ukai.
\newblock Asymptotic distribution of eigenvalues of the kernel in the
  {K}irkwood-{R}iseman integral equation.
\newblock {\em J. Mathematical Phys.}, 12:83--92, 1971.

\end{thebibliography}

\def\cprime{$'$} \def\cprime{$'$} \def\cydot{\leavevmode\raise.4ex\hbox{.}}
  \def\cprime{$'$} \def\cprime{$'$} \def\cprime{$'$}

\end{document}